\documentclass[reqno,10pt,oneside]{amsart}

\usepackage{amsthm,amsmath,amssymb,amscd}
\usepackage{mathrsfs}
\usepackage{lscape}
\usepackage{graphics} 
\usepackage[all,cmtip]{xy}
\usepackage{graphicx}
\usepackage{subfigure}
\usepackage{url}	
\usepackage{hyperref}
\usepackage{color}
\usepackage[usenames,dvipsnames]{xcolor}
\usepackage{verbatim}
\usepackage{fancyhdr}
\usepackage{geometry}
\usepackage{cancel}
\usepackage[normalem]{ulem}
\usepackage{mathtools}
\usepackage[english]{babel}
\mathtoolsset{showonlyrefs,showmanualtags}
\numberwithin{equation}{section}

\newtheorem{teo}{Theorem}[section]
 
\newtheorem{prop}[teo]{Proposition}

\newtheorem{lemma}[teo]{Lemma}

\newtheorem{example}[teo]{Example}
\theoremstyle{definition} 
\newtheorem{definition}[teo]{Definition}
\theoremstyle{remark}
\newtheorem{remark}[teo]{Remark}
\newtheorem*{cav}{\textbf{CAVEAT}}
\newtheorem*{notation}{\textbf{Notation}}
\newcommand{\del}{\partial}
\newcommand{\QQ}{\mathbb{Q}}

\newcommand{\tr}{\text{tr}}

\newcommand{\st}{\left(}
\newcommand{\dt}{\right)}
\newcommand{\sq}{\left[}
\newcommand{\dq}{\right]}
\newcommand{\sg}{\left\{}
\newcommand{\dg}{\right\}}
\newcommand{\p}{  {\bf p} }
\newcommand{\q}{ {\bf q} }
\newcommand{\hg}{ {\bf h} }
\newcommand{\kg}{  {\bf k} }
\newcommand{\ag}{ {\bf a}}
\newcommand{\bg}{  {\bf b}}
\newcommand{\cg}{{\bf c}}

\newcommand{\aaa}{{\bf a}}
\newcommand{\bbb}{{\bf b}}
\newcommand{\ccc}{{\bf c}}

\newcommand{\GGG}{ \mathbf G}
\newcommand{\etag}{\boldsymbol{\eta}}
\newcommand{\dd}{\partial\overline{\partial}}
\newcommand{\NN}{\mathbb{N}}

\newcommand{\s}{\textbf{S}}

\newcommand{\vol}{\text{Vol}_{\omega}}

\newcommand{\RR}{\mathbb{R}}
\newcommand{\ZZ}{\mathbb{Z}}
\newcommand{\CC}{\mathbb{C}}
\newcommand{\PP}{\mathbb{P}}

\newcommand{\Sp}{\mathbb{S}}

\newcommand{\ambbda}{\mathcal{B}_{\alpha}}
\newcommand{\dombd}{\mathcal{B}\st\kappa, \delta \dt}
\newcommand{\dombdmod}{\mathfrak{B}\st\kappa, \beta',\sigma' \dt}
\newcommand{\K}{K\"{a}hler}

\newcommand{\aaaa}{\boldsymbol{\alpha}}
\newcommand{\bbbb}{\boldsymbol{\beta}}
\newcommand{\cccc}{\boldsymbol{\gamma}}

\newcommand{\hkoid}{H_{h_{j}^{(\dagger)},k_{j}^{(\dagger)}}^{out}}
\newcommand{\hko}{{\bf H}_{\hg,\kg}^{out}}

\newcommand{\Lg}{\mathbb{L}_{\omega}}

\newcommand{\Cc}{C_{\delta-4}^{0,\alpha}\st M_{\p} \dt}

\newcommand{\Cqdd}{C_{\delta}^{4,\alpha}\st M_{\p} \dt \oplus \mathcal{D}_{\p}\st \bg,\cg \dt}
\newcommand{\Cqddd}{C_{\delta}^{4,\alpha}\st M_{\p} \dt \oplus \mathcal{E}_{\p} \oplus \mathcal{D}_{\p}\st \bg,\cg \dt}

\newcommand{\Pbe}{{\bf P}_{\bg,\etag}}

\newcommand{\rep}{r_\varepsilon}
\newcommand{\xroi}{\chi_{j,r_0}}

\newcommand{\Rep}{R_\varepsilon}
\newcommand{\hkjj}{H_{\tilde{h},\tilde{k}}^{in}}

\newcommand{\hkii}{{\bf H}_{\tilde{h},\tilde{k}}^{in}}
\newcommand{\Le}{\mathbb{L}_{\eta}}

\newcommand{\Cdx}{C_{\delta}^{4,\alpha}\st X_{\Gamma} \dt }
\newcommand{\Ccx}{C_{\delta-4}^{0,\alpha}\st X_{\Gamma} \dt }

\newcommand{\etat}{\eta_{\tilde{b},\tilde{h},\tilde{k}}}

\newcommand{\csfii}{f_{\tilde{b},\tilde{h},\tilde{k}}^{in}}

\newcommand{\cga}{c(\Gamma)}

\newcommand{\cgaj}{c\st\Gamma_{j}\dt}

\newcommand{\Pbg}{{\bf P}_{\tilde{b},\omega}}

\begin{document}

\title[On the Kummer construction for Kcsc metrics]{On the Kummer construction for Kcsc metrics}
\author[Claudio Arezzo] {Claudio Arezzo}
\address{ICTP Trieste and Univ. of Parma, arezzo@ictp.it }
\author{Riccardo Lena}
\address{Univesrit\`a di Parma, riccardo.lena@gmail.com}
\author{Lorenzo Mazzieri}
\address{Scuola Normale Superiore, Pisa, l.mazzieri@sns.it}

\begin{abstract}

Given a compact constant scalar curvature K\"ahler orbifold, with nontrivial holomorphic vector fields, whose singularities admit a local ALE \K\ Ricci-flat resolution, we find sufficient conditions on the position of the singular points to ensure the existence of a global constant scalar curvature \K\ desingularization. This generalizes the results previously obtained by the first author with F. Pacard. A series of explicit examples is discussed.\\

\end{abstract}

\maketitle

\vspace{-,15in}

{\it{1991 Math. Subject Classification:}} 58E11, 32C17.

\section{Introduction}

\noindent The aim of this paper is to extend the celebrated Kummer's construction for Calabi-Yau manifolds
(\cite{Page}, \cite{Top}, \cite{lsin}, and \cite{j} for a number of generalisations) to construct new families of K\"ahler constant scalar curvature (Kcsc from now on) metrics on 
compact complex manifolds and orbifolds. 


In order to state our results precisely, let us briefly recall that one starts with  a Kcsc base $M$ with {\em{isolated quotient singularities}}, hence locally of the form $\CC^m/\Gamma_j$, where $m$ is the complex dimension of $M$, $j\in J$ parametrizes the set of points we want to desingularize, and $\Gamma_j$ is a finite subgroup of $U(m)$ acting freely away from the origin.

\noindent Given such a singular object one would like to replace a small neighborhood of a singular point and replace it with a large piece of a K\"ahler resolution $\pi\colon (X_{\Gamma}, \eta) \rightarrow \CC^m/\Gamma$ keeping the scalar curvature constant (and close to the starting one). For such a construction to even have a chance to preserve the Kcsc equation it is necessary that $(X_{\Gamma}, \eta)$ is scalar flat, i.e. it is necessary to assume that {\em{$\CC^m/\Gamma_j$ has a scalar flat ALE resolution.}}

Having then fixed a set of singular points $\{p_1, \dots, p_n\} \subset M$ each corresponding to a group $\Gamma_j$, and denoted by
$ B_{j,r}  : =  \{  z \in {\mathbb C}^{m} / \, \Gamma_j
\, : \,  |z| < r \}, $
we can define, for
all $r >0$ small enough (say $r \in (0, r_0)$)
\begin{equation}
M_r : = M \setminus \cup_j \, B_{j,r}  . \label{eq:2.3}
\end{equation}

\noindent On the other side, for each $j = 1, \ldots, n$, we are given
a $m$-dimensional \K\ manifold $(X_{\Gamma_j}, \eta_j)$, with one end biholomorphic to a
neighborhood of infinity in ${\mathbb C}^{m} / \, \Gamma_j$. 
\noindent Dual to the previous notations on the base manifold, we set 
$C_{j , R}  : =  \{ x \in {\mathbb C}^{n} / \, \Gamma_j \,
: \, |x | >  R \}, $
the complement of a closed large ball and the complement of an open
large ball in $X_{\Gamma_j}$ (in the  coordinates which parameterize a
neighborhood of infinity in $X_{\Gamma_j}$). We define, for all $R > 0$ large
enough (say $R > R_0$)

\begin{equation}
X_{\Gamma_j, R} : = X_{\Gamma_j} \setminus C_{j , R}. \label{eq:2.5}
\end{equation}
which corresponds to the manifold $X_{\Gamma_j}$ whose end has been
truncated. The boundary of $X_{\Gamma_j,R}$ is denoted by $\del C_{j,R}$.

\medskip

\noindent We are now in a position to describe the generalized connected sum
construction. Indeed, for all
$\varepsilon \in (0, r_0/R_0)$, we choose $r_{\varepsilon} \in (\varepsilon \, R_0 , r_0)$ and
define
\begin{equation}
R_{\varepsilon} : =  \frac{r_{\varepsilon}}{\varepsilon} . \label{eq:2.6}
\end{equation}
By construction
\[
{\tilde{M}} : = M \sqcup _{{p_{1}, \varepsilon}} X_{\Gamma_1} \sqcup_{{p_{2},\varepsilon}} \dots
\sqcup _{{p_n, \varepsilon}} X_{\Gamma_n} ,
\]
is obtained by connecting  $M_{r_\varepsilon}$ with the truncated ALE spaces
$X_{\Gamma_{1, R_\varepsilon}}, \ldots, X_{\Gamma_{n , R_\varepsilon}}$. The identification of the
boundary $\del B_{j , r_{\varepsilon}}$ in $M_{r_{\varepsilon}}$ with the boundary $\del
C_{j , R_\varepsilon}$ of $X_{\Gamma_j, R_\varepsilon}$ is performed using the change of
variables
\[
(z^{1} , \ldots, z^{m} )  = \varepsilon \, (x^{1} , \ldots, x^{m}) ,
\]
where $(z^{1}, \ldots, z^{m} )$ are the coordinates in $B_{j , r_0}$
and $(x^{1}, \ldots, x^{m})$ are the coordinates in $C_{j , R_0}$.

\smallskip

It was proved in \cite{ap1} that if no nontrivial holomorphic vector fields exist on $\st M, \omega, g\dt$ the ALE scalar flat condition on the model is also sufficient to construct a family parametrized by the gluing parameter $\varepsilon$ on the manifold (or orbifold) obtained by this procedure. On the other hand, the known picture for the blow up of smooth points, suggests that the number and position of points should be relevant to achieve the same existence theorem in presence of continuous symmetries. In fact, being the linearized scalar curvature operator $\Lg$ given by
\begin{equation}
{\mathbb L}_\omega f \,  =  \, \Delta^{2}_\omega f \,   +  \, 4 \, \langle \, \rho_\omega \, | \,   i\dd f \, \rangle  \, , 
\end{equation}

\noindent we have to look at the positions of points relative to the elements of $ker(\Lg) \, = \,span_{\RR} \, \{\varphi_0, \varphi_1, \ldots , \varphi_d \} ,$
where $\varphi_0 \equiv 1$, $d$ is a positive integer and $\varphi_1, \ldots, \varphi_d$ is a collection of linearly independent functions in $\ker(\Lg)$ with zero mean and normalized in such a way that $||\varphi_i||_{L^2(M)} = 1$, $i=1, \ldots, d$,.

Interestingly, the analysis required to achieve the final goal strongly depends on some further structure of the ``local model" $X_{{\Gamma_j}}$ and in particular on whether the metric $\eta_j$ is {\em Ricci-flat} or merely scalar flat.

As it turns out, the hardest case is when the resolution is Ricci flat (which in particular forces the group $\Gamma_j$ to be in $SU(m)$) since these metrics do not present the leading non-euclidean term in the expansion of their potential, and this is the case we 
treat in this paper. The following is our main result which gives the new conditions on the ``symplectic" positions of the singular points for the Kcsc equation to be solvable:

\begin{teo}\label{maintheorem}
Let $\st M,\omega, g\dt$ be a compact $m$-dimensional Kcsc orbifold  with isolated singularities and constant scalar curvature equal to $s_{\omega}$. Let $\p\:=\sg p_1,\ldots,p_{N}\dg\subseteq M$ the set of points  with neighborhoods biholomorphic to  a ball of  $\CC^m/\Gamma_{j}$ where, for $j=1,\ldots, N$, the $\Gamma_{j}$'s are nontrivial subgroups of $SU(m)$ of order $|\Gamma_{j}|$ and such that $\CC^{m}/\Gamma_{j}$ admits an ALE  Kahler Ricci-flat resolution $\st X_{\Gamma_{j}},h_{j},\eta_{j} \dt$. Let 
\begin{align}
\ker\st   \Lg  \dt =&	\,span_{\RR	}\left\{1,\varphi_1,\ldots,\varphi_d \right\}\,.
\end{align}
be the space of Hamiltonian potentials of Killing fields with zeros. 
Suppose moreover that there exist $ \bg \in (\mathbb{R}^{+})^{N}$ and $\cg\in\RR^{N}$ such that
\begin{displaymath}\label{eq:matricebal}
		\left\{\begin{array}{lcl}
		\sum_{j=1}^{N}b_{j}\Delta_{\omega}\varphi_{i}\st p_{j} \dt+c_{j}\varphi_{i}\st p_{j} \dt=0 && i=1,\ldots, d\\
		&&\\
		\st b_{j}\Delta_{\omega}\varphi_{i}\st p_{j} \dt+c_{j}\varphi_{i}\st p_{j} \dt \dt_{\substack{1\leq i\leq d\\1\leq j\leq N}}&& \textrm{has full rank.}
		\end{array}\right.
\end{displaymath}
If in addition the condition
\begin{equation}
\label{eq:tuning}
c_{j}=s_{\omega}b_{j}\,
\end{equation}
is satisfied, then there exists $\bar{\varepsilon}$ such that for every $\varepsilon \in (0,\bar{\varepsilon})$ the orbifold
\[
\tilde{M} : = M \sqcup _{{p_{1}, \varepsilon}} X_{\Gamma_1} \sqcup_{{p_{2},\varepsilon}} \dots
\sqcup _{{p_N, \varepsilon}} X_{\Gamma_N}
\]
has a Kcsc metric in the class 
\begin{equation}
\pi^{*}\sq\omega \dq+ \sum_{j=1}^{N}\varepsilon^{2m}\tilde{b}_{j}^{2m}\sq \tilde{\eta}_{j} \dq\qquad\textrm{ with }\qquad \mathfrak{i}_{j}^{*}\sq \tilde{\eta}_{j} \dq=[\eta_{j}]
\end{equation}
\noindent where $\pi$ is the canonical surjection of $\tilde{M}$ onto $M$ and $\mathfrak{i}_{j}$ the natural embedding of $X_{\Gamma_{j},\Rep}$ into $\tilde{M}$. Moreover 
\begin{equation}
\left|\tilde{b}_{j}^{2m} - \frac{|\Gamma_{j}|b_{j}}{2\st m-1 \dt}\right| \leq \mathsf{C} \varepsilon^{\gamma}\qquad\textrm{ for some }\qquad \gamma>0\,,
\end{equation}
where $|\Gamma_{j}|$ denotes the order of the group.
\end{teo}

Whether, given $\Gamma$ in $SU(m)$, a Ricci flat K\"ahler resolution exists
is by itself an important problem in different areas of mathematics and we will not digress on it here. It suffices to recall the reader that 
Ricci flat  models do exist for any subgroup of $SU(m)$ with $m=2$, thanks to the work of Kronheimer, while in higher dimensions one needs to assume the existence of a K\"ahler crepant resolution and then apply deep results by Joyce \cite{j}, Goto ({\cite{Goto}), Van Coevering (\cite{VanCoevering}) and Conlon-Hein \cite{ConlonHein}. 
In particular $m=3$ works fine again for any $\Gamma$ in $SU(3)$.

The role of the equation (\ref{eq:tuning}) is particularly interesting. We will show in Section $5$ that without this assumption it is possible to construct Kcsc metrics on the manifolds with boundaries obtained by removing small neighbourhoods of the singularities on the base and large pieces of the ends of the models. We believe such a result should be of independent interest and it justifies the choice of using a Cauchy-data matching technique instead of the more common pre-gluing type argument.
Equation (\ref{eq:tuning}) is on the other hand crucial in order to prove, as we do in Section $6$, that there exists at least one truncated metric on the base which matches exactly one truncated metric on the model. It is also worth observing that without equation  
(\ref{eq:tuning}) we would have a space of solutions of dimension $2N-d$ gluing Ricci-flat models, opposed to the $N-d$ dimensional space of solutions of the corresponding problem when scalar flat, non Ricci-flat, models are glued as in the case of blow ups. Equation (\ref{eq:tuning}) reduces the number of parameters exactly to the same size as the previously known cases.

Theorem \ref{maintheorem} deserves few comments: first of all it would of great interest to interpret these new balancing conditions in terms of the algebraic data of the orbifold, at least when starting with a polarized object, very much in the spirit of Stoppa's interpretation of the blow-up picture (\cite{Stoppa}).

Our results can also be seen as ``singular perturbation" results applied to the original singular space {\em{fixing}} the complex structure and deforming the K\"ahler class. A very different, though parallel in spirit, analysis can be done by thinking of keeping the K\"ahler class fixed and {\em{moving the complex structure}}. Unfortunately nobody has been able to prove gluing theorems for integrable complex structures so far, but {\em{assuming that such a deformation exists}}, this dual analysis, with no holomorphic vector fields and in complex dimension two, has been done in the important work on Spotti (\cite{Spotti}) in the Einstein  and special ordinary double point case, and by Biquard-Rollin (\cite{BiquardRollin}) in the Kcsc case for general $\QQ$-Gorenstein singularities. 

Many of the technical difficulties encountered in proving Theorem \ref{maintheorem} could be avoided if one seeks extremal metrics instead of Kcsc ones.
This fact, already observed by Tipler for surfaces with cyclic quotient singularities in \cite{Tip}, is now rigorously proved in \cite{ALM}. Nevertheless going back from extremal
to Kcsc would require knowing the behaviour of Futaki's invariant under resolution of singularities, which at the moment seems out of reach. The analogue approach for blowing up smooth points has been carried out by Stoppa (\cite{Stoppa}), Della Vedova - Zuddas (\cite{DVZ}), and 
 G. Szekelyhidi (\cite{Gabor1}, \cite{Gabor2}). 
\vspace{2mm}

Turning back to our results, we can then look for new examples of full or partial desingularizations of Kcsc orbifolds. 
Of course it will be very hard on a general orbifold to compute $\Delta_{\omega}\varphi_j$. On the other hand, assuming for example that $M$ is Einstein and using \begin{equation}
\Delta_{\omega}\varphi_{j}=-\frac{s_{\omega}}{m}\varphi_{j}\, ,
\end{equation}
the balancing condition requires only the knowledge of the value of the $\varphi_j$ at the singular points. Moreover these values are easily computed for example in toric setting by the well known relationship between the evaluation of the potentials $\varphi_j$ and the image point via the moment map. 
 
With these classical observations one can then look for toric K\"ahler-Einstein orbifolds with isolated quotient singularities to test to which of them our results can be applied. In complex dimension $2$ things are pretty simple and in fact two such examples are 
\begin{itemize}
\item
$\st\PP^{1}\times\PP^{1},\pi_{1}^{*}\omega_{FS}+\pi_{2}^{*}\omega_{FS}\dt$ with $\ZZ_{2}$ acting by 
\begin{equation}\st[x_{0}:x_{1}],[y_{0}:y_{1}]\dt\longrightarrow  \st[x_{0}:-x_{1}],[y_{0}:-y_{1}]\dt\end{equation} 

This orbifold is isomorphic to the intersection of two singular quadrics in $\PP^4$.
\begin{equation}
\sg z_{0}z_{3}-z_{4}^{2}=0 \dg\cap \sg z_{1}z_{2}-z_{4}^{2}=0 \dg
\end{equation}

\item

$\st\PP^{2},\omega_{FS}\dt$ with $\ZZ_{3}$ acting by
\begin{equation}[z_{0}:z_{1}:z_{2}]\longrightarrow  [x_{0}:\zeta_{3}x_{1}:\zeta_{3}^{2}x_{2}]\qquad \zeta_{3}\neq 1, \zeta_{3}^{3}=1\end{equation}

This orbifold is isomorphic to the singular cubic surface in $\PP^{3}$
\begin{equation}
\sg z_{0}z_{1}z_{2}-z_{3}^{3}=0\dg\,.
\end{equation}

\end{itemize}

In both cases we will show in Section $7$ that our results provide a {\em{full}} Kcsc (clearly {\em{not} \K -Einstein}) desingularization (in the first case applied to $4$ singular $SU(2)$ points, while $3$ $SU(2)$ points in the second). It is worth noting that both these orbifolds are also limits of smooth \K - Einstein surfaces. This can be seen in various ways: either applying Tian's resolution of the Calabi Conjecture (\cite{Tian}) or by  \cite{AGP} in the first case, and Odaka-Spotti-Sun above mentioned result to both.

Working out higher dimensional examples turned out to be much more challenging than we expected. Even making use of the beautiful 
database of Toric Fano Threefolds run by G. Brown and A. Kasprzyk (\cite{GRD}, see also \cite{kasprzyk}) and their amazing help in implementing 
a complete search of Einstein ones with isolated singularities, we could only extract orbifolds where only a partial Kcsc resolution is possible.
In fact they produced a complete list (see \cite{GRD2}) of toric Fano threefolds s.t.
\begin{itemize}
\item
 they have only isolated quotient singular points;
\item
their moment polytope has barycenter in the origin (this implies the Einstein condition, thanks to a well known result by Mabuchi \cite{Mabuchi});
\item
each singular point is a $\CC^3/\Gamma$, $\Gamma \in U(3)$.
\end{itemize}
For example, let  $X^{(1)}$ be the toric K\"ahler-Einstein threefold whose  1-dimensional fan  $\Sigma^{(1)}_{1}$ is generated by points
\begin{equation}\Sigma^{(1)}_{1}=\left\{(1,3,-1), (-1,0,-1), (-1,-3,1), (-1,0,0), (1,0,0), (0,0,1), (0,0,-1), (1,0,1)\right\}
\end{equation}
and its $3$-dimensional fan $\Sigma^{(1)}_{3}$  is generated by $12$ cones 
\begin{align}
C_{1}:=& \left<        (-1,  0, -1),(-1, -3,  1),(-1,  0,  0)\right>\\
C_{2}:=& \left<        ( 1,  3, -1),(-1,  0, -1),(-1,  0,  0)\right>\\
C_{3}:=& \left<        (-1, -3,  1),(-1,  0,  0),( 0,  0,  1)\right>\\
C_{4}:=&\left<        ( 1,  3, -1),(-1,  0,  0),( 0,  0,  1)\right>\\
C_{5}:= &\left<        ( 1,  3, -1),(-1,  0, -1),( 0,  0, -1)\right>\\
C_{6}:=&\left<        (-1,  0, -1),(-1, -3,  1),( 0,  0, -1)\right>\\
C_{7}:= & \left<        (-1, -3,  1),( 1,  0,  0),( 0,  0, -1)\right>\\
C_{8}:=& \left<        (1,  3, -1),(1,  0,  0),(0,  0, -1)\right>\\
C_{9}:= &\left<        (1,  3, -1),(0,  0,  1),(1,  0,  1)\right>\\
C_{10}:=& \left<        (-1, -3,  1),( 1,  0,  0),( 1,  0,  1)\right>\\
C_{11}:=& \left<        (1,  3, -1),(1,  0,  0),(1,  0,  1)\right>\\
C_{12}:= & \left<        (-1, -3,  1),( 0,  0,  1),( 1,  0,  1)\right>
\end{align}
All these cones are singular and $C_{1},C_{4},C_{5},C_{7},C_{11},C_{12}$ are cones relative to affine open subsets of $X^{(1)}$ containing a $SU(3)$ singularity, while the others  are cones relative to affine open subsets of $X^{(1)}$ containing a $U(3)$ (non Ricci flat) singularity. 
\noindent We will show in Section $7$ that these $6$ $SU(3)$ singularities do satisfy all the requirements of Theorem \ref{maintheorem}.

\smallskip

\noindent {\bf{Structure of the paper:}} in Section $2$ we  collect some known facts and we prove a crucial refinement (Proposition \ref{asintpsieta}) of results of Joyce, Tian-Yau and others on the asymptotics of a \K\ Ricci flat metric on a crepant resolution. 

In Section $3$ we collect, with complete proofs, all results needed at the linear level on the linearized scalar curvature operator on the base orbifold. In particular we construct global functions in the kernel of the linearized operator with prescribed blow up behaviour near the singularities
(see Proposition \ref{loc_structure}).

Section $4$ contains all the (weighted) linear analysis on a scalar flat \K\ resolution of an isolated singularity. These results are significantly different from what was known, in that our problem forces us to use weights in a different, more delicate, range.

We emphasise that Sections $3$ and $4$ describe the complete picture of the weighted linear analysis needed not just to prove our main result, and in fact it will be used by the authors in a forthcoming paper to prove a result similar to Theorem \ref{maintheorem} for general scalar flat resolutions. We believe these sections clarifies many similar analyses present in the literature.

In Section $5$ the existence of truncated Kcsc metrics on the base and on the models is proved in Propositions \ref{crucialbase}  and \ref{crucialmodello}. 

Section $6$ contains the proof of Theorem \ref{maintheorem} by proving the mentioned Cauchy-data matching property of the truncated metrics under the assumption 
(\ref{eq:tuning}).

\noindent Section $7$ gives a complete description of the above mentioned examples.
\vspace{3mm}

\noindent {\bf{Aknowledgments:}} We wish to thank Frank Pacard and Gabor Szekelyhidi for many discussions on this topic. We also wish to express our deep gratitude to Gavin Brown and Alexander Kasprzyk for their help in not drowning in the Fano toric threefolds world. 
The authors have been partially supported by the FIRB Project ``Geometria Differenziale Complessa e Dinamica Olomorfa".

\section{Notations and preliminaries}\label{preliminaries}

\subsection{Eigenfunctions and eigenvalues of $\Delta_{\Sp^{2m-1}}$.}

\label{eigen}

In order to fix some notation which will be used throughout the paper, we agree that $\Sp^{2m-1}$ is the unit sphere of real dimension $2m-1$, equipped with the standard round metric inherited from $(\mathbb{C}^m, g_{eucl})$. We will denote by $\{\phi_k\}_{k\in \NN}$ a complete orthonormal system of the Hilbert space $L^2(\Sp^{2m-1})$, given by eigeinfunctions of the Laplace-Beltrami operator $\Delta_{\Sp^{2m-1}}$, so that, for every $k \in \NN$,
$$
\Delta_{\Sp^{2m-1}} \phi_k \,  = \,  \lambda_k \phi_k
$$ 
and $\{\lambda_k \}_{k \in \mathbb{N}}$ are the eigenvalues of $\Delta_{\Sp^{2m-1}}$ {\em{counted with multiplicity}}. We will also indicate by $\Phi_j$ the generic element of the $j$-th eigenspace of $\Delta_{\Sp^{2m-1}}$, so that, for every $j \in \NN$, 
$$
\Delta_{\Sp^{2m-1}} \Phi_j \,  = \,  \Lambda_j \Phi_j \, 
$$ and $\{\Lambda_j \}_{j \in \mathbb{N}}$ are the eigenvalue of $\Sp^{2m-1}$ {\em{counted without multiplicity}}. In particular, we have that
$\Lambda_j =  - j(2m - 2 + j)$, for every $j \in \NN$. If $\Gamma \triangleleft U(m)$ is a finite subgroup of the unitary group acting on $\CC ^m$ having the origin as its only fixed point, we denote by $\{\Lambda^{\Gamma}_j\}_{j \in \NN}$ the eigenvalues {\em{counted without multiplicity}} of the operator $\Delta_{\Sp^{2m-1}}$ restricted to the $\Gamma$-invariant functions.
For future convenience we introduce the following notation, given $f\in L^{2}\st \Sp^{2m-1} \dt$ we denote  with $f^{(k)}$ the  $L^2\st \Sp^{2m-1} \dt$-projection of $f$ on the $\Lambda_{k}$-eigenspace of $\Delta_{\Sp^{2m-1}}$ and 
\begin{equation}
f^{(\dagger)}:=f-f^{(0)}
\end{equation}

\subsection{The scalar curvature equation}

We let $(M,g,\omega)$ be a \K\ orbifold with complex dimension equal to $m$, where $g$ is the \K\ metric and $\omega$ is the \K\ form. Notice that we allow the Riemannian orbifold $(M,g)$ to be incomplete, since in the following we will be eventually led to consider punctured orbifolds. We denote by $s_\omega$ the scalar curvature of the \K\ metric $g$ and by $\rho_\omega$ its Ricci form.
%
In the following it will be useful to consider cohomologous deformations of the \K\ form $\omega$. Hence, for a smooth real function $f \in C^\infty(M)$ such that $\omega + i  \dd  f  >  0 $, we set
\[
\omega_f \,   =  \, \omega + i  \dd  f \,  ,
\]
and we will refer to $f$ as the deformation potential.
Since we want to understand the behavior of the scalar curvature under deformations of this type, it is convenient to consider the following differential operator 
$$
\mathbf{S}_\omega(\cdot) \, : \, C^\infty(M) \longrightarrow C^\infty(M) \, , \qquad \qquad
 f \,\longmapsto \,\mathbf{S}_{\omega}(f) := s_{\omega+ i\dd f} \, ,
$$ 
which associate to a deformation potential $f$ the scalar curvature of the corresponding metric. Following the formal computations given in~\cite{ls}, we obtain the formal expansion 
\begin{equation}
\label{eq:espsg}
\mathbf{S}_{\omega}(f) \,\, = \,\,  s_{\omega} - \, \frac{1}{2} \, {\mathbb L}_{\omega}f \, + \, \frac{1}{2} \mathbb{N}_{\omega}(f) \, ,
\end{equation}
where the linearized scalar curvature operator $\Lg$ is given by
\begin{equation}
{\mathbb L}_\omega f \,  =  \, \Delta^{2}_\omega f \,   +  \, 4 \, \langle \, \rho_\omega \, | \,   i\dd f \, \rangle  \, . \label{eq:defLg}
\end{equation}

Once we introduce the bilinear operator $\circ$ acting on  tensors in  $\st TM^{*} \dt^{(1,0)}\otimes \st TM^{*} \dt^{(0,1)}$ as
\begin{equation}\label{eq: circ}
\st T\circ U\dt_{i\bar{l}}:= T_{i\bar{\jmath}}g^{k\bar{\jmath}}U_{k\bar{l}}\qquad  T,U \in \st TM^{*} \dt^{(1,0)}\otimes \st TM^{*} \dt^{(0,1)}\,,
\end{equation}
the nonlinear remainder $\NN_{\omega}$ takes the form
\begin{equation}
\label{eq:defQg}
\NN_{\omega}(f) \, = \, 8 \tr_{\omega}\,(i\dd f \circ i\dd f \circ \rho_\omega) \, - \,  8\tr _{\omega}\, (i\dd f \circ i\dd \, \Delta_{\omega}f ) \, + \, 4 \Delta_{\omega} \, \tr_{\omega} \, (i\dd f\circ i\dd f) \, + \,  2\RR_{\omega}(f)\, ,
\end{equation}
with $\RR_\omega(f)$  the collections of all higher order terms.

\subsection{The K\"ahler potential of a Kcsc orbifold}

We let $(M,g,\omega)$ be a compact constant scalar curvature \K\ orbifold  without boundary with complex dimension equal to $m$. Unless otherwise stated the singularities are assumed to be isolated.
Combining the local $\dd$-lemma with the equations of the previous subsection, we are now in the position to give a more precise description of the local structure of the \K\ potential of a Kcsc metric.

\begin{prop}
\label{proprietacsck}
Let $(M,g, \omega)$ be a \K\ orbifold. Then, given any point $p\in M$,
 there exists a holomorphic coordinate chart $(U, z^1, \ldots, z^m)$ centered at $p$ such that the \K\ form can be written as
$$
\omega \,\, = \,\, i \dd \, \bigg(\,\frac{|z|^2}{2} + \psi_\omega \bigg) \, , \qquad \hbox{with} \qquad \psi_\omega \, = \, \mathcal{O}(|z|^4) \, .
$$
If in addition the scalar curvature $s_g$ of the metric $g$ is constant, then $\psi_g$ is a real analytic function on $U$, and one can write
\begin{equation}
\label{eq:decpsig}
\psi_\omega (z, \overline{z}) \, = \,  \sum_{k=0}^{+\infty}\Psi_{4+k}(z, \overline{z}) \, ,
\end{equation}
where, for every $k\in \mathbb{N}$, the component $\Psi_{4+k}$ is a real homogeneous polynomial in the variables $z$ and $\overline{z}$ of degree $4+k$. In particular, we have that $\Psi_4$ and $\Psi_5$ satisfy the equations
\begin{align}
\label{eq:bilapp4}
\Delta^2 \, \Psi_4 & =  -2s_\omega \, ,\\
\label{eq:bilapp5} 
\Delta^2 \, \Psi_5 & =  0 \, ,
\end{align}
where $\Delta$ is the Euclidean Laplace operator of $\mathbb{C}^m$. Finally, the polynomial $\Psi_4$ can be written as
\begin{equation}
\label{eq:decp4}
\Psi_4\st z, \overline{z}\dt \,\, = 
\,\,  \Big(- \frac{s_{\omega}}{16m(m +1)} \, + \, \Phi_2 \, + \,  
\Phi_4  \, \Big) \,  |z|^4 \, ,
\end{equation}
where $\Phi_2$ and $\Phi_4$ are functions in the second and fourth eigenspace of $\Delta_{\mathbb{S}^{2m-1}}$, respectively.

%

%


%





%



\end{prop}

\begin{proof}
Without loss of generality, we assume that $p$ is a smooth point, since, if it is not, it is sufficient to consider the local lifting of the quantities involved. The first assertion is a consequence of the $\dd$-lemma combined with the existence of normal coordinates and it is a classical fact. The real analiticity of $\psi_\omega$ follows by elliptic regularity of solutions of the constant scalar curvature equation $\mathbf{S}_{eucl}(\psi_\omega) \, = \, s_\omega$, which, according to~\eqref{eq:espsg}, \eqref{eq:defLg} and \eqref{eq:defQg}, reads
\begin{equation}
\Delta^2\psi_\omega \,\, = \,\,  -2s_\omega \, + \, 8 \,  \tr_{\omega}(i\dd \psi_\omega \circ i\dd \Delta \psi_\omega) \, + \, 4 \, \Delta \, \tr_{\omega} (i\dd \psi_\omega \circ i\dd \psi_\omega) \, + \,  2 \, \RR_{eucl}(\psi_\omega) \,.
\end{equation}
Having the expansion~\eqref{eq:decpsig} at hand, the equations \eqref{eq:bilapp4}, \eqref{eq:bilapp5} are now obvious, while to prove equation \eqref{eq:decp4}
we just observe that since $\Psi_4$ is a real polynomial of order $4$, it must be an even function. In particular, its restriction to $\mathbb{S}^{2m-1}$ is forced to have trivial projection along the eigenspaces of $-\Delta_{\Sp^{2m-1}}$ corresponding to the eigenvalues $\Lambda_{2k+1}$, for every $k\geq 0$. Hence, $\Psi_4$ can be written as
$$
\Psi_4 \st z,\overline{z}\dt \,\, = \,\,  \big( \Phi_0 + \Phi_2 + \Phi_4  \big) |z|^{4} \, ,
$$
where the $\Phi_k$'s are functions in the $k$-th eigenspace of $\Delta_{\Sp^{2m-1}}$. The fact that $\Phi_0 = - s_\omega/16m(m+1)$ is now an easy consequence of equation~\eqref{eq:bilapp4}.
\end{proof}

\subsection{The \K\ potential of a scalar flat ALE \K\ resolution.}

We start by recalling the concept of Asymptotically Locally Euclidean (ALE for short) K\"ahler resolution of an isolated quotient singularitiy. We let $\Gamma\triangleleft U(m)$ be a finite subgroup of the unitary group acting freely away from the origin.
and we say that a complete noncompact \K\ manifold $(X_\Gamma, h, \eta)$ of complex dimension $m$, where $h$ is the \K\ metric and $\eta$ is the \K\ form, is an ALE \K\ manifold with group $\Gamma$ if there exist a positive radius $R>0$ and a quotient map   
$\pi: X_{\Gamma}\rightarrow \CC^{m}/\Gamma$,
such that 
\begin{equation}
\pi: X_{\Gamma}\setminus \pi^{-1}(B_R)  \longrightarrow \st\CC^{m} \setminus B_{R} \dt/\Gamma
\end{equation}
is a biholomorphism and in standard Euclidean coordinates the metric $\pi_*h$ satisfies the expansion 
\begin{equation}
\left|  \frac{\partial^\alpha}{\partial x^\alpha} \left(  \big (\pi_{*}h)_{i \bar{j}}  \, - \, \frac{1}{2} \, \delta_{i\bar{j}}  \right) \right| \,\, = \,\,  \mathcal{O}\st |x|^{-\tau - |\alpha|}\dt\,, 
\end{equation}
for some $\tau>0$ and every multindex $\alpha \in \NN^m$. 

\begin{remark} The reader must be aware of the fact that the above definition gives only a special class of \K\ ALE manifolds. In particular we are identifying the complex structure outside a compact subset with the standard one, while in general it could be only asymptotic to it and in fact the complex structure could not even admit holomorphic coordinates at infinity as shown for example by Honda (\cite{Honda}) also in the scalar flat case.
\end{remark}

\begin{remark}
In the following, we will make as systematic use of $\pi$ as an identification and, consequently, we will make no difference between $h$ and $\pi_{*} h$ as well as between $\eta$ and $\pi_* \eta$.
\end{remark}

\begin{remark}
\label{nolinear}
It is a simple exercise to prove that if $\Gamma$ is nontrivial, then there are no $\Gamma$-invariant linear functions on $\CC ^m$, and thus, with the notations introduced in section~\ref{eigen}, we have that $\Lambda^{\Gamma}_1 > \Lambda_1$. 
This will be repeatedly used in our arguments in Proposition \ref{asintpsieta}, Proposition \ref{GAP} and Lemma \ref{stimabiarmonichebase}. 

\end{remark}

\vspace{11pt}

\noindent We are now ready to present a result which describe the asymptotic behaviour of the \K\ potential of a scalar flat ALE \K\ metric. This can be though as the analogous of Proposition~\ref{proprietacsck}. We omit the proof because in the spirit it is very similar to the one of the aforementioned proposition and the details can be found in~\cite{ap1} 

\begin{prop}\label{asintpsieta2}
Let $(X_\Gamma, h, \eta)$ be a scalar flat ALE \K\ resolution of an isolated quotient singularitiy and let $\pi: X_\Gamma \to \mathbb{C}^m/\Gamma$ be the quotient map. Then for $R>0$ large enough, we have that on $X_\Gamma\setminus \pi^{-1}(B_R)$ the K\"ahler form can be written as
\begin{equation}
\eta \,\, = \,\, i\partial\overline{\partial} \st \,  \frac{|x|^2}{2} \, + \, e({\Gamma}) \, |x|^{4-2m}  \, -  \, c({\Gamma}) \, |x|^{2 - 2m} \,  + \, \psi_{\eta}\st x \dt \dt \,,  \qquad \hbox{with} \qquad \psi_\eta \, = \,  \mathcal{O}(|x|^{-2m}) \, ,
\end{equation}
for some real constants $e({\Gamma})$ and $c({\Gamma})$.
Moreover, the radial component $\psi_{\eta}^{(0)}$ in the Fourier decomposition of $\psi_\eta$ is such that
$$
\psi_{\eta}^{(0)}\st |x| \dt=\mathcal{O}\st |x|^{6-4m} \dt \, .
$$ 
\end{prop}

\noindent In the case where the ALE \K\ metric is Ricci-flat it is possible to obtain sharper  estimates for the deviation of the \K\ potential from the Euclidean one, indeed it happens that $e\st \Gamma\dt=0$. This is far form being obvious and in fact it is an important result of Joyce (\cite{j}, Theorem 8.2.3 pag 175). With the following proposition we  now give an improvement of Joyce's result which will turn out to be crucial in the rest of the paper.

\begin{prop}
\label{asintpsieta}
Let $( {X}_{\Gamma}, h, \eta )$ be as in Proposition \ref{asintpsieta2}. Moreover  let $\Gamma \triangleleft U(m)$ be nontrivial and $e\st \Gamma \dt=0$. Then for $R>0$ large enough, we have that on $X_\Gamma\setminus \pi^{-1}(B_R)$ the K\"ahler form can be written as
\begin{equation}
{\eta} \,\, = \,\, i\partial\overline{\partial} \st \, \frac{|x|^2}{2} \, - \, c({\Gamma}) \, |x|^{2 - 2m} \, + \, \psi_{\eta} \st x \dt \dt  \,,  \qquad \hbox{with} \qquad \psi_\eta \, = \,  \mathcal{O}(|x|^{-2m}) \, ,\label{eq: poteta}
\end{equation}
for some positive real constant $c({\Gamma})>0$. 
Moreover, the radial component $\psi_{\eta}^{(0)}$ in the Fourier decomposition of $\psi_\eta$ is such that
$$
\psi_{\eta}^{(0)}\st |x| \dt=\mathcal{O}\st |x|^{2-4m} \dt \, .
$$ 
\end{prop} 

\begin{proof}
By~\cite[Theorem 8.2.3]{j}, we have that on $X_\Gamma\setminus\pi^{-1}(B_R)$ the K\"ahler form $\eta$ can be written  as
$$
{\eta} \,\, =  \,\, i\dd\st \, \frac{|x|^{2}}{2} - c({\Gamma}) \, |x|^{2 - 2m} +\psi_{\eta}\st x \dt\dt   
\quad\quad \hbox{with} \quad\quad 
\psi_{\eta}\st x \dt=\mathcal{O}\st |x|^{2 - 2m - \gamma} \dt \, ,
$$
for some $\gamma \in (0,1)$.
Since $({X}_{\Gamma},h)$ is  scalar flat, arguing as in Proposition~\ref{proprietacsck}, we deduce that $\psi_{\eta}$ is a real analytic function. To obtain the desired estimates on the decay of $\psi_\eta$, we are going to make use of the equation $\mathbf{S}_{eucl} (\psi_{\eta} -c({\Gamma})|x|^{2-2m}) = 0$. By means of identity~\eqref{eq:espsg}, \eqref{eq:defLg} and \eqref{eq:defQg}, this can be rephrased in terms of $\psi_\eta$ as follows
\begin{align}
\label{eq:psieta}
\nonumber
{\Delta^2\psi_\eta} \,\, = &  \,\,\,   8\, \tr  \big(  i\dd \st \psi_{\eta} -c({\Gamma})\, |x|^{2-2m}\dt \circ i\dd\Delta\psi_{\eta} \big) \, \\
 &+ \, 4 \, \Delta \, \tr \big(  i\dd\st \psi_{\eta} -c({\Gamma}) \, |x|^{2-2m} \dt  \circ i\dd\st \psi_{\eta} -c({\Gamma}) \, |x|^{2-2m} \dt         \big) 
 \\&\,\, + 2\RR_{eucl}\big( \psi_{\eta} -c({\Gamma}) \,|x|^{2-2m}\big) \, ,
\end{align} 
where, in writing the first summand on the right hand side, we have used the fact that $\Delta |x|^{2-2m} = 0$. Since $\psi_\eta = \mathcal{O} ( |x|^{2 - 2m - \gamma} )$, for some $\gamma \in (0,1)$, it is straightforward to see that all of the terms on the right hand side can be estimated as $\mathcal{O}(|x|^{-2-4m-\gamma})$, with the only exception of the purely radial term 
$$
\Delta \,  \tr \big( (i\dd |x|^{2 - 2m} ) \circ (i\dd |x|^{2 - 2m} ) \big) \,\, = \,\, \mathcal{O}(|x|^{-2-4m})\,.
$$
For sake of convenience, we set now the right hand side of the above equation equal to $F/2$, so that 
$$
\Delta^2 \psi_\eta \,\, = \,\, F \, .
$$
It is now convenient to expand both $\psi_\eta$ and $F$ in Fourier series as 
$$
\psi_{\eta}(x) \, = \, \sum_{k=0}^{ +\infty} \psi_{\eta}^{(k)} (|x|) \, \phi_{k}(x/|x|) \quad \quad \hbox{and} \quad \quad F(x) \, = \, \sum_{k=0}^{ +\infty} F^{(k)}(|x|) \, \phi_{k}(x/|x|) \, ,
$$
where the functions $\{\phi_k\}_{k \in \mathbb{N}}$, are the eigenfunctions of the spherical laplacian $\Delta_{\Sp^{2m-1}}$ on $\Sp^{2m-1}$, counted with multplicity. Since $\phi_0 \equiv |\Sp^{2m-1}|^{-1/2}$, we will refer to $\psi_\eta^{(0)}$ and $F^{(0)}$ as the radial part of $\psi_\eta$ and $F$, respectively. We also notice that in the forthcoming discussion it will be important to select among the eigenfunctions $\phi_k$'s, only the ones which are $\Gamma$-invariant, in order to respect the quotient structure. So far, we have seen that $F^{(0)} = \mathcal{O}(|x|^{-2-4m})$ and $F^{(k)} = \mathcal{O}(|x|^{-2-4m-\gamma})$, for $k \geq 1$. On the other hand, using the linear ODE satisfied by the components $\psi_\eta^{(k)}$, it is not hard to see that their general expression is given by
$$
\psi_{\eta}^{(k)} (|x|)\,\, = \,\, a_k |x|^{4 - 2m - \alpha(k)} +b_k|x|^{2 - 2m - \alpha(k)}+ c_k |x|^{\alpha(k)} + d_k |x|^{\alpha(k) + 2} +\tilde{\psi}_{\eta}^{(k)}(|x|) \,,
$$
where, in view of the behavior of the $F^{(k)}$'s, the functions $\tilde{\psi}_\eta^{(k)}$ are such that
$$
\tilde{\psi}_\eta^{(0)} = \mathcal{O}(|x|^{2-4m}) \quad \quad \hbox{and} \quad \quad \tilde{\psi}_\eta^{(k)} = \mathcal{O}(|x|^{2-4m-\gamma}), \quad \hbox{for $k \geq 1$}\, ,
$$
and the integers $\alpha(k)$'s are such that $\alpha(k)=h$ if and only if $\phi_k$ belongs to the $h$-th eigenspace.
Since the cited Joyce's result implies that $\psi_\eta^{(k)} = \mathcal{O}(|x|^{2-2m-\gamma})$, it is easy to deduce that $c_k =0 =d_k$, for every $k \in \mathbb{N}$. Moreover, we have that $a_0 = 0 = b_0$ and thus $\psi_\eta^{(0)} = \mathcal{O}(|x|^{2-4m})$, as wanted. The same kind of considerations imply that the components $\psi_\eta^{(k)}$'s satisfy the desired estimates for every $k \geq 2m + 1$, that is for every $k$ such that $\alpha(k) \geq 2$. For $1\leq k \leq 2m$, we have that $a_k = 0$, but a priori nothing can be said about the $b_k$'s and thus at a first glance, one has that
$$
\psi_{\eta}^{(k)} (|x|)\,\, = \,\, b_k|x|^{1 - 2m } +\tilde{\psi}_{\eta}^{(k)}(|x|) \, , \quad \hbox{for $1\leq k \leq 2m$} \, .
$$
As it has been pointed out in Remark~\ref{nolinear}, there are no $\Gamma$-invariant eigenfunctions for $\Delta_{\Sp^{2m-1}}$ in the first eigenspace. This means that the components $\psi_\eta^{k}$'s, with $1\leq k \leq 2m$ do not appear in the Fourier expansion of $\psi_\eta$ and hence $\psi_\eta(x) = \mathcal{O}(|x|^{-2m})$.
\end{proof}

\noindent  If the space $\st X_{\Gamma}, h, \eta \dt$ is Ricci-flat then the decaying rate at infinity of $\psi_{\eta}$ doesn't improve as one could expect, indeed it is the same as that  in Proposition \ref{asintpsieta}. However, Ricci-flat ALE K\"ahler manifolds enjoy another property, probably well known to experts but apparently not easy to find in the literature, needed in the sequel. 

 \begin{lemma}\label{misuraeuclidea}
Let $\st X_{\Gamma}, h, \eta \dt$ be a Ricci flat ALE K\"ahler resolution of an isolated quotient singularitiy 
and $\pi:X_{\Gamma}\rightarrow \CC^{m}/\Gamma$ be the quotient map.
 Then on $X_{\Gamma}\setminus \pi^{-1}\st 0 \dt$ we have
\begin{equation}d\mu_{\eta}= \pi^{*}d\mu_{0}\,,\end{equation}
and for $R>0$
\begin{equation}\text{Vol}_{\eta}\st X_{\Gamma,R}\dt=\frac{|\Sp^{2m-1}|}{2m\left| \Gamma\right|}R^{2m}\,.\end{equation}
\end{lemma}
\begin{proof}
Let $\pi_{\Gamma}:\CC^{m}\rightarrow \CC^{m}/\Gamma$ the canonical holomorphic quotient map, since 
\begin{equation}\rho_{\eta} =0\,,\end{equation}
on $\st\CC^{m}\setminus B_{R}\dt/\Gamma$  we have 
\begin{equation}i\dd\sq\log\st\det\st \st\pi_{\Gamma}\dt^{*}\st\pi^{-1}\dt^{*}\eta\dt\dt\dq=0\,.\end{equation} 
We want to prove that on $\CC^{m}\setminus \sg 0 \dg$
\begin{equation}\det\st\st\pi_{\Gamma}\dt^{*}\st\pi^{-1}\dt^{*}\eta\dt\equiv \frac{1}{2^{m}}\,.\end{equation}
By Proposition \ref{asintpsieta} we have  on $\CC^{m}\setminus B_{R}$ 
\begin{equation}\st\pi_{\Gamma}\dt^{*}\st\pi^{-1}\dt^{*}\eta_{i\bar{\jmath}}=\frac{\delta_{i\bar{\jmath}}}{2}- \cga \partial_{i}\overline{\partial}_{j} \left|x\right|^{2-2m} +	\mathcal{O}\st \left|x\right|^{-2m}  \dt \end{equation}

\noindent that implies immediately

\begin{align}
\log\st \det\st \st\pi_{\Gamma}\dt^{*}\st\pi^{-1}\dt^{*}\eta \dt \dt=&-m\log\st 2 \dt+\mathcal{O}\st \left|x\right|^{-2-2m} \dt\,.
\end{align}

\noindent On $\CC^{m}\setminus B_{R}$ we have
\begin{align}
i\dd \log\st \det\st \st\pi_{\Gamma}\dt^{*}\st\pi^{-1}\dt^{*}\eta \dt \dt=&-id\st \partial \log\st \det\st  \st\pi_{\Gamma}\dt^{*}\st\pi^{-1}\dt^{*}\eta \dt\dt\dt\,, 
\end{align}

so 
\begin{equation}\partial \log\st \det\st \st\pi_{\Gamma}\dt^{*}\st\pi^{-1}\dt^{*}\eta \dt \dt\in H^{1}\st \CC^{m}\setminus B_{R},\CC \dt\end{equation} but $H^{1}\st \CC^{m}\setminus B_{R},\CC \dt=0$  
and there exists $h_{1}\in C^{1}\st \CC^{m}\setminus B_{R},\CC  \dt$ such that
\begin{equation}\partial \log\st \det\st \st\pi_{\Gamma}\dt^{*}\st\pi^{-1}\dt^{*}\eta \dt \dt=dh_{1}=\partial h_{1} +\overline{\partial}h_{1}\qquad\qquad \overline{\partial}h_{1}=0\,.\end{equation}

Analogously, there is $h_{2}\in C^{1}\st \CC^{m}\setminus B_{R},\CC  \dt$ such that
\begin{equation}\overline{\partial} \sq\log\st \det\st \st\pi_{\Gamma}\dt^{*}\st\pi^{-1}\dt^{*}\eta \dt \dt -h_{1} \dq=dh_{2}=\partial h_{2}+\overline{\partial}h_{2}\qquad \qquad \partial h_{2}=0\,.\end{equation}

It is now clear that 
\begin{align}
d\sq \log\st \det\st \st\pi_{\Gamma}\dt^{*}\st\pi^{-1}\dt^{*}\eta \dt \dt-h_{1}-h_{2} \dq=&0\,.
\end{align}
We conclude that on $\CC^{m}\setminus B_{R}$
\begin{equation}\log\st \det\st \st\pi_{\Gamma}\dt^{*}\st\pi^{-1}\dt^{*}\eta \dt \dt=h_{1}+h_{2}+K\qquad K\in \RR \qquad \mathfrak{Im}h_{2}=-\mathfrak{Im}h_{1} \end{equation}
moreover  $h_{1},\overline{h_{2}}$ are holomorphic on $\CC^{m}\setminus B_{R}$ and by Hartogs extension theorem they are extendable to functions $H_{1},H_{2}$ holomorphic on $\CC^{m}$. Since $H_{1},H_{2}$ are holomorphic, their real and imaginary parts are harmonic with respect to the euclidean metric on $\CC^{m}$ and by assumptions on $\eta$  we have on  $\CC^{m}\setminus B_{R}$
\begin{equation}\mathfrak{Re}H_{1}+\mathfrak{Im}H_{2}+K=-m\log\st 2 \dt+\mathcal{O}\st \left|x\right|^{-2-2m} \dt\,.\end{equation}  
Since $\mathfrak{Re}H_{1}+\mathfrak{Im}H_{2}+K$ is harmonic and bounded, Liouville theorem implies it is constant, so that
\begin{equation} \det\st \st\pi_{\Gamma}\dt^{*}\st\pi^{-1}\dt^{*}\eta \dt=\frac{1}{2^{m}}\end{equation}

\noindent We can now see that

\begin{align}
\frac{1}{m!}\st\pi_{\Gamma}\dt^{*}\sq \st\pi^{-1}\dt^{*}\eta \dq^{\wedge m}=&d\mu_{0}\,.
\end{align}

\noindent and then

\begin{equation}
\text{Vol}_{\eta}\st X_{\Gamma,R}\dt=\int_{B_{R}/\Gamma\setminus \sg 0 \dg}d\mu_{\st\pi^{-1}\dt^{*}\eta}=\frac{| \Sp^{2m-1} |}{2m\left| \Gamma\right|}R^{2m}
\end{equation}

\noindent 	so the lemma follows.
\end{proof}

The above proposition might be well known to experts but we couldn't find any reference.




\section{Linear analysis on a Kcsc orbifold}\label{linearanalysis}

\noindent In this section we develop the linear analysis for the operator $\Lg$ and we do it in full generality even if, in this work, we will use  only some particular cases of this theory. We  distinguish between two sets of points: $\sg p_1,\ldots,p_{N}\dg$  with neighborhoods biholomorphic to  a ball of  $\CC^m/\Gamma_{j}$ with $\Gamma_{j}$ nontrivial such that $\CC^{m}/\Gamma_{j}$ admits an ALE  Kahler scalar-flat resolution  $\st X_{\Gamma_{j}},h,\eta_{j} \dt$ with $e\st \Gamma_{j} \dt=0$ and the set (possibly empty) $\sg q_1,\ldots,q_{K}\dg$ whose points have  neighborhoods biholomorphic to a ball of  $\CC^m/\Gamma_{N+l}$ such that
$\CC^{m}/\Gamma_{N+l}$ admits a scalar flat ALE resolution $(Y_{\Gamma_{N+l}},k_{l},\theta_{l})$ with $e({\Gamma_{N+l}})\neq 0$. To simplify the notation we set
\begin{equation*}
\p \, := \, \sg p_{1},\ldots,p_{N}\dg, \quad \q \, := \, \sg q_{1},\ldots,q_{K}\dg  , \quad \hbox{and} \quad M_{\p,\q} \, := \, M\setminus \st \p\cup\q \dt\,.
\end{equation*}

\begin{cav}
{\em We agree that, if $\q=\emptyset$, then $M_{ \p }:=M_{\p,\emptyset}$. When this case occurs and  whenever an object, that could be a function or a tensor, has indices relative to elements of $\q$ we set these indices to $0$.  }
\end{cav}

\subsection{The bounded kernel of $\Lg$.} As usual we let $(M, g, \omega)$ be a compact Kcsc  orbifold with isolated singularities and we assume that the kernel of the linearized scalar curvature operator $\Lg$ defined in~\eqref{eq:defLg} is nontrivial, in the sense that it contains also nonconstant functions. By the standard Fredholm theory for self-adjoint elliptic operators, we have that such a kernel is always finite dimensional. Throughout the paper we will assume that it is $(d+1)$-dimensional and we will set
\begin{equation}
\label{nontrivial_ker}
\ker (\Lg) \,\, = \,\, span_{\RR} \, \{\varphi_0, \varphi_1, \ldots , \varphi_d \} \, ,
\end{equation}
where $\varphi_0 \equiv 1$, $d$ is a positive integer and $\varphi_1, \ldots, \varphi_d$ is a collection of linearly independent functions in $\ker(\Lg)$ with zero mean and normalized in such a way that $||\varphi_i||_{L^2(M)} = 1$, $i=1, \ldots, d$, for sake of simplicity. From~\cite{ls} we recover the following charachterization of $\ker(\Lg)$.
\begin{prop}
\label{espsg}
Let $(M,g,\omega)$ be a compact constant scalar curvature \K\ orbifold  with isolated singularities.
Then, the subspace of $\ker(\Lg)$ given by the elements with zero mean is in one to one correspondence with the space of holomorphic
vector fields which vanish somewhere in $M$.
\end{prop}

The aim of this section is to study the solvability of the linear problem
\begin{equation}\label{eq:linear}
 \Lg u = f
\end{equation}
on the complement of the singular points in $M$. In order to do that, we introduce some notation as well as an appropriate functional setting. We consider geodesics balls $B_{r_{0}}\st p_{j} \dt, B_{r_{0}}\st q_{l} \dt$ of radius $r_0>0$, with K\"ahler normal coordinates centered at the points $p_{j}$'s and $q_{l}$'s and we set
$$
M_{r_{0}} \,\, := \,\, M\setminus \bigg(   \bigcup_{j=1}^{N} B_{r_{0}}\st p_{j} \dt\,\cup\,  \bigcup_{l=1}^{K} B_{r_{0}}\st q_{l} \dt  \bigg) \, .
$$
For $\delta\in\mathbb{R}$ and $\alpha\in (0,1)$, we define the weighted H\"older space $C_{\delta}^{k,\alpha}\st M_{\p,\q} \dt$
as the set of functions $f\in C_{loc}^{k,\alpha}\st M_{\p,\q} \dt$ such that the norm
\begin{align*}
\left\|f\right\|_{C_{\delta}^{k,\alpha}\st M_{\p,\q} \dt}
\, := \,\,\,  &\left\|f\right\|_{C^{k,\alpha}( M_{r_{0}} )}
\hspace{-1,5cm}
&+\sup_{ 0<r\leq r_{0}} r^{-\delta}\sum_{j=1}^{N} \left\|\left.f(r\cdot)\right|_{B_{r_{0}}\st p_{j} \dt}\right\|_{C^{k,\alpha}\st B_{2}\setminus B_{1}  \dt}\\
& &+\sup_{ 0<r\leq r_{0}} r^{-\delta}\sum_{l=1}^{K}\left\|\left.f(r\cdot)\right|_{B_{r_{0}}\st q_{l} \dt}\right\|_{C^{k,\alpha}\st B_{2}\setminus B_{1}  \dt}
\end{align*}
is finite.
We observe that the typical function $f\in C_{\delta}^{4,\alpha}\st M_{\p,\q} \dt$
beheaves like
\begin{equation*}
f(\cdot)
\, \,= \,\, \mathcal{O}\big( d_{\omega}\st p_{j}, \cdot \dt^{\delta}\big) \,, \quad \hbox{on} \quad {B_{r_{0}}\st p_{j}\dt}\qquad \hbox{and} \qquad   f (\cdot)
\, \, = \,\, \mathcal{O}\big( d_{\omega}\st q_{j}, \cdot \dt^{\delta}\big)
\,, \quad \hbox{on} \quad {B_{r_{0}}\st q_{j}\dt} \, ,
\end{equation*}
where $d_{\omega}$ is the Riemannian distance induced by the Kahler metric $\omega$.

We are now in the position to solve equation~\eqref{eq:linear} in the case where the datum $f$ is {\em orthogonal} to $\ker(\Lg)$. By this we mean that, looking at $f$ as a distribution, we have \begin{equation}
\label{eq:orto}
\langle f  \, | \, \varphi_i \rangle_{\mathscr{D}' \times \mathscr{D} } \,\, = \,\, 0 \, ,
\end{equation}
for every $i=0,\ldots ,d$, where we denoted by $\langle \cdot \, | \, \cdot\cdot \, \rangle_{\mathscr{D}' \times \mathscr{D} }$ the distributional pairing and the functions $\varphi_i$'s are as in~\eqref{nontrivial_ker}. It is worth pointing out that since the functions in $\ker(\Lg)$ are smooth, everything makes sense.

To solve equation \eqref{eq:linear} we need to ensure the Fredholmness of the operator $\Lg$ on the functional spaces we have chosen. The Fredholm property depends heavily on the choice of weights, indeed the operator $\Lg$ is Fredholm if and only if the weight is not an indicial root (for definition of indicial roots we refer to \cite{ap2}) at any of the points $p_{j}$'s or $q_{l}$'s. Since in normal coordinates on a punctured ball, the principal part of our operator $\Lg$ is 'asymptotic' to the Euclidean Laplacian $\Delta$, then the set of indicial roots of $\Lg$ at the center of the ball coincides with the set of indicial roots of $\Delta$ at $0$. We recall that the set of indicial roots of $\Delta$ at $0$ is given by $\ZZ\setminus \sg 5-2m, \ldots ,-1 \dg$ for $m\geq 3$ and $\ZZ$ for $m=2$.

By the analysis in~\cite{ap1}, we recover the following result, which provides the existence of solutions in Sobolev spaces for the linearized equation together with {\em a priori} estimates in suitable weighted H\"older spaces.

\begin{teo}
\label{invertibilitapesatobase}
For every $f \in L^p(M)$, $p>1$, satisfying the orthogonality condition~\eqref{eq:orto}, there exists a unique solution $u \in W^{4,p}(M)$ to
$$
\Lg u \,\, = \,\, f \, ,
$$
which satisfy the condition~\eqref{eq:orto}.
Moreover, the following estimates hold true.
\begin{itemize}
\item
If $m\geq 3$ and in addition $f\in C^{0,\alpha}_{\delta-4}(M_{\p,\q})$ with $\delta \in (4-2m \, ,0 )$, then the solution $u$ belongs to $C^{4,\alpha}_{\delta}(M_{\p,\q})$ and satisfy the estimates
\begin{equation}
\left\|u\right\|_{C^{4,\alpha}_{\delta}(M_{\p,\q})}  \,\, \leq \,\,  C \, \left\|f\right\|_{C^{0,\alpha}_{\delta-4}(M_{\p,\q})} \, ,
\label{eq:stimapesatabase}
\end{equation}
for some positive constant $C>0$.
\item
If $m=2$ and in addition $f\in C^{0,\alpha}_{\delta-4}(M_{\p,\q})$ with $\delta \in (0 \,,1)$, then the solution $u$ belongs to $C^{4,\alpha}_{loc}(M_{\p,\q})$ and satisfy the following estimates
\begin{equation}
\bigg\| \,  u - \sum_{j=1}^{N}u ( p_{j}) \chi_{p_j}  - \sum_{l=1}^{K}u (q_{l})\chi_{q_l} \, \bigg\|_{C^{4,\alpha}_{\delta}(M_{\p,\q})}   \!\!\!\! + \,
\sum_{j=1}^{N}|u (p_{j})| \, + \sum_{l=1}^{K}|u(q_{l})|  \,\, \leq \,\,  C  \, \left\|f\right\|_{C^{0,\alpha}_{\delta-4}(M_{\p,\q})} \, ,
\label{eq:stimapesatabase2}
\end{equation}
where $C>0$ is a positive constant and the functions $\chi_{p_1}, \ldots, \chi_{p_N}$ and $\chi_{q_1}, \ldots, \chi_{q_K}$ are smooth cutoff functions supported on small balls centered at the points $p_{1}, \ldots, p_N$ and $q_1, \ldots, q_K$, respectively  and identically equal to $1$ in a neighborhood of these points.
\end{itemize}
\end{teo}

\begin{remark}
Some comments are in order about the choice of the weighted functional setting. Concerning the case $m\geq 3$ we observe that
the choice of the weight $\delta$ in the interval $(4-2m,0)$ is motivated by the fact that only for $\delta$ in this range the kernel of $\Lg$ viewed as an operator from $C_{\delta}^{4,\alpha}\st M_{\p,\q} \dt$ to $C_{\delta-4}^{0,\alpha}\st M_{\p,\q} \dt$ coincides with the bounded kernel, which has been denoted for short by $\ker\st \Lg \dt$.In the case $m=2$ it is no longer possible to make a similar choice, since $4-2m$ becomes $0$ and thus, at a first glance, the  natural choice for the weight is not evident. One possibility is to take the weight in the first indicial interval before $0$, which for $m=2$ is given $(-1,0)$.  In this case, one would get a functional space which is strictly larger than the bounded kernel $\ker\st \Lg \dt$. We prefer instead to choose the weight in the first indicial interval after $0$, which for $m=2$ is given by $(0,1)$. This time, the bounded kernel of $\Lg$ is no longer contained in the possible domains of our operator, since the functions belonging to these spaces have to vanish at points $\p$ and $\q$.  On one hand this is responsible for the more complicate expression in the {\em a priori} estimate~\eqref{eq:stimapesatabase2}, but one the other hand this choice of the weight will reveal to be more fruitful.
Indeed, in view of the linear analysis on ALE \K\ manifolds performed in section~\ref{lineareALE} and with the notation introduced therein, one has that the corresponding linearized scalar curvature operator
\begin{equation*}
\Le : C_{\delta}^{4,\alpha}\st  X_{\Gamma} \dt\rightarrow C_{\delta-4}^{0,\alpha}\st  X_{\Gamma} \dt
\end{equation*}
admits an inverse (up to a constant) for $\delta \in (0,1)$. Since the possibility of choosing the same weight for the linear analysis on both the base orbifold and the model spaces will be crucial in the subsequent nonlinear arguments, this yields a reasonable justification of our choices. In the same spirit, we point out that, for $m=3$ and $\delta \in (4-2m,0)$ the operator $\mathbb{L}_\eta$ defined above is invertible, as it is proven in Theorem~\ref{isomorfismopesati}.
\end{remark}

In order to drop the orthogonality assumption~\eqref{eq:orto} in Theorem~\ref{invertibilitapesatobase} and tackle the general case, we first need to investigate the behaviour of the fundamental solutions of the operator $\Lg$. This will be done in the following subsection.


\subsection{Multi-poles fundamental solutions of $\Lg$.}\label{multipoles}

The aim of this subsection is twofold. On one hand, we want to produce the tools for solving equation~\eqref{eq:linear} on $M_{\p,\q}$, when $f$ is not necessarily {\em orthogonal} to $\ker{(\Lg)}$. On the other hand, we are going to determine under which global conditions on $\ker(\Lg)$ we can produce a function, which near the singularities behaves like the principal non euclidean part of the \K\ potential of the corresponding ALE resolution. In concrete, building on Propositions~\ref{asintpsieta} and~\ref{asintpsieta2}, we aim to establish the existence of a function, which blows up like $|z|^{2-2m}$ near the $p_j$'s and like $|z|^{4-2m}$ near the $q_l$'s. Such a function will then be added to the original \K\ potential of the base orbifold in order to make it closer to the one of the resolution. At the same time, for obvious reasons, it is important to guarantee that this new \K\ potential will produce on $M_{\p,\q}$ the smallest possible deviation from the original scalar curvature, at least at the linear level.
Thinking of $g$ as a perturbation of the flat metric at small scale, we have that $\Lg$ can be thought of as a perturbation of $\Delta^2$. Since $|z|^{2-2m}$ and $|z|^{4-2m}$ satisfy equations of the form
$$
\Delta^2(A|z|^{2-2m} + B |z|^{4-2m}) = C \Delta\delta_0 + D\delta_0 \,\, ,
$$
where $\delta_0$ is the Dirac distribution centered at the origin and $A,B,C$ and $D$ are suitable constants, we are led to study these type of equations on $M$ for the operator $\Lg$.


\begin{prop}
\label{balancrough}
Let $(M,g,\omega)$ be compact Kcsc orbifold of complex dimension $m$ and let ${\rm ker} {(\Lg)} = { span}\{\varphi_0, \varphi_1, \dots,\varphi_d\}$, as in~\eqref{nontrivial_ker}. Let $(f_0, \ldots, f_d )$ be a vector in $\RR^{d+1}$. Assume that the following {\em linear balancing condition} holds
\begin{eqnarray}
\label{eq:generalbal}
f_i \, + \, \sum_{l=1}^K a_l\varphi_i(q_l) \, +\, \sum_{j=1}^Nb_j(\Delta\varphi_i)(p_j) \, + \, \sum_{j=1}^Nc_j\varphi_i(p_j)   &  = &  0 \,, \quad\quad\quad\qquad\qquad \hbox{$i = 1, \dots, d$} \, , \\
\label{eq:fixnu}
f_0 \, {\rm Vol}_\omega(M)   \, + \, \sum_{l=1}^K a_l \, + \, \sum_{j=1}^N c_j  & = & \nu  \, {\rm Vol}_\omega(M) \, ,
\end{eqnarray}
for some choice of the coefficients $\nu$,  ${\bf a}=(a_1, \dots, a_K)$, ${\bf b} =(b_1, \dots, b_N)$ and ${\bf c} =(c_1, \dots, c_N)$.
Then, there exist a distributional solution ${U} \in \mathscr{D}'(M)$ to the equation
\begin{equation}
\Lg   [{U}]  \,  + \, \nu  \,\,\,  = \,\,\,   \sum_{i=0}^d f_i \, \varphi_i \, + \, \sum_{l=1}^K a_l\, \delta_{q_{l}} \, + \, \sum_{j=1}^N  b_j \, \Delta\delta_{p_{j}} \, + \, \sum_{j=1}^N c_j \, \delta_{p_{j}} \, , \qquad \hbox{in \,\,\,$M$}\, . \label{eq:LGabcd}
\end{equation}
\end{prop}

\begin{proof}
Let us first remark that equations \eqref{eq:generalbal} and \eqref{eq:fixnu} imply that, for any $\varphi \in \ker(\Lg)$, one has that $\langle \, T \, | \,\varphi \, \rangle_{\mathscr{D}' \times \mathscr{D}} \, = \, 0$, where $T \in \mathscr{D}'$ is the distribution defined by
$$
T \,\,  = \,\,  \sum_{i=1}^{d}f_{i} \, \varphi_{i} \, + \, \sum_{l=1}^K a_l \, \delta_{q_{l}} \, + \, \sum_{j=1}^N b_j \, \Delta\delta_{p_{j}} \, + \, \sum_{j=1}^N c_j \, \delta_{p_{j}}  \, -  \, \nu \, .
$$
Having this in mind, we let $U \in \mathscr{D}'$ be the unique distribution such that, for every $\psi \in C^\infty(M)$
$$
\left< \, U \, |\, \psi\, \right>_{\mathscr{D}^{'} \times \mathscr{D}} \,\,  = \,\,
\left< \, T  \, |\,  \mathbb{J}_{\omega}[ \psi^{\perp}] \, \right>_{\mathscr{D}^{'} \times \mathscr{D}} \, ,
$$
where $\psi^\perp$, the component of $\psi$ which is {\em orthogonal} to $\ker(\Lg)$, is given by
\begin{equation*}
\psi^{\perp} \,\, = \,\, \psi \, - \, \frac{1}{\vol\st M \dt}\int_{M}\psi \,d\mu_{\omega} \, - \, \sum_{i=1}^{d}\varphi_{i}\int_{M}\psi \varphi_{i}\,d\mu_{\omega}\,,
\end{equation*}
and $\mathbb{J}_{\omega} : L^{2}\st M \dt/\ker\st \Lg \dt\rightarrow W^{4,2}\st M \dt/\ker\st \Lg \dt$ is inverse of $\Lg$ restricted to the orthogonal complement of $\ker(\Lg)$, given by Proposition \ref{invertibilitapesatobase}.
We claim that the distribution $U$ defined above satisfies the equation \eqref{eq:LGabcd} in the sense of distributions. With the notations just introduced, we need to show that, for every $\psi \in C^\infty(M)$, it holds
$$
\left< \, \Lg  [U ]\, |\, \psi \, \right>_{\mathscr{D}^{'} \times \mathscr{D}} \,\, = \,\,
\left< \, T \, | \, \psi \, \right> _{\mathscr{D}^{'} \times \mathscr{D}} \, .
$$
Using the definition of $U$ and the fact that $\Lg$ is formally selfadjoint, we compute
\begin{eqnarray*}
\left< \, \Lg  [U ]\, |\, \psi \, \right>_{\mathscr{D}^{'} \times \mathscr{D}} & = & \left< \, U \, |\, \Lg [\psi] \, \right>_{\mathscr{D}^{'} \times \mathscr{D}} \,\,\, = \,\,\, \left< \, U \, |\, \Lg [\psi^\perp] \, \right>_{\mathscr{D}^{'} \times \mathscr{D}} \,\,\, = \,\,\, \left< \, T \, |\, \mathbb{J}_\omega  \big[   (\Lg [\psi^\perp])^\perp  \big] \, \right>_{\mathscr{D}^{'} \times \mathscr{D}}  \\
& = & \left< \, T \, | \, \psi^\perp  \right> _{\mathscr{D}^{'} \times \mathscr{D}} \,\, \,= \,\,\, \left< \, T \, | \, \psi \, \right> _{\mathscr{D}^{'} \times \mathscr{D}} \,,
\end{eqnarray*}
since $\psi - \psi^\perp \in \ker(\Lg)$, and thus $\left< \, T \, | \, \psi - \psi^\perp  \right> _{\mathscr{D}^{'} \times \mathscr{D}} = 0$, by a previous observation.
%
This completes the proof of the proposition.
\end{proof}

\begin{remark}
\label{gabc}
When $f_i = 0$, for $i=0, \ldots, d$, we only impose the balancing condition~\eqref{eq:generalbal}, which specializes to
\begin{equation}\label{eq: balancingbc1}
\sum_{l=1}^K a_l\varphi_i(q_l) \, +\, \sum_{j=1}^Nb_j(\Delta\varphi_i)(p_j) \, + \, \sum_{j=1}^Nc_j\varphi_i(p_j)  \,\,  = \,\,  0 \,,
\end{equation}
and we obtain a real number $\nu_{\aaa,\ccc}$, defined by the relation
\begin{equation}\label{eq: balancingbc2}
\sum_{l=1}^K a_l \, + \, \sum_{j=1}^N c_j  \,\,= \,\, \nu_{\aaa,\ccc}  \, {\rm Vol}_\omega(M) \, ,
\end{equation}
and a distribution $\GGG_{\aaa,\bbb,\ccc} \in \mathscr{D}'(M)$, which satisfies the equation
\begin{eqnarray*}
\Lg \left[  \mathbf{G}_{\aaa,\bbb,\ccc} \right] \,  + \, \nu_{\aaa,\ccc}  & =&   \sum_{l=1}^K a_l\, \delta_{q_{l}} \, + \, \sum_{j=1}^N  b_j \, \Delta\delta_{p_{j}} \, + \, \sum_{j=1}^N c_j \, \delta_{p_{j}} \, , \qquad \hbox{in \,\,\,$M$}\, .
\end{eqnarray*}
We will refer to $\GGG_{\aaa,\bbb,\ccc}$ as a {\em multi-poles fundamental solution} of $\Lg$.
\end{remark}

The following two lemmata and the subsequent proposition~\eqref{loc_structure} will give us a precise description of the behavior of a {\em multi-poles fundamental solution}  $\GGG_{\aaa,\bbb,\ccc}$ of $\Lg$ around the singular points. The same considerations obviously apply to a distributional solution $U$ of the equation~\eqref{eq:LGabcd}. The first observation in this direction can be found in \cite{ap2} and we report it here for sake of completeness. 	
\begin{lemma}
\label{Gbilapl}
Let $(M,g,\omega)$ be a Kcsc orbifold of complex dimension $m\geq 2$ and let $M_q = M \setminus \{ q\}$, with $q \in M$. Then, the following holds true.
\begin{itemize}
\item If $m\geq 3$, there exists a function
$G_{\Delta\Delta}(q,\cdot) \in \mathcal{C}_{4-2m}^{4,\alpha}(M_q) \cap \mathcal{C}^{\infty}_{loc}(M_q)$, orthogonal to $\ker(\Lg)$ inthe sense of~\eqref{eq:orto}, such that
\begin{equation}\label{eq:greenbilapl}
\Lg[G_{\Delta\Delta}(q,\cdot)]  \,\, + \,\,
\frac{2(m-1) \, |\Sp^{2m-1}|}{|\Gamma|} \,\, \big[\, 4(m-2) \,\, \delta_q \, \big] \,\,  \in \, \mathcal{C}^{0,\alpha}(M) \, ,
\end{equation}
where $|\Gamma|$ is the order of the orbifold group at $q$.
Moreover, if $z$ are holomorphic coordinates centered at $q$, it holds the expansion
\begin{equation}
\label{eq:expgreenbilapl}
G_{\Delta\Delta}(q,z)  \,\, = \,\,  |z|^{4-2m} + \, \mathcal{O}(|z|^{6-2m}) \, .
\end{equation}
\item If $m=2$, there exists a function
$G_{\Delta\Delta}(q,\cdot) \in \mathcal{C}^{\infty}_{loc}(M_q)$, orthogonal to $\ker(\Lg)$ inthe sense of~\eqref{eq:orto}, such that
\begin{equation}
\label{eq:bluebilapl2}
\Lg[G_{\Delta\Delta}(q,\cdot)] \,\, - \,\, \frac{4|\Sp^{3}|}{|\Gamma|} \, \delta_q  \,\, \in  \,\, \mathcal{C}^{0,\alpha}(M) \, ,
\end{equation}
 where $|\Gamma|$ is the order of the orbifold group at $q$.
Moreover, if $z$ are holomorphic coordinates centered at $q$, it holds the expansion
\begin{equation}
\label{eq:expbluebilapl2}
G_{\Delta\Delta}(q,\cdot) \,\, = \,\, \log(|z|) \, + \, C_{q} \, + \, \mathcal{O}(|z|^{2}) \, ,
\end{equation}
for some constant $C_q \in \RR$.
\end{itemize}
\end{lemma}

Before stating the next lemma, it is worth pointing out that $G_{\Delta \Delta} (q, \cdot)$ has the same rate of blow up as the Green function of the bi-Laplacian operator $\Delta^2$. Since we want to produce a local approximation of the {\em multi-poles fundamental solution} $\GGG_{\aaa,\bbb,\ccc}$\,, we also need a profile whose blow up rate around the singular points is the same as the one of the Green function of the Laplace operator. This will be responsible for the $\Delta \delta_p$'s terms. 
\begin{lemma}
\label{Glapl}
Let $(M,g,\omega)$ be a Kcsc orbifold of complex dimension $m\geq 2$ and let $M_p = M \setminus \{ p \}$, with $p \in M$. Then, the following holds true.
\begin{itemize}
\item
If $m \geq 3$, there exists a function
$G_{\Delta}(p,\cdot) \in \mathcal{C}_{2-2m}^{4,\alpha}(M_p) \cap \mathcal{C}_{loc}^{\infty}(M_p)$, orthogonal to $\ker(\Lg)$ inthe sense of~\eqref{eq:orto}, such that
\begin{equation}
\label{eq:greenlapl}
\Lg[G_{\Delta}(p,\cdot)] \,\,  - \,\, \frac{2(m-1) \, |\Sp^{2m-1}|}{|\Gamma|}  \,\, \Big[   \, \Delta\delta_p \,\,
+\,\,
\frac{s_\omega (m^2-m+2)}{m(m+1)}
\, \,  \delta_p  \, \Big]    \,\,  \in \, \mathcal{C}^{0,\alpha}(M) \, ,
\end{equation}
{where $|\Gamma|$ is the cardinality of the orbifold group at $p$} and $s_\omega$ is the constant scalar curvature of the orbifold.
Moreover, if $z$ are holomorphic coordinates centered at $p$, it holds the expansion
\begin{equation}
\label{eq:expgreenlapl}
G_{\Delta }(p,\cdot) \,\,  =  \,\, |z|^{2-2m}  \, + \, |z|^{4-2m} \, ( \, \Phi_2 + \Phi_4 \, )  \, + \,  |z|^{5-2m} \,
\sum_{j=0}^2\Phi_{2j+1}
 \, + \,  \mathcal{O}(|z|^{6-2m}) \, ,
\end{equation}
for suitable smooth $\Gamma$-invariant functions $\Phi_j$'s defined on $\mathbb{S}^{2m-1}$ and belonging to the $j$-th eigenspace of the operator $\Delta_{\mathbb{S}^{2m-1}}$.

\smallskip

\item
If $m = 2$,  there exists a function
$G_{\Delta}(p,\cdot) \in \mathcal{C}_{-2}^{4,\alpha}(M_p) \cap \mathcal{C}_{loc}^{\infty}(M_p)$, orthogonal to $\ker(\Lg)$ inthe sense of~\eqref{eq:orto}, such that
\begin{equation}
\label{eq:greenlapl2}
\Lg[G_{\Delta}(p,\cdot)] \,\,  - \,\, \frac{ |\Sp^{3}|}{|\Gamma|}  \,\, \Delta\delta_p \,\,
- \,\,\frac{ \,  s_\omega \, 2  \, |\Sp^{3}|}{3 \, |\Gamma|} \, \,  \delta_p  \,\,  \in \, \mathcal{C}^{0,\alpha}(M) \, ,
\end{equation}
where $|\Gamma|$ is the cardinality of the orbifold group at $p$ and $s_\omega$ is the constant scalar curvature of the orbifold.
Moreover, if $z$ are holomorphic coordinates centered at $p$, it holds the expansion
\begin{equation}
\label{eq:expbluelapl2}
G_{\Delta}(p,\cdot) \,\, = \,\,  |z|^{-2} \, + \, \log(|z|)(\Phi_2 + \Phi_4) \, + \, C_{p} \, + \,  |z| \,
 \sum_{h=0}^{2}\Phi_{2h+1} \,  + \,  \mathcal{O}(|z|^{2})
\end{equation}
for some constant $C_p \in \RR$, some $H \in \mathbb{N}$ and suitable smooth $\Gamma$-invariant functions $\Phi_h$'s defined on $\mathbb{S}^{3}$ and belonging to the  $h$-th eigenspace of the operator $\Delta_{\mathbb{S}^{3}}$. 
\end{itemize}
\end{lemma}

%

%




%

%

\begin{proof} We focus on the case $m\geq 3$ and since the computations for the case $m=2$ are very similar, we leave them to the  reader.
To prove the existence of $G_{\Delta}\st p,\cdot\dt$, we fix a coordinate chart centered at $p$ and we consider the Green function for the Euclidean Laplacian $|z|^{2-2m}$. In the spirit of Proposition~\ref{proprietacsck}, we compute
\begin{align*}
\Lg [ \,  |z|^{2-2m} \, ] \,\,\, =  \,\,\,&\st \, \Lg \, - \, \Delta^{2} \, \dt [\, |z|^{2-2m} \, ]\\
= \,\,\, &  - \, 4 \, \tr  \st \,  i\dd |z|^{2-2m}\circ i\dd\Delta \psi_{\omega} \, \dt  \, - \, 4 \, \tr \st \,  i\dd \psi_{\omega} \circ i\dd\Delta |z|^{2-2m} \, \dt\\
&-4\, \Delta\, \tr \st\,  i\dd \psi_{\omega} \circ i\dd|z|^{2-2m} \, \dt \, + \, \mathcal{O}\st|z|^{2-2m}\dt \\
=&-\frac{m}{4|z|^{2m}}\Delta^{2}\Psi_{4}+\frac{m\st m+1 \dt}{|z|^{2m+2}}\Delta\Psi_{4}-\frac{m}{4}\Delta\st \frac{\Delta\Psi_{4}}{|z|^{2m}} \dt\\
&+4m\st m+1 \dt\Delta\tr\st \frac{\Psi_{4}}{|z|^{2m+2}}  \dt+\mathcal{O}\st|z|^{2-2m}\dt
\end{align*}
where we used the explicit form of  $\Psi_{4}$	
\begin{equation}
\Psi_{4}\st z,\overline{z} \dt=-\frac{1}{4}\sum_{i,j,k,l=1}^{m}R_{i\bar{\jmath}k\bar{l}}z^{i}\overline{z^{j}}z^{k}\overline{z^{l}}
\end{equation}
and the complex form of the euclidean laplace operator
\begin{equation}
\Delta=4\sum_{i=1}^{m} \frac{\partial^{2}}{\partial z^{i}\partial \overline{z^{i}}}\,.  
\end{equation}
Expanding the real analytic function $\psi_\omega$ as $\psi_\omega \, = \, |z|^4 \, (\Phi_0 + \Phi_2 + \Phi_4) \, + \, |z|^5 \, (\Phi_1 + \Phi_3 + \Phi_5) \, + \, \mathcal{O}(|z^6|)$, where, for $h=0, 1,2$, the $\Phi_{2h}$'s and the $\Phi_{2h+1}$'s are suitable $\Gamma$-invariant functions in the $h$-th eigenspace of $\Delta_{\Sp^{2m-1}}$, we obtain
\begin{align*}
\Lg [ \,  |z|^{2-2m} \, ] \,\,\, =  \,\,\, &{|z|^{-2m}} \, \sum_{h=0}^2 c_{2h} \, \Phi_{2h} \, + \,
|z|^{1-2m} \, \sum_{h=0}^2 c_{2h+1} \, \Phi_{2h+1} \,  + \, \mathcal{O}\st |z|^{2-2m} \dt\,,
\end{align*}
where $c_0, \ldots, c_5$ are suitable constants. It is a straightforward but remarkable consequence of formula~\eqref{eq:decp4}, the fact that $c_0=0$.
\,It is then possible to introduce the corrections
$$
V_{4} \,\, = \,\, |z|^{4-2m} \, (\, C_2 \, \Phi_2 \, + \, C_4 \, \Phi_4 \,)
\qquad \hbox{and} \qquad  V_{5} \,\, = \,\, |z|^{5-2m} \, \sum_{h=0}^2 C_{2h+1} \, \Phi_{2h+1}  \,,
$$
where the coefficients $C_1, \ldots, C_5$ are so chosen that
$$
\Delta^2 \, [\, V_4 \, + \,  V_5 \, ] \,\, = \,\, {|z|^{-2m}} \,
(\, c_2 \, \Phi_2 \, + \, c_4 \, \Phi_4 \,)
\, + \,
|z|^{1-2m} \, \sum_{h=0}^2 c_{2h+1} \, \Phi_{2h+1} \, .
$$
This implies in turn that
$
\Lg \,\big[ \,  |z|^{2-2m}  - \, V_4 \,- \, V_5  \,\big] \, \, = \, \, \mathcal{O}(|z|^{2-2m}) \, .
$
Using the fact that in normal coordinates centered at $p$ the Euclidean bi-Laplacian operator $\Delta^2$ yields a good approximation of $\Lg$, it is not hard to construct a function $W\in C_{6-2m}^{4,\alpha} ( B_{r_{0}}^{*} )$ on a sufficiently small punctured ball $B_{r_0}^*$ centered at $p$, such that
$$
\Lg \,\big[ \,  |z|^{2-2m}  - \, V_4 \,- \, V_5  \, - \, W \,\big] \, \,  \in \,\,   C^{0,\alpha} ( B_{r_{0}}^{*} )    \, .
$$
By means of a smooth cut-off function $\chi$, compactly supported in $B_{r_0}$ and identically equal to $1$ in $B_{r_0/2}$, we obtain a globally defined function in $L^1(M)$, namely
$$
U_p \,\, = \,\,  \chi \, \bigg( \,   |z|^{2-2m}  - \, |z|^{4-2m} \, (\, C_2 \, \Phi_2 \, + \, C_4 \, \Phi_4 \,)  \, - \,  |z|^{5-2m} \, \sum_{h=0}^2 C_{2h+1} \, \Phi_{2h+1}  \, - \, W \, \bigg)
$$
In order to guarantee the orthogonality condition~\eqref{eq:orto}, we set
\begin{equation*}
G_{\Delta} (p,\cdot \,)  \,\,\, = \,\,\, U_{p} (\cdot) \,\, - \,\,  \frac{1}{\vol\st M \dt} \int_{M} U_{p} \, d\mu_{\omega} \,\, -\,\, \sum_{i=1}^{d} \, \varphi_{i}(\cdot)\int_{M} U_{p} \, \varphi_{i} \, d\mu_{\omega}
\end{equation*}
and we claim that $\Lg [G_{\Delta} (p,\cdot \,) ]$ satisfies the desired distributional identity. To see this, we set $M_{\varepsilon} \, = \, M\setminus B_{\varepsilon}$, where $B_\varepsilon$ is a ball of radius $\varepsilon$ centered at $p$, and we integrate $\Lg [G_{\Delta} (p,\cdot \,) ] \, = \, \Lg\, [U_p]$ against a test function $\phi \in C^{\infty}(M)$.
Setting
$$
\rho_{\omega}^{0} \,\, = \,\, \rho_{\omega} \, - \, \frac{s_{\omega}}{2m}\omega \, ,
$$
and using formula \eqref{eq:defLg}, it is convenient to write
\begin{equation*}
\Lg [U_p] \,\, = \,\, \Delta_{\omega}^{2} \, U_p\, + \, \frac{s_{\omega}}{m}\, \Delta_{\omega} \, U_p\, + \, 4\left<\, \rho_{\omega}^{0} \, | \, i\dd U_p \, \right> \,,
\end{equation*}
so that we have
\begin{align*}
\int_{M_{\varepsilon}}  \phi \,\, \Lg\sq U_p \dq \,d\mu_{\omega} \,\,\,  = \,\,\, &\int_{M_{\varepsilon}}  \phi \,\Big(\Delta_{\omega}^{2} \, + \, \frac{s_{\omega}}{m}\Delta_{\omega} \Big) \big[ U_p \big]\, d\mu_{\omega}
\, + \, 4 \int_{M_{\varepsilon}}  \phi \,\left<\, \rho_{\omega}^{0} \, | \, i\dd U_p \, \right>\, d\mu_{\omega}\,.
\end{align*}
We first integrate by parts the first summand on the right hand side and we take the limit for $\varepsilon \to 0$, obtaining
\begin{align*}
\lim_{\varepsilon\rightarrow 0} \int_{M_{\varepsilon}} \!\! \phi \,\Big(\Delta_{\omega}^{2} \, + \, \frac{s_{\omega}}{m}\Delta_{\omega} \Big) \big[ U_p \big]\, d\mu_{\omega}
\,\,\, = \,\,\, & \int_{M} \!\! U_p  \, \,\Big(\Delta_{\omega}^{2} \, + \, \frac{s_{\omega}}{m}\Delta_{\omega} \Big) \big[ \phi \big]\, d\mu_{\omega} \, + \, \lim_{\varepsilon\rightarrow 0}\int_{\partial M_{\varepsilon}} \!\!\! \phi \,\, \partial_{\nu}(\Delta_{\omega} U_p) \, d\sigma_{\omega}\\
&+\lim_{\varepsilon\rightarrow 0}\int_{\partial M_{\varepsilon}} \!\!\!\!   (\Delta_{\omega} \phi) \,\, \partial_{\nu} U_p \, d\sigma_{\omega}
\, + \, \frac{s_{\omega}}{m} \, \lim_{\varepsilon\rightarrow 0}\int_{\partial M_{\varepsilon}} \!\!\!\! \phi \,\, \partial_{\nu} U_p \, d\sigma_{\omega}
\end{align*}
where $d\sigma_{\omega}$ is the restriction of the measure $d\mu_{\omega}$ to $\partial M_{\varepsilon}$ and $\nu$ is the exterior unit normal to $\partial M_{\varepsilon}$. Combining the definition of $U_p$ with the standard development of the area element, it is easy to deduce that
$$
\lim_{\varepsilon\rightarrow 0}\int_{\partial M_{\varepsilon}} \!\!\!\!   (\Delta_{\omega} \phi) \,\, \partial_{\nu} U_p \, d\sigma_{\omega}
\, + \, \frac{s_{\omega}}{m} \, \lim_{\varepsilon\rightarrow 0}\int_{\partial M_{\varepsilon}} \!\!\!\! \phi \,\, \partial_{\nu} U_p \, d\sigma_{\omega} \,\,\, = \,\,\, \frac{2\, (m-1) \, |\Sp^{2m-1}|}{\left|\Gamma\right|} \, \Big[ \,  \Delta_{\omega} \phi \, ( p ) \, + \, \frac{s_{\omega}}{m} \, \phi \, (p) \,  \Big] \, .
$$
To treat the last boundary term, we use Proposition~\ref{proprietacsck} and we compute
\begin{align*}
\partial_{\nu}\, (\Delta_{\omega} U_p) \,\,\,
&= \,\,\, {\left|z\right|^{1-2m}}   \Big(  \, \frac{2s_{\omega} \st m-1 \dt^{3}}{m\st m+1 \dt } \, + \, K_2 \, \Phi_{2} \, + \, K_4 \, \Phi_{4} \, \Big) \, + \, \mathcal{O}\big( \left|z\right|^{2-2m} \big)\,,
\end{align*}
for suitable constants $K_2$ and $K_4$. Hence, we get
$$
\lim_{\varepsilon\rightarrow 0}\int_{\partial M_{\varepsilon}} \!\!\! \phi \,\,  \partial_{\nu}(\Delta_{\omega} U_p) \, d\sigma_{\omega} \,\,\, = \,\,\, \frac{2\, (m-1) \, |\Sp^{2m-1}|}{\left|\Gamma\right|}  \, \Big[ \frac{s_\omega (m-1)^2}{m(m+1)} \, \phi \, (p) \Big] \, .
$$
In conclusion we have that
\begin{align*}
\left<\, \Big(\Delta_{\omega}^{2} \, + \, \frac{s_{\omega}}{m}\Delta_{\omega} \Big) \big[ U_p \big]  \, \Big| \, \phi    \,\right>_{\mathscr{D}' \times \mathscr{D}} \,\, = \,\, &\int_{M} \!\! U_p  \, \,\Big(\Delta_{\omega}^{2} \, + \, \frac{s_{\omega}}{m}\Delta_{\omega} \Big) \big[ \phi \big]\, d\mu_{\omega} \\
& + \, \frac{2\, (m-1) \, |\Sp^{2m-1}|}{\left|\Gamma\right|} \, \Big[ \Delta_\omega \phi \,(p) \, + \, \frac{s_\omega (m^2-m+2)}{m(m+1)} \, \phi \, (p) \, \Big] \, .
\end{align*}

We now pass to consider the term contanining $\rho_{\omega}^{0}$. An integration by parts gives
\begin{align*}
\lim_{\varepsilon\rightarrow 0}\int_{M_{\varepsilon}}\!\!\! \phi \,  \left< \, \rho_{\omega}^{0} \, | \,  i\dd U_p \, \right> \, d\mu_{\omega}
\,\,\, = \,\,\, & \int_{M}   U_p \,\left< \, \rho_{\omega}^{0} \, | \, i\dd \phi \, \right>\,d\mu_{\omega} \\
& + \, \lim_{\varepsilon\rightarrow 0}  \int_{\partial M_{\varepsilon}} \!\!\! \phi \,\,
X(U_p) \lrcorner \,  d\mu_{\omega} \, + \, \lim_{\varepsilon\rightarrow 0}  \int_{\partial M_{\varepsilon}} \!\!\! U_p \,\,
\overline{X(\phi)} \lrcorner \,  d\mu_{\omega}  \, ,
\end{align*}
where, for a given function $u \in C^1(M_p)$, the vector field $X(u)$ is defined as $
X(u)  =  \big( \, \rho_\omega^0 (\partial^\sharp u \, , \, \cdot \, ) \,  \big)^\sharp $. It is easy to check that second boundary term vanishes in the limit. We claim that the same is true for the first boundary term. To prove this, we recall the expansions
\begin{eqnarray*}
\st\rho_{\omega}^{0}\dt_{i\bar{\jmath}} &=& \st \lambda_{i}\st p \dt- \frac{s_{\omega}}{2m} \dt\delta_{i\bar{\jmath}} \, + \, \mathcal{O}\st |z| \dt\,,\\
\partial^{\sharp} U_p &=&\sum_{i=1}^{m}   \big(  \st 1-m \dt \, {|z|^{-2m}}{z^{i}} \, + \, \mathcal{O}\st |z|^{2-2m}\dt  \big)  \, \frac{\partial}{\partial z^i} \\
d\mu_{\omega}&=&\st 1 +\mathcal{O}\st |z|^{2} \dt\dt \, d\mu_{0}\,,
\end{eqnarray*}
where the $\lambda_{i}$'s are the eigenvalues of the matrix $\st \rho_{\omega}^{0}\dt_{i\bar{\jmath}}$ and $d \mu_0$ is the Euclideam volume form. This implies
\begin{align*}
X(U_p) \,  \lrcorner \, d \mu_{\omega}
\,\,\,\, = \,\,\,\, &\st 1-m \dt  \, \sum_{i=1}^{m}\st \lambda_{i}\st p \dt- \frac{s_{\omega}}{2m} \dt  z^{i} \, \frac{\partial}{\partial z^{i}} \, \lrcorner \, d\mu_{0}  \, + \, \mathcal{O}\st |z| \dt\,.
\end{align*}
On the other hand, by the symmetry of $d \mu_0$, it is easy to deduce that
\begin{equation*}
\int_{\partial M_{\varepsilon}} z^{1} \, \frac{\partial}{\partial z^{1}} \, \lrcorner  \, d\mu_{0} \,\, = \,\, \ldots \,\, = \,\, \int_{\partial M_{\varepsilon}}z^{m} \, \frac{\partial}{\partial z^{m}} \, \lrcorner  \, d\mu_{0} \, .
\end{equation*}
The claim is now a straightforward consequence. In synthesis, we have obtained
\begin{align*}
\left<\, \Lg \big[ U_p \big]  \, \big| \, \phi   \,\right>_{\mathscr{D}' \times \mathscr{D}} \,\,\, = \,\,\, &\int_{M} \!\! U_p  \, \,\Lg \big[ \phi \big]\, d\mu_{\omega} \, + \, \frac{2\, (m-1) \, |\Sp^{2m-1}|}{\left|\Gamma\right|} \, \Big[ \,\Delta_\omega \phi \,(p) \,\, + \,\, \frac{s_\omega (m^2-m+2)}{m(m+1)} \, \phi \, (p) \, \Big] \,
\end{align*}
and the lemma is proven.
\end{proof}
%
%
%

Having at hand the above lemmata, we are now in the position to describe the local structure around the singular points of the {\em multi-poles fundamental solutions} $\GGG_{\aaa,\bbb,\ccc}$ constructed in Remark~\ref{gabc} through Proposition~\ref{balancrough}. For $m\geq 3$, it is sufficient to apply the operator $\Lg$ to the expression
\begin{eqnarray*}
\GGG_{\aaa,\bbb,\ccc} & + & \sum_{l=1}^K  \, \frac{a_l}{4(m-2)} \,\, \bigg[   \, \frac{|\Gamma_{N+l}|}{ 2 (m-1) |\Sp^{2m-1}|} \,\,  G_{\Delta\Delta}(q_l,\cdot)  \, \bigg]  \\
& + & \sum_{j=1}^N  \, \bigg(     \frac{c_j}{4(m-2)}  \, - \, \frac{s_\omega \, (m^2-m+2) \, b_j}{(m-2)m(m+1)}  \,   \bigg) \,\, \bigg[   \, \frac{|\Gamma_{j}|}{ 2 (m-1) |\Sp^{2m-1}|} \,\,  G_{\Delta\Delta}(p_j,\cdot)  \, \bigg]  \\
& - & \sum_{j=1}^N  \, b_j \,\, \bigg[   \, \frac{|\Gamma_{j}|}{ 2 (m-1) |\Sp^{2m-1}|} \,\, G_{\Delta}(p_j,\cdot)  \, \bigg]  \, ,
\end{eqnarray*}
to get a function in $C^{0,\alpha}(M)$. For $m=2$, one can obtain the same conclusion, applying the operator $\Lg$ to the expression 
\begin{eqnarray*}
\GGG_{\aaa,\bbb,\ccc} & - & \sum_{l=1}^K  \, \frac{a_l}{4} \,\, \bigg[   \, \frac{|\Gamma_{N+l}|}{ |\Sp^{3}|} \,\, G_{\Delta\Delta}(q_l,\cdot)  \, \bigg]  \\
& - & \sum_{j=1}^N  \, \bigg(     \frac{c_j}{4}  \, - \, \frac{s_\omega  \, b_j}{6}  \,   \bigg) \,\, \bigg[   \, \frac{|\Gamma_{j}|}{ |\Sp^{3}|} \,\, G_{\Delta\Delta}(p_j,\cdot)  \, \bigg]  \\
& - & \sum_{j=1}^N  \, b_j \,\, \bigg[   \, \frac{|\Gamma_{j}|}{ 2|\Sp^{3}|} \,\,  G_{\Delta}(p_j,\cdot) \, \bigg]  \, .
\end{eqnarray*}
Combining the previous observations with the standard elliptic regularity theory, we obtain the following proposition.
\begin{prop}
\label{loc_structure}
Let $(M,g,\omega)$ be a compact Kcsc orbifold of complex dimension $m \geq 2$, let ${\rm Ker}{(\Lg)} = {span}\{\varphi_0, \varphi_1, \dots,
\varphi_d\}$, as in~\eqref{nontrivial_ker} and let $\GGG_{\aaa,\bbb,\ccc}$ be as in Remark~\ref{gabc}. Then, we have that
$$
\GGG_{\aaa,\bbb,\ccc}  \,\,\, \in \,\,\,   \mathcal{C}^{\infty}_{loc}(M_{\p,\q}) \, .
$$
Moreover, if $z^1, \ldots , z^m$ are local coordinates centered at the singular points, then the following holds.
\begin{itemize}
\item If $m\geq 3$, then $\GGG_{\aaa,\bbb,\ccc}$  blows up like $|z|^{2-2m}$ at the points points of $p_1, \ldots, p_N$ and like $|z|^{4-2m}$ at the points $q_1, \ldots, q_K$.
\item If $m=2$, then $\GGG_{\aaa,\bbb,\ccc}$  blows up like $|z|^{-2}$ at the points $p_1, \ldots, p_N$ and like $\log\st |z|\dt$ at the points $q_1, \ldots, q_K$.
\end{itemize}
\end{prop}

%

%




%

\subsection{Solution of the linearized scalar curvature equation.}\label{solutionlinearized}

In this subsection, we are going to describe the possible choices for a right inverse of the operator $\Lg$, in a suitable functional setting. Since this operator is formally selfadjoint and since we are assuming that its kernel is nontrivial, we expect the presence of a nontrivial cokernel. To overcome this difficulty, we are going to consider some appropriate finite dimensional extensions of the natural domain of $\Lg$, which, according to Theorem~\ref{invertibilitapesatobase}, is given by $C^{4,\alpha}_\delta(M_{\p,\q})$, with $\delta \in (4-2m,0)$ if $m\geq 3$ and $\delta \in (0,1)$ if $m=2$. Building on the analysis of the previous section, we are going to introduce the following {\em deficiency spaces}. Given a triple of vectors $\boldsymbol\alpha \in \RR^K$ and $\bbbb, \cccc \in \RR^N$, we set, for $m \geq 3$, $l=1,\ldots,K$ and $j=1, \ldots, N$,
\begin{align}
W^l_{\aaaa}  = &
 -\frac{\alpha_l}{4(m-2)} \,\, \bigg[   \, \frac{|\Gamma_{N+l}|}{ 2 (m-1) |\Sp^{2m-1}|} \,\, G_{\Delta\Delta}(q_l,\cdot)  \, \bigg] \, , \\
&\\
W^j_{\bbbb,\cccc}   = &\,\,   \beta_j \,\, \bigg[   \, \frac{|\Gamma_{j}|}{ 2 (m-1) |\Sp^{2m-1}|} \,\, G_{\Delta}(p_j,\cdot)  \, \bigg] \\
&  - \,  \bigg(     \frac{\gamma_j}{4(m-2)}  \, - \, \frac{s_\omega \, (m^2-m+2) \, \beta_j}{(m-2)m(m+1)}  \,  \bigg) \,\, \bigg[   \, \frac{|\Gamma_{j}|}{ 2 (m-1) |\Sp^{2m-1}|} \,\,  G_{\Delta\Delta}(p_j,\cdot)  \, \bigg]  \, ,\label{eq:Wjbc}
\end{align}
whereas, for $m=2$, $l=1,\ldots,K$ and $j=1, \ldots, N$, we set
\begin{eqnarray*}
W^{l}_{\aaaa} &  = &  \, \alpha_{l} \,\, \bigg[   \, \frac{|\Gamma_{N+l}|}{ 4|\Sp^{3}|} \,\, G_{\Delta\Delta}(q_l,\cdot)  \, \bigg]  \, , \\
W^j_{\bbbb,\cccc}& = &  \beta_j \,\, \bigg[   \, \frac{|\Gamma_{j}|}{ |\Sp^{3}|} \,\,  G_{\Delta}(p_j,\cdot) \, \bigg] \, + \, \bigg(     \frac{\gamma_j}{4}  \, - \, \frac{s_\omega  \, \beta_j}{6}  \,   \bigg) \,\, \bigg[   \, \frac{|\Gamma_{j}|}{ |\Sp^{3}|} \,\, G_{\Delta\Delta}(p_j,\cdot)  \, \bigg]  \, .
\end{eqnarray*}
We are now in the position to define the {\em deficiency spaces}
\begin{eqnarray*}\label{deficiency1}
\mathcal{D}_{\q}(\aaaa) \,\, = \,\, \mbox{\em span} \,\Big\{ \, W^l_{\aaaa}\,  : \, {l=1,\ldots , K}  \, \Big\}  \quad & \hbox{and} & \quad
\mathcal{D}_{\p}(\bbbb, \cccc)  \,\,= \,\,  \mbox{\em span}\, \Big\{ \,
W^j_{\bbbb,\cccc} \,  :  \, {j=1, \ldots , N} \, \Big\} \, .
\end{eqnarray*}
These are finite dimensional vector spaces and they can be endowed with the following norm. If $V = \sum_{l=1}^K V^l \, W_{\aaaa}^l \in \mathcal{D}_\q(\aaaa)$ and $U = \sum_{j=1}^N U^j W_{\bbbb,\cccc}^j \in \mathcal{D}_\p(\bbbb, \cccc)$, we set
$$
\left\| V \right\|_{\mathcal{D}_\q(\aaaa)} \,\, = \,\, \sum_{l=1}^K\,  | V^l | \qquad \hbox{and} \qquad \left\| U \right\|_{\mathcal{D}_\p(\bbbb,\cccc)} \,\, = \,\, \sum_{j=1}^N\,  | U^j |\, .
$$
We will also make use of the shorthand notation $\mathcal{D}_{\p,\q}(\aaaa,\bbbb,\cccc)$ to indicate the direct sum $\mathcal{D}_{\q}(\aaaa) \oplus \mathcal{D}_{\p}(\bbbb,\cccc)$ of the {\em deficiency spaces} introduced above, endowed with the obvious norm $\left\| \, \cdot \, \right\|_{\mathcal{D}_\q(\aaaa)} +  \left\| \, \cdot \, \right\|_{\mathcal{D}_\p(\bbbb,\cccc)}$.

To treat the case $m=2$, it is convenient to introduce further finite dimensional extensions of the domain $C^{4,\alpha}_\delta(M_{\p,\q})$, with $\delta \in (0,1)$. These will be called {\em extra deficiency spaces} and they are defined as
\begin{eqnarray*}\label{deficiency2}
\mathcal{E}_{\q} \,\, = \,\, \mbox{\em span} \,\big\{ \, \chi_{q_l}\,  : \, {l=1,\ldots , K}  \, \big\}  \quad & \hbox{and} & \quad
\mathcal{E}_{\p}  \,\, = \,\, \mbox{\em span}\, \big\{ \,
\chi_{p_j} \,  :  \, {j=1, \ldots , N} \, \big\} \, ,
\end{eqnarray*}
where the functions $\chi_{p_1}, \ldots, \chi_{p_N},\chi_{q_1}, \ldots, \chi_{q_K}$ are smooth cutoff functions supported on small balls centered at the points $p_{1}, \ldots, p_N, q_1, \ldots, q_K$ and identically equal to $1$ in a neighborhood of these points. Given two functions $X = \sum_{j=1}^N X^j \chi_{p_j}\in \mathcal{E}_{\p}$ and $Y= \sum_{l=1}^K Y^l \chi_{q_l} \in \mathcal{E}_{\q}$, we set
$$
\left\| Y \right\|_{\mathcal{E}_\q} \,\, = \,\, \sum_{l=1}^K\, | Y^l | \qquad \hbox{and} \qquad \left\| X \right\|_{\mathcal{E}_\p} \,\, = \,\, \sum_{j=1}^N\,  | X^j |\, .
$$
We will also make use of the shorthand notation $\mathcal{E}_{\p,\q}$ to indicate the direct sum $\mathcal{E}_{\q} \oplus \mathcal{E}_{\p}$ of the {\em extra deficiency spaces} introduced above, endowed with the obvious norm $\left\| \, \cdot \, \right\|_{\mathcal{E}_\q} +  \left\| \, \cdot \, \right\|_{\mathcal{E}_\p}$. Notice that, with these notation, the estimate~\eqref{eq:stimapesatabase2} in Theorem~\ref{invertibilitapesatobase} reads
$$
|| \, \widetilde{u} \,||_{C^{4,\alpha}_\delta(M_{\p,\q}) }  \, + \, ||  \stackrel{\circ}{u} ||_{\mathcal{E}_{\p,\q}} \,\, \leq \,\,  C \, || \, f \, ||_{C^{0,\alpha}_{\delta-4}} \, ,
$$
where $u \, = \, \widetilde{u} \,+ \!\stackrel{\circ}{u} \,\, \in{{C}^{4,\alpha}_{\delta}(M_{\p ,\q})      \, \oplus \,     \mathcal{E}_{\p,\q} }$ and $f \in {C}^{0,\alpha}_{\delta-4}(M_{\p ,\q})$ are functions satisfying the equation $\Lg [u] = f$ as well as the orthogonality condition~\eqref{eq:orto} and $\delta \in (0,1)$.

{
\begin{remark} We notice {\em en passant} that a function $\GGG_{\aaa,\bbb,\ccc}$ constructed as in Remark~\ref{gabc} behaves like $W_\aaa^l$ near the point $q_l$, for $l=1, \ldots, K$ and like $W_{\bbb,\ccc}^j$, near the point $p_j$, for $j=1,\ldots, N$. In fact, it satisfies
\begin{equation*}
\Lg \Big[\, \GGG_{\aaa,\bbb,\ccc}  \, - \, \sum_{l=1}^K W_\aaa^l     \, - \, \sum_{j=1}^N W_{\bbb,\ccc}^j  \,\Big] \,\,\, \in \,\,\, C^{0,\alpha}(M) \, .
\end{equation*}
\end{remark}
}

We recall that we have assumed that the bounded kernel of $\Lg$ is $(d+1)$-dimensional and that it is spanned by $\{\varphi_0, \varphi_1, \ldots , \varphi_d \}$, where $\varphi_0 \equiv 1$ and $\varphi_1, \ldots, \varphi_d$, with $d\geq 1$, is a collection of mutually $L^2(M)$-orthogonal smooth functions with zero mean and $L^2(M)$-norm equal to $1$. Given a triple of vectors $\boldsymbol\alpha \in \RR^K$ and $\bbbb, \cccc \in \RR^N$, it is convenient to introduce the following matrices
\begin{align}
\Xi_{il}(\boldsymbol{\alpha} ) \,\, := & \,\,\,\, \alpha_{l} \, \varphi_{i}(q_{l}) \,,   & \hbox{for} \quad  i=1 \dots, d \quad \hbox{and} \quad  l=1, \dots, K \, ,
\label{eq:nondeggen2}\\
\Theta_{ij}(\boldsymbol{\beta},\boldsymbol{\gamma} ) \,:=&\,\,\,  \beta_{j} \, \Delta\varphi_{i}(p_{j})  \, + \,
\gamma_{j} \, \varphi_{i}(p_{j}) \, , &  \,\hbox{for} \quad i=1 \dots, d \quad \hbox{and} \quad  j=1, \dots, N \, .
\label{eq:nondeggen}
\end{align}
These will help us in formulating our {\em nondegeneracy assumption}. We are now in the position to state the main results of our linear analysis on the base obifold.

\begin{teo}
\label{invertibilitapesatodef}
Let $(M,g,\omega)$ be a compact Kcsc orbifold of complex dimension $m \geq 2$ and let ${\rm Ker}{(\Lg)} = {span}\{\varphi_0, \varphi_1, \dots, \varphi_d\}$. Assume that the following {\em nondegeneracy condition} is satisfied: a triple of vectors
$\boldsymbol\alpha \in \RR^K$ and $\bbbb, \cccc \in \RR^N$ is given such that the $d \times (N+K)$ matrix
\begin{equation*}
\st\left.  \st\Xi_{il}(\aaaa)\dt_{\substack{ 1 \leq i\leq d\\ 1 \leq l \leq K }}   \,\, \right|\,\, \st\Theta_{ij}(\bbbb,\cccc)\dt_{\substack{ 1 \leq i\leq d\\ 1 \leq j \leq N }}  \dt
\end{equation*}
has full rank. Then, the following holds.
\begin{itemize}
\item If $m \geq 3$, then for every $f\in{C}^{0,\alpha}_{\delta-4}(M_{\p ,\q})$ with $\delta\in (4-2m,0)$, there exist real number $\nu$ and a function
$$
u\, = \, \widetilde{u} \, + \, \widehat{u}\, \, \in \, {C}^{4,\alpha}_{\delta}(M_{\p ,\q}) \, \oplus \, \mathcal{D}_{\p,\q}(\aaaa,\bbbb, \cccc)
$$
such that
\begin{equation}
\label{eq:linpro}
\Lg u \, + \, \nu \,\,  = \,\,  f\, , \qquad \hbox{in} \quad {M_{\p,\q}} \,.
\end{equation}
Moreover, there exists a positive constant $ C = C(\aaaa,\bbbb,\cccc,\delta )>0$ such that
\begin{equation}
|\,\nu\,| \,\, + \,\, ||\,\widetilde{u} \,||_{{C}^{4,\alpha}_{\delta}(M_{\p ,\q})  }
\, + \, ||\,\widehat{u} \,||_{  \mathcal{D}_{\p,\q}(\aaaa,\bbbb, \cccc) }
 \,\,\, \leq \,\,\, C \, || \, f \,   ||_{ \mathcal{C}^{0,\alpha}_{\delta-4}(M_{\p ,\q}) } \, .
\end{equation}
\item
If $m =2$, then for every $f\in {C}^{0,\alpha}_{\delta-4}(M_{\p ,\q})$ with $\delta\in (0,1)$, there exist real number $\nu$ and a function
$$
u \, = \, \widetilde{u} \,+ \,
\stackrel{\circ}{u}
 \,+ \, \widehat{u}\, \,\in \,\, {C}^{4,\alpha}_{\delta}(M_{\p ,\q}) \, \oplus \mathcal{E}_{\p,\q} \, \oplus \, \mathcal{D}_{\p,\q}(\aaaa,\bbbb, \cccc)
$$
such that
\begin{equation}
\label{eq:linpro2}
\Lg u \, + \, \nu \,\,  = \,\,  f\, , \qquad \hbox{in} \quad {M_{\p,\q}} \,.
\end{equation}
Moreover, there exists a positive constant $ C = C(\aaaa,\bbbb,\cccc,\delta )>0$ such that
\begin{equation}
|\,\nu\,| \, \, + \,\, ||\,\widetilde{u} \,||_{ {C}^{4,\alpha}_{\delta}(M_{\p ,\q}) }  \,+ \,
||\stackrel{\circ}{u} ||_{ \mathcal{E}_{\p,\q} }   \, + \,
||\,\widehat{u} \,||_{ \mathcal{D}_{\p,\q}(\aaaa,\bbbb, \cccc) }
\,\,\, \leq \,\,\, C \, || \, f \,   ||_{ {C}^{0,\alpha}_{\delta-4}(M_{\p ,\q}) }
\end{equation}
\end{itemize}
\end{teo}
\begin{proof} We only prove the statement in the case $m\geq 3$, since it is completely analogous in the other case. For sake of simplicity we assume $\aaaa= \mathbf{0} \in \RR^K$, so that the nondegeneracy condition becomes equivalent to the requirement that the matrix
$$
\st\Theta_{ij}(\bbbb,\cccc)\dt_{\substack{ 1 \leq i\leq d\\ 1 \leq j \leq N }}
$$
has full rank. Under these assumptions, the {\em deficiency space} $\mathcal{D}_{\p,\q}(\aaaa,\bbbb,\cccc)$ reduces to $\mathcal{D}_\p(\bbbb,\cccc)$. In order to split our problem, it is convenient to set
\begin{equation}
f^{\perp} \,\, = \,\, f \, - \, \frac{1}{\vol\st M \dt}\int_{M} f \,d\mu_{\omega} \, - \, \sum_{i=1}^{d}\varphi_{i}\int_{M} f\varphi_{i}\,d\mu_{\omega}\,,
\end{equation}
so that $f^\perp$ satisfies the orthogonality conditions~\eqref{eq:orto}. By Theorem~\ref{invertibilitapesatobase}, we obtain the existence of a function $u^\perp \in C^{4,\alpha}_\delta(M_{\p,\q})$, which satisfies the equation
\begin{equation*}
\Lg \, [ \, u^\perp ] \,\, = \,\, f^\perp \, ,
\end{equation*}
together with the orthogonality conditions~\eqref{eq:orto} and the desired estimate~\eqref{eq:stimapesatabase}. To complete the resolution of equation~\eqref{eq:linpro}, we set
$$
f_0 \,\, = \,\, \frac{1}{\vol\st M \dt}\int_{M}\!\! f \,d\mu_{\omega} \qquad\quad \hbox{and} \quad\qquad f_i \,\, = \,\, \int_{M}\!\! f\varphi_{i}\,d\mu_{\omega} \, , \qquad \hbox{for $i=1,\ldots, d$} \, .
$$
Recalling the definition of $\Theta_{ij}(\bbbb,\cccc)$ and using the {\em nondegeneracy condition}, we select a solution $(\nu, U_1, \ldots, U_N) \in \RR^{N+1}$ to the following system of {\em linear balancing conditions}
\begin{eqnarray}
f_i \,\,  +\,\,  \sum_{j=1}^N \, U^j \, \big[ \, \beta_j \, (\Delta\varphi_i)(p_j) \, + \,  \gamma_j \, \varphi_i(p_j)  \, \big]  &  = &  0 \,, \quad\quad\qquad\qquad\qquad \hbox{$i = 1, \dots, d$}  , \\
f_0 \, {\rm Vol}_\omega(M)   \,\,  + \,\, \sum_{j=1}^N \, U^j \,\gamma_j  & = & \nu  \, {\rm Vol}_\omega(M) \, .
\end{eqnarray}
It is worth pointing out that in general this choice is not unique, since it depends in the choice of a right inverse for the matrix $\Theta_{ij}(\bbbb,\cccc)$. Theorem~\ref{balancrough} implies then the existence of a distribution $U \in \mathscr{D}'(M)$ which satisfies
\begin{equation*}
\Lg   [{U}]  \,  + \, \nu  \,\,\,  = \,\,\,   \sum_{i=0}^d f_i \, \varphi_i \, + \, \sum_{j=1}^N \, U^j \beta_j \, \Delta\delta_{p_{j}} \, + \, \sum_{j=1}^N U^j  \gamma_j \, \delta_{p_{j}} \, , \qquad \hbox{in \,\,\,$M$}\, .
\end{equation*}
Arguing as in Proposition~\ref{loc_structure}, it is not hard to show that $U \in C^{\infty}_{loc}(M_\p)$. In particular the function $u^\perp + U \in C^{4,\alpha}_{loc}(M_\p)$ satisfies the equation
\begin{equation*}
\Lg   [u^\perp + U]  \,  + \, \nu  \,\,\,  = \,\,\, f  \, , \qquad \hbox{in \,\,\,$M_{\p}$}\, .
\end{equation*}
To complete the proof of our statement, we need to describe the local structure of $U$ in more details. First, we observe that, by the very definition of the deficiency spaces, one has
$$
\Lg \, \big[  W^j_{\bbbb,\cccc} \big] \,\, = \,\, \beta_j \, \Delta \delta_{p_j} \, + \, \gamma_j \, \delta_{p_j} \, + \, V^j_{\bbbb,\cccc} \, ,
$$
where, for every $j=1, \ldots, N$, the function $V^j_{\bbbb,\cccc}$ is in $C^{\infty}(M)$. Combining this fact with the linear balancing conditions, we deduce that
\begin{eqnarray*}
\Lg \, \Big[ \, U \, - \, \sum_{j=1}^N \, U^j  \, W^j_{\bbbb,\cccc}\Big] & = &  f_0 \, - \,  \nu  \, + \, \sum_{i=1}^d f_i \, \phi_i \, - \, \sum_{j=1}^N \, U^j \, V_{\bbbb,\cccc}^j \\
& = &   \frac{1}{\vol(M)} \, \sum_{j=1}^N \, U^j \, \gamma_j \, - \, \sum_{i=1}^d \, \sum_{j=1}^N  \, U^j \, \Theta_{ij}(\bbbb,\cccc) \, \phi_i \, - \, \sum_{j=1}^N \, U^j \, V_{\bbbb,\cccc}^j  \, .
\end{eqnarray*}
By the definition of $V^j_{\bbbb,\cccc}$ it follows that
$$
\int_M V^j_{\bbbb,\cccc} \, \phi_0 \,\, d\mu_\omega \,\, = \,\, -\, \gamma_j \qquad \quad \hbox{and} \qquad \quad \int_M V^j_{\bbbb,\cccc} \, \phi_i\,\, d\mu_\omega \,\, = \,\, - \, \Theta_{ij}(\bbbb,\cccc)
$$
and thus, it is easy to check the right hand side of the equation above is orthogonal to $\ker(\Lg)$. Hence, using Theorem~\ref{invertibilitapesatobase} and by the elliptic regularity, we deduce the existence of a smooth function $\overline u \in C^{\infty}(M)$ which satisfies
$$
\Lg \, [\, \overline{u} \,] \,\,\, = \,\,\, \frac{1}{\vol(M)} \, \sum_{j=1}^N \, U^j \, \gamma_j \, - \, \sum_{i=1}^d \, \sum_{j=1}^N  \, U^j \, \Theta_{ij}(\bbbb,\cccc) \, \phi_i \, - \, \sum_{j=1}^N \, U^j \, V_{\bbbb,\cccc}^j \,, \qquad \hbox{in $M$.}
$$
Setting $\widehat{u} = \sum_{j=1}^N \, U^j  \, W^j_{\bbbb,\cccc}$, we have obtained that $\Lg\, [\,U\,] \, = \, \Lg \, [\,\widehat{u} \,+ \, \overline{u}\, ]$, hence
\begin{equation*}
\Lg   [\, u^\perp \! + \,\overline{u}  + \, \widehat{u}\,]  \, + \, \nu  \,\,\,  = \,\,\, f  \, , \qquad \hbox{in \,\,\,$M_{\p}$}\, ,
\end{equation*}
with $\widetilde{u} = (u^\perp +\, \overline{u}) \, \in \, C^{4,\alpha}_\delta(M_\p)$ and $\widehat{u} \in \mathcal{D}_\p(\bbbb,\cccc)$. Moreover, combining the estimate~\eqref{eq:stimapesatabase} with our construction, it is clear that, for suitable positive constants $C_0,\ldots, C_3$, possibly depending on $\bbbb,\cccc$ and $\delta$, it holds
\begin{eqnarray*}
|| \,u \, ||_{C^{4,\alpha}_{\delta}(M_\p) \oplus \mathcal{D}_\p(\bbbb,\cccc)} \!\!\!\! &= & \!\!\!\! || \,\widetilde{u} \, ||_{C^{4,\alpha}_{\delta}(M_\p)}    + ||\,\widehat{u} \, ||_{\mathcal{D}_\p(\bbbb,\cccc)}  \, \leq \,   || \, u^\perp ||_{C^{4,\alpha}_{\delta}(M_\p) }   +  || \, \overline{u} \,||_{C^{4,\alpha}_{\delta}(M_\p)}  +  || \, \widehat{u} \, ||_{ \mathcal{D}_\p(\bbbb,\cccc)} \phantom{\sum_j=1^N}\\
& \leq & \!\!\!\! C_0 \, || \, f^\perp||_{C^{0,\alpha}_{\delta-4}(M_\p)}  +  C_1  \, \sum_{j=1}^N \, |U^j| \,  \leq \, C_2 \, \Big( \, ||\, f^\perp \, ||_{C^{0,\alpha}_{\delta-4}(M_\p)}  \, + \,    \, \sum_{i=1}^d \, |f_i| \, \,  \Big) \\
& \leq & \!\!\!\!  C_3\, || \, f\,||_{C^{0,\alpha}_{\delta-4}(M_\p)}\,, \phantom{\bigg(\sum_j=1^N\bigg)}
\end{eqnarray*}
which is the desired estimate. Finally, we observe that the constant $\nu$ as well can be easily estimated in terms of the norm of $f$. This concludes the proof of the theorem.
\end{proof}
\begin{remark}\label{inversadef}
In other words, with the notations introduced in the proof of the previous theorem, we have proven that, for $m\geq 3$ and $\delta\in(4-2m,0)$, the operator
\begin{eqnarray*}
\mathbb{L}^{(\delta)}_{\aaaa,\bbbb,\cccc} \,\, : \, \, {C}^{4,\alpha}_{\delta}(M_{\p ,\q}) \, \oplus \, \mathcal{D}_{\p,\q}(\aaaa,\bbbb, \cccc)   \, \times \, \RR \!\!& \longrightarrow &\!\! {C}^{0,\alpha}_{\delta-4}(M_{\p ,\q}) \\
(\, \widetilde{u}\, +\, \widehat{u} \,\,,\,\, \nu\, ) \!\!& \longmapsto &\!\! \Lg \, [\,\widetilde{u}\, +\, \widehat{u} \,] \, + \, \nu \, ,
\end{eqnarray*}
with $\bbbb,\cccc$ and $\aaaa$ satisfying the {\em nondegeneracy condition}, admits a (in general not unique) bounded right inverse
\begin{eqnarray*}
\mathbb{J}^{(\delta)}_{\aaaa,\bbbb,\cccc} \,\, : \, \, {C}^{0,\alpha}_{\delta-4}(M_{\p ,\q})  \!\!& \longrightarrow &\!\!
{C}^{4,\alpha}_{\delta}(M_{\p ,\q}) \, \oplus \, \mathcal{D}_{\p,\q}(\aaaa,\bbbb, \cccc)   \, \times \, \RR \, ,
\end{eqnarray*}
so that $ \big( \,\mathbb{L}^{(\delta)}_{\aaaa,\bbbb,\cccc} \circ  \mathbb{J}^{(\delta)}_{\aaaa,\bbbb,\cccc}\, \big) \, [\, f \,] \, = \, f$, for every $f \in {C}^{0,\alpha}_{\delta-4}(M_{\p ,\q})$ and
$$
\big\|\,      \mathbb{J}^{(\delta)}_{\aaaa,\bbbb,\cccc}  \, [\, f \,]   \,  \big\|_{{C}^{4,\alpha}_{\delta}(M_{\p ,\q}) \, \oplus \, \mathcal{D}_{\p,\q}(\aaaa,\bbbb, \cccc)   \, \times \, \RR } \,\,\, \leq \,\,\, C \,\, || \, f \, ||_{{C}^{0,\alpha}_{\delta-4}(M_{\p ,\q}) } \, .
$$
Of course, the analogous conclusion holds in the case $m=2$.
\end{remark}

\section{Linear Analysis on ALE manifolds}\label{lineareALE} 
We now reproduce an analysis similar to the one just completed on the base orbifold on our 
model ALE resolutions of isolated singularities. We define also in this setting  weighted H\"older spaces. Since we will use duality arguments we introduce   also  weighted Sobolev spaces. Let $\st X_{\Gamma},h,\eta \dt$ be an $ALE$ K\"ahler resolution of isolated singularity and set 
\begin{equation}X_{\Gamma,R_{0}}=\pi^{-1}\st B_{R_{0}} \dt\,.\end{equation}
where $\pi:X_{\Gamma}\rightarrow \CC^{m}/\Gamma$ is the canonical projection. This can be thought as the counterpart in $X_{\Gamma}$ of  $M_{r_{0}}$ in $M$. For $\delta\in \RR$ and $\alpha\in (0,1)$, the weighted H\"older space $C_{\delta}^{k,\alpha}\st X_{\Gamma} \dt$ is the set of functions $f\in C_{loc}^{k,\alpha}(X_{\Gamma})$ such that
\begin{equation}
\left\|f\right\|_{C_{\delta}^{k,\alpha}\st X_{\Gamma} \dt}:=\left\|f\right\|_{C^{k,\alpha}\st X_{\Gamma,R_{0}} \dt}+\sup_{R\geq R_{0}}R^{-\delta}\left\|f\st R\cdot \dt\right\|_{C^{k,\alpha}\st B_{1}\setminus B_{1/2} \dt}<+\infty\,.
\end{equation} \label{eq:weightedHolderALE} 
In order to define weighted Sobolev spaces we have to introduce a distance-like function $\gamma\in C_{loc}^{\infty}\st X_{\Gamma} \dt$ defined as 
\begin{equation}
\gamma\st p \dt:=\chi\st p \dt + \st 1-\chi\st p \dt \dt |x\st p \dt|\qquad p\in X_{\Gamma}
\end{equation}
with $\chi$ a smooth cutoff function identically $1$ on $X_{\Gamma,R_{0}}$ and identically $0$ on $X_{\Gamma}\setminus X_{\Gamma,2R_{0}}$.
For $\delta\in \RR$, the weighted Sobolev space $W_{\delta}^{k,2}\st X_{\Gamma} \dt$ is the set of functions $ f\in L_{loc}^{1}(X_{\Gamma})$ such that
\begin{equation}\label{eq:weightedSobolevALE}
\left\|f\right\|_{W_{\delta}^{k,2}\st X_{\Gamma} \dt}:=\sqrt{\sum_{j=0}^{k}\int_{X}\left|\gamma^{-\delta-m+j}\nabla^{(j)}f\right|_{\eta}^{2}\,d\mu_{\eta}}<+\infty\,
\end{equation}
 where 
\begin{equation}
\nabla^{(j)}f:=\underbrace{\nabla \circ \cdots \circ \nabla}_{j\textrm{ times }} f\,.
\end{equation}

We recall now the natural duality between weighted spaces 
\begin{equation}\left<\cdot|\cdot\right>_{\eta}\,:\, L_{\delta}^{2}\st X_{\Gamma} \dt\times L_{-2m-\delta}^{2}\st X_{\Gamma} \dt\rightarrow \RR\end{equation} 
defined as
\begin{equation}\left<f|g\right>_{\eta}:=\int_{X}f\,g\,d\mu_{\eta}\,.\label{eq: dualita}\end{equation}

\begin{remark}
 We note that a function $f\in W_{\delta}^{k,2}\st X_{\Gamma}\dt\cap C_{loc}^{\infty}\st X_{\Gamma} \dt$ on the set $X_{\Gamma}\setminus X_{\Gamma,R_{0}}$ beheaves like
\begin{equation}
f|_{X_{\Gamma}\setminus X_{\Gamma,R_{0}}}\st p \dt=\mathcal{O}\st |x\st p \dt|^{\delta'} \dt \qquad\textrm{ for dome } \delta'<\delta\,.
\end{equation}
and a function $f\in C^{k\,\alpha}\st X_{\Gamma}\dt$ on the set $X\setminus X_{\Gamma,R_{0}}$ typically beheaves like
\begin{equation}
f|_{X_{\Gamma}\setminus X_{\Gamma,R_{0}}}\st p \dt=\mathcal{O}\st |x\st p \dt|^{\delta} \dt\,.
\end{equation}
We also note that for every $\delta'<\delta$ we have the inclusion
\begin{equation}
C_{\delta}^{k,\alpha}\st X_{\Gamma} \dt\subseteq W_{\delta'}^{k,2}\st X_{\Gamma} \dt\,.
\end{equation}
\end{remark}

\noindent The main task of this section is to solve the linearized constant scalar curvature equation
\begin{equation}
\Le u=f\,. \label{eq: Leta}
\end{equation}
We recall that by \eqref{eq:defLg}
\begin{equation}
\Le u=\Delta_{\eta}^{2}u+4\left< \rho_{\eta} | i\dd u\right>
\end{equation}
 and, since $\st X_{\Gamma}, h,\eta \dt$ is scalar flat,  $\Le$ is formally self-adjoint. We also notice that if $\st X_{\Gamma}, h,\eta \dt$ is Ricci-flat, the operator $\Le$ reduces to the $\eta$ bi-Laplacian operator. Since we want to study the operator $\Le$ on weighted spaces we have to be careful on the choice of weights. Indeed to have Fredholm properties we must avoid the indicial roots at infinity of $\Le$ that, thanks to the decay of the metric, coincide with those of euclidean bi-Laplace operator $\Delta^{2}$ . We recall that the set of indicial roots at infinity for $\Delta^{2}$ on $\CC^{m}$ is $\ZZ\setminus \sg 5-2m, \ldots ,-1 \dg$ for $m\geq 3$ and $\ZZ$ for $m=2$.  Let $\delta\in \RR$ with
\begin{equation}
\delta\notin \ZZ\setminus\sg  5-2m,\ldots, -1 \dg \,.
\end{equation}
for $m\geq 3 $ and $\delta\notin \ZZ$ for $m=2$, then the operator
\begin{equation}
\mathbb{L}_{\eta}^{(\delta)}:W_{\delta}^{4,2}\st X_{\Gamma} \dt\rightarrow L_{\delta-4}^{2}\st X_{\Gamma} \dt\,.
\end{equation}
is Fredholm and its cokernel is the kernel of its adjoint under duality \eqref{eq: dualita}
\begin{equation}
\mathbb{L}_{\eta}^{(-2m-\delta)}:W_{-2m-\delta}^{4,2}\st X_{\Gamma} \dt\rightarrow L_{-2m-4-\delta}^{2}\st X_{\Gamma} \dt\,.
\end{equation}
For $ALE$ K\"ahler manifolds a result analogous to Proposition \ref{invertibilitapesatobase} holds true.

\begin{prop}\label{isomorfismopesati}
Let $(X_{\Gamma},h,\eta)$ a scalar flat $ALE$ K\"ahler resolution. If $m\geq 3$  and $\delta\in (4-2m,0)$, then 
\begin{equation}
\mathbb{L}_{\eta}^{(\delta)} : C_{\delta}^{4,\alpha}\st X_{\Gamma} \dt\longrightarrow C_{\delta-4}^{0,\alpha}\st X_{\Gamma} \dt
\end{equation}
 is invertible. If $m=2$ and $\delta\in (0,1)$, then
\begin{equation}
\mathbb{L}_{\eta}^{(\delta)} : C_{\delta}^{4,\alpha}\st X_{\Gamma} \dt\longrightarrow C_{\delta-4}^{0,\alpha}\st X_{\Gamma} \dt
\end{equation}
 is surjective with one dimensional kernel spanned by the constant function.
\end{prop}

\begin{remark}\label{inversamod}
Rephrasing Proposition \ref{isomorfismopesati} we can say that for $\delta\in (4-2m)$ if $m\geq 3$ and $\delta\in (0,1)$ if $m=2$ the operator   
\begin{equation}
\mathbb{L}_{\eta}^{(\delta)} : C_{\delta}^{4,\alpha}\st X_{\Gamma} \dt\longrightarrow C_{\delta-4}^{0,\alpha}\st X_{\Gamma} \dt
\end{equation}
has a continuous right inverse
\begin{equation}\label{eq:inversamod}
\mathbb{J}^{(\delta)} :   C_{\delta-4}^{0,\alpha}\st X_{\Gamma} \dt \longrightarrow    C_{\delta}^{4,\alpha}\st X_{\Gamma} \dt\,.
\end{equation}
\end{remark}

\noindent The proof of the above result follows standard lines (see e.g. Theorem 10.2.1 and Proposition 11.1.1 in \cite{Pacard-notes}). 
We focus now on the asymptotic expansions of various operators on $ALE$ manifolds.

\begin{lemma}\label{espansioniALE}
Let $\st X_{\Gamma},h,\eta \dt$ be a scalar flat $ALE$-K\"ahler resolution with $e\st \Gamma \dt=0$. Then on the coordinate chart at infinity we have the following expansions  
\begin{itemize}
\item for the inverse of the metric $\eta^{i\bar{\jmath}}$
\begin{equation}\eta^{i\bar{\jmath}}=2\sq \delta^{i\bar{\jmath}}-\frac{2 \cga \st m-1 \dt}{\left|x\right|^{2m}} \st \delta_{i\bar{\jmath}}-m\frac{\overline{x^{i}}x^{j}}{\left|x\right|^{2}}  \dt+\mathcal{O}\st \left|x\right|^{-2-2m} \dt \dq\,;\label{eq:inversaeta} \end{equation}

\item for the unit normal vector to the sphere $\left|x\right|=\rho$ 
\begin{equation}\nu= \frac{1}{|x|}\st  x^{i}\frac{\partial}{\partial x^{i}}+\overline{x^{i}}\frac{\partial }{\partial \overline{x^{i}}}\dt    \sq 1+\frac{ \cga \st m-1 \dt^{2}}{|x|^{2m}} +\mathcal{O}\st \left|x\right|^{-2-2m} \dt\dq\,;\label{eq:espnormaleALE}\end{equation}

\item for the laplacian $\Delta_{\eta}$
\begin{equation}\Delta_{\eta}=\sq1-\frac{2 \cga \st m-1 \dt}{\left|x\right|^{2m}}\dq\Delta+\sq\frac{8 \cga \st m-1 \dt m}{\left|x\right|^{2m+2}}\overline{x^{i}}x^{j}+\mathcal{O}\st \left|x\right|^{-2-2m} \dt \dq\partial_{j}{\partial}_{\bar{\imath}}\,.\label{eq:esplapALE}\end{equation}
\end{itemize}
\end{lemma}

\noindent The proof of the above lemma consists of straightforward computations and is therefore omitted.

\noindent We conclude this section with an observation regarding fine mapping properties of

\begin{equation}
\mathbb{L}_{\eta}^{(\delta)}: W_{\delta}^{4,2}\st X_{\Gamma} \dt\rightarrow L_{\delta-4}^{2}\st X_{\Gamma} \dt
\end{equation}
that will  be useful in Subsection \ref{analisinonlinearemodello} in a crucial point where we show how the nonlinear analysis constrains the choice of balancing parameters.
In the following proposition we want to solve the equation
\begin{equation}
\Le\sq u \dq=f
\end{equation}
with $f\in L_{\delta-4}^{2}\st X_{\Gamma} \dt$ ($C_{\delta-4}^{0,\alpha}\st X_{\Gamma} \dt$). In general, when $\delta\in (2-2m,4-2m)$, the indicial root $3-2m$ imposes to the solution $u$ to have a component with asymptotic growth $|x|^{3-2m}$. The keypoint of Proposition is that if $\Gamma$ is non trivial this doesn't occur.

\begin{prop}\label{GAP}
Let $\st X_{\Gamma},h,\eta \dt$ be a scalar-flat  ALE K\"ahler resolution with $e\st \Gamma\dt=0$ and nontrivial $\Gamma\triangleleft U(m)$. For $\delta\in\st  2-2m , 4-2m \dt$, the equation
\begin{equation}
\Le\sq u\dq= f
\end{equation} 
with $f\in L_{\delta-4}^{2}\st X_{\Gamma} \dt$ (respectively $f\in C_{\delta-4}^{0,\alpha}\st X_{\Gamma} \dt$)   is solvable for $u\in W_{\delta}^{4,2}\st X_{\Gamma} \dt$ (respectively $u\in C_{\delta}^{4,\alpha}\st X_{\Gamma} \dt$) if and only if 
\begin{equation}
\int_{X_{\Gamma}}f\,d\mu_{\eta}=0\,.
\end{equation}
\end{prop}

\begin{proof} We are going to prove  the following characterization:
\begin{equation}\label{eq:caratt}
\mathbb{L}_{\eta}^{(\delta)}\sq W_{\delta}^{4,2}\st X_{\Gamma} \dt\dq=\sg f\in L_{\delta}^{2}\st X_{\Gamma} \dt\,| \, \int_{X_{\Gamma}}f\,d\mu_{\eta}=0 \dg\,.
\end{equation}
\noindent Since $\Le$ is formally selfadjoint  we can identify, via duality \eqref{eq: dualita}, the cokernel of 
\begin{equation}
\mathbb{L}_{\eta}^{(\delta)}: W_{\delta}^{4,2}\st X_{\Gamma} \dt\rightarrow L_{\delta-4}^{2}\st X_{\Gamma} \dt \qquad \delta\in \st2 -2m,4-2m \dt
\end{equation} 
with the kernel of 
\begin{equation}
\mathbb{L}_{\eta}^{(-2m-\delta)}: W_{-2m-\delta}^{4,2}\st X_{\Gamma} \dt\rightarrow L_{-2m-4-\delta}^{2}\st X_{\Gamma} \dt\,.
\end{equation}
We want to identify generators of this kernel. Let then $u\in W_{\delta}^{4,2}\st X_{\Gamma} \dt$ such that
\begin{equation}
\Le\sq u\dq=0\,,
\end{equation}
\noindent with $\delta\in (0,2)$,we want ot prove that $u\equiv c_{0}$ for some $c_{0}\in \RR$.  By standard elliptic regularity we have that $u\in C_{loc}^{\omega}\st X_{\Gamma} \dt$. On $X_{\Gamma}\setminus  X_{\Gamma,R}$ we consider the Fourier expansion of $u$  
\begin{equation}u=\sum_{k=0}^{+\infty} u^{(k)}\st |x| \dt \phi_{k}\,,\end{equation}     
with $u^{(k)}\in C_{\delta}^{n,\alpha}\st [R,+\infty) \dt$ for any $n\in \NN$ and this sum is $C^{n,\alpha}$-convergent on compact sets. Then, using expansions\eqref{eq:inversaeta}, \eqref{eq:espnormaleALE},\eqref{eq:esplapALE}, we have on $X_{\Gamma}\setminus X_{\Gamma,R}$
\begin{equation}
0=\,\Delta_{\eta}^{2}\sq u\dq =\, \sum_{k=0}^{+\infty}\Delta^{2}\sq u^{(k)}\st |x| \dt  \phi_{k}\dq +|x|^{-2m}L_{4}\sq u \dq+|x|^{-1-2m}L_{3}\sq u \dq+|x|^{-2-2m}L_{2}\sq u \dq\,.
\end{equation}

\noindent where the  $L_{k}$'s are differential operators of order $k$ and uniformly bounded coefficients.  The equation 
\begin{equation}\sum_{k=0}^{+\infty} \Delta^{2}\sq u^{(k)}\st |x| \dt \phi_{k} \dq =-|x|^{-2m}L_{4}\sq u \dq-|x|^{-1-2m}L_{3}\sq u \dq-|x|^{-2-2m}L_{2}\sq u \dq\end{equation}

\noindent implies  
\begin{equation}\Delta^{2}\sq u^{(k)} \phi_{k} \dq \in C_{\delta-2m-4}^{n,\alpha}\st X_{\Gamma}\setminus X_{\Gamma,R} \dt\qquad \textrm{ for }  k\geq 0\,.\end{equation}
\noindent Suppose  by contradiction that 
\begin{equation}\limsup_{|x|\rightarrow +\infty}|u|>0\,.\end{equation}
\noindent Since $u^{(k)}\phi_{k}\in C_{\delta}^{n,\alpha}\st X_{\Gamma}\setminus X_{\Gamma,R} \dt$ the only possibilities are
\begin{align} 
u^{(0)}\st |x| \dt=& c_{0}+\upsilon_{0}\st |x|\dt\\ 
u^{(1)}\st |x| \dt=& \st |x| +\upsilon_{1}\st |x| \dt\dt \phi_{1}
\end{align}
\noindent with $\upsilon_{0},\upsilon_{1}\in C_{\delta-2m}^{n,\alpha}\st [R,+\infty) \dt$ and $c_{0}\in \RR$. But there are not $\phi_{1}$ that are $\Gamma$-invariant (see Remark \ref{nolinear}) since $\Gamma$ is nontrivial, so the only possibility is that
\begin{equation}u^{(0)}\st |x| \dt= c_{0}+\upsilon_{0}\st |x| \dt\,.\end{equation}
\noindent We now show that $u$ is actually constant, indeed $u-c_{0}\in C_{\delta-2m}^{n,\alpha}\st X \dt$ and
\begin{equation}\Le\sq u-c_{0} \dq=\Delta_{\eta}^{2}\sq u-c_{0} \dq=0\end{equation}
\noindent so by Proposition \ref{isomorfismopesati} we can conclude
\begin{equation}u-c_{0}\equiv 0\,.\end{equation}
The proposition now follows immediately.

\end{proof}


\section{Nonlinear analysis}
\noindent In this section we collect all the estimates needed in the proof of Theorem\ref{maintheorem}. As in \cite{ap1} and \cite{ap2} we produce Kcsc metrics on orbifolds with boundary  which we believe could be of independent interest (Propositions \ref{crucialbase}, \ref{crucialmodello}). \newline
\newline
{\em From now on we will assume that the points in $\p\subset M$ have resolutions which are Ricci-Flat  ALE K\"ahler manifold. }
\newline
\begin{remark}
We recall that, by \cite[Theorem 8.2.3]{j},  when an ALE K\"ahler manifold is Ricci-flat then $e\st \Gamma \dt=0$. 
\end{remark}

\noindent Given $\varepsilon$ sufficiently small we look at the truncated orbifolds $M_{\rep}$  and $X_{\Gamma_{j},\Rep}$ for $j=1,\ldots, N $  where we impose the following relations:
\begin{equation}
\rep=\varepsilon^{\frac{2m-1}{2m+1}}=\varepsilon \Rep.
\end{equation}

\noindent We want to construct families of Kcsc metrics on $M_{\rep}$ and $X_{\Gamma_{j},\Rep}$ perturbing K\"ahler potentials of $\omega$ and $\eta_{j}$'s. We build these perturbations in such a way that they depend on parameters that we call {\em pseudo- boundary data} and  we can also prescribe, with some freedom,  {\em principal asymptotics} of the resulting Kcsc metrics.   
By {\em principal asymptotics} we mean the terms of the potentials of the families of Kcsc metrics on $M_{\rep}$ that near points $p_{j}$ beheave like $|z|^{2-2m}$ or $|z|^{4-2m}$  and the terms of the potentials of the families of Kcsc metrics on $X_{\Gamma_{j},\Rep}$ approaching infinity beheave like $|x|^{2-2m}$ or $|x|^{4-2m}$.  
In a second moment we choose the exact shape of these asymptotics by specifying some free parameters ({\em tuning}). The {\em pseudo-boundary data} form a particular set of functions on the unit sphere and they are the parameters that rule the behavior of the families of Kcsc metrics at the boundaries $\partial M_{\rep}$ and $\partial X_{\Gamma_{j},\Rep}$. They are the main tool for gluing the various families of metrics to a unique Kcsc metric on the resulting manifold, indeed their arbitrariness will allow us to perform the procedure of data matching. We call them {\em pseudo-boundary data} 
because they represent small perturbations of the (suitably rescaled) potentials of the Kcsc metrics at the boundaries.

\begin{notation}
{\em For the rest of the section $\chi_{j}$ will denote a  smooth cutoff functions identically equal to $1$ on $B_{2r_{0}}\st p_{j} \dt$ and identically equal to $0$ outside $B_{3r_{0}}\st p_{j} \dt$.}
\end{notation}  

\subsection{{\em Pseudo-boundary data and euclidean Biharmonic extensions}}\label{biharmonicext}

 A key technical tool to implement such a 
strategy is given by using  outer (which will be transplanted on the base orbifold) and inner (transplanted on the model) euclidean biharmonic extensions of functions on the unit sphere. We define now the outer biharmonic extensions of functions on the unit sphere.  
Let $(h,k)\in C^{4,\alpha}\st\Sp^{2m-1}\dt\times C^{4,\alpha}\st\Sp^{2m-1}\dt$ the outer biharmonic extension of $(h,k)$ is the function $H_{h,k}^{o}\in C^{4,\alpha}\st \CC^{m}\setminus B_{1} \dt$ solution fo the boundary value problem
\begin{equation}\begin{cases}
\Delta^{2} H_{h,k}^{out}=0 & \textrm{ on } \CC^{m}\setminus B_{1}\\
H_{h,k}^{out}=h &\textrm{ on } \partial B_{1}\\
\Delta H_{h,k}^{out}=k& \textrm{ on } \partial B_{1}
\end{cases}\end{equation}
Moreover $H_{h,k}^{out}$ has the following expansion in Fourier series for $m\geq 3$
\begin{equation}H^{out}_{h,k}:=\sum_{\gamma=0}^{+\infty}\st  \st  h^{(\gamma)}+\frac{k^{(\gamma)}}{4(m+\gamma-2)}  \dt |w|^{2-2m-\gamma}-\frac{k^{(\gamma)}}{4(m+\gamma-2)}|w|^{4-2m-\gamma} \dt \phi_\gamma\,,\label{eq:bihout3}\end{equation}
and for $m=2$
\begin{equation} H^{out}_{h,k}:=   h^{(0)}|w|^{-2}+\frac{k^{(0)}}{2}\log\st|w|\dt+  \sum_{\gamma=1}^{+\infty}\st  \st  h^{(\gamma)}+\frac{k^{(\gamma)}}{4\gamma}  \dt |w|^{-2-\gamma}-\frac{k^{(\gamma)}}{4\gamma}|w|^{-\gamma} \dt \phi_\gamma\,.\label{eq:bihout2}\end{equation}

\begin{remark}
In the sequel we will take $\Gamma$-invariant $(h,k)\in C^{4,\alpha}\st\Sp^{2m-1}\dt\times C^{4,\alpha}\st\Sp^{2m-1}\dt$ and by the Remark \ref{nolinear} we will 
have no terms with $\phi_{1}$ in the formulas  \eqref{eq:bihout3} and \eqref{eq:bihout2} for nontrivial $\Gamma$.
\end{remark}

We define also the inner biharmonic extensions of functions on the unit sphere. Let $\st \tilde{h}, \tilde{k}\dt\in C^{4,\alpha}\st \Sp^{2m-1} \dt\times C^{2,\alpha}\st \Sp^{2m-1} \dt$, the biharmonic extension $\hkjj$ on $B_1$ of $\st \tilde{h}, \tilde{k}\dt$ is the function $\hkjj\in C^{4,\alpha}\st \overline{B_{1}} \dt$ given by the solution of the boundary value problem
\begin{equation}\begin{cases}
\Delta^{2}H_{\tilde{h},\tilde{k}}^{in}=0 & w\in  B_{1} \\
H_{\tilde{h},\tilde{k}}^{in}=\tilde{h}& w\in \partial B_{1}\\
\Delta H_{\tilde{h},\tilde{k}}^{in}=\tilde{k}& w\in \partial B_{1}\\
\end{cases}\,.\end{equation}

\noindent The function $\hkjj$  has moreover the expansion

\begin{equation}\hkjj\st w\dt=\sum_{\gamma=0}^{+\infty}\left(\left(\tilde{h}^{(\gamma)}-\frac{\tilde{k}^{(\gamma)}}{4(m+\gamma)}  \dt |w|^\gamma+\frac{\tilde{k}^{(\gamma)}}{4(m+\gamma)}|w|^{\gamma+2} \dt \phi_\gamma\,.\end{equation}

\begin{remark}
Again, if the group $\Gamma$ is non trivial and  for $\Gamma$-invariant $(h,k)$, by Remark \ref{nolinear} , there  will be no $\phi_1$-term in the above summations. So we will have
\begin{equation}H_{h,k}^{in}=\left(\tilde{h}^{(0)}-\frac{\tilde{k}^{(0)}}{4m}  \dt  +\frac{\tilde{k}^{(0)}}{4m}|w|^{2}+  \sum_{\gamma=2}^{+\infty}\left(\left(\tilde{h}^{(\gamma)}-\frac{\tilde{k}^{(\gamma)}}{4(m+\gamma)}  \dt |w|^{\gamma}+\frac{\tilde{k}^{(\gamma)}}{4(m+\gamma)}|w|^{\gamma+2} \dt \phi_{\gamma}\,.\end{equation}
\end{remark}

As in \cite{ap1}, \cite{ap2} we introduce some functional spaces that will be needed in the sequel that will naturally work as ``space of parameters" for our construction:

\begin{equation}
\mathcal{B}_{j}:=C^{4,\alpha}\st \Sp^{2m-1}/\Gamma_{j} \dt\times C^{2,\alpha}\st \Sp^{2m-1}/\Gamma_{j} \dt
\end{equation}

\begin{equation}
\mathcal{B}:=\prod_{j=1}^{N}\mathcal{B}_{j}
\end{equation}

\begin{equation}\label{eq:dombd}
\mathcal{B}\st\kappa, \delta\dt:=\sg \st \hg,\kg \dt\in \mathcal{B}\,\left|\, \left\|h^{(0)}_{j},k^{(0)}_{j}\right\|_{\mathcal{B}_{j}}\leq \kappa \varepsilon^{4m+2}\rep^{-6m+4-\delta},  \left\|h^{(\dagger)}_{j},k^{(\dagger)}_{j}\right\|_{\mathcal{B}_{j}}\leq \kappa \varepsilon^{2m+4}\rep^{2-4m-\delta} \right. \dg
\end{equation}

\noindent We call the functions in $\mathcal{B}\st\kappa, \delta\dt$ {\em pseudo-boundary data} and  will be used to parametrize solution of the Kcsc problem near a given
``skeleton" solution built by hand to match some of the first orders of the metrics coming on the two sides of the gluing.

\subsection{Kcsc metrics on the truncated base orbifold.}\label{nonlinbase}

\noindent We start with a Kcsc orbifold $\st M,\omega, g \dt$ with isolated singular points such that there is a subset of sungular points  $\p\subset M$ whose elements have  resolutions which are Ricci-flat  ALE K\"ahler manifold. We want to find $F_{{\bf 0},\bg,\cg,\hg,\kg}^{out}\in C^{4,\alpha}\st M_{\rep} \dt$ such that 

\begin{equation}\omega_{{\bf 0},\bg,\cg,\hg,\kg}:=\omega +i\dd F_{{\bf 0},\bg,\cg,\hg,\kg}^{out}\end{equation} 

\noindent is a metric on $M_{\rep}$ and its scalar curvature $s_{\omega_{{\bf 0},\bg,\cg,\hg,\kg}}$ is a small perturbation of the scalar curvature $s_{\omega}$ of the reference K\"ahler metric on $M$.

\noindent   The function $F_{{\bf 0},\bg,\cg,\hg,\kg}^{out}$ consists of four blocks
\begin{equation}
F_{{\bf 0},\bg,\cg,\hg,\kg}^{out}:=-\varepsilon^{2m}{\bf{G}}_{\bf{0},\bg,\cg}+\Pbe+ \hko+ f_{{\bf 0},\bg,\cg,\hg,\kg}^{out}\,
\end{equation}  
the  skeleton $\varepsilon^{2m}{\bf{G}}_{\bf{0},\bg,\cg}$, extensions of {\em pseudo-boundary data} $\hko$, transplanted potentials of $\eta_{j}$'s $\Pbe$ and a "small" correction term $f_{{\bf 0},\bg,\cg,\hg,\kg}^{out}$ that has to be determined. We want $F_{{\bf 0},\bg,\cg,\hg,\kg}^{out}$  be a small perturbation of $\omega$ and hence we can use the expansion in Proposition \ref{espsg} to look for the equation that $f_{{\bf 0},\bg,\cg,\hg,\kg}^{out}$ has to satisfy on $M_{\rep}$. We have  
\begin{align}\label{eq:basegrezza1}
s_{\omega_{{\bf 0},\bg,\cg,\hg,\kg}}=&\,\s_{\omega}\st -\varepsilon^{2m}{\bf{G}}_{\bf{0},\bg,\cg}+\Pbe+ \hko+ f\dt\\
=&s_{\omega}-\frac{1}{2} \varepsilon^{2m}\nu_{{\bf 0},\cg}-\frac{1}{2}\Lg\sq \Pbe \dq -\frac{1}{2}\Lg \sq \hko \dq -\frac{1}{2}\Lg \sq f\dq\\
&+\frac{1}{2}\NN _{\omega}\st -\varepsilon^{2m}{\bf{G}}_{\bf{0},\bg,\cg}+\Pbe+ \hko+ f\dt
\end{align}
where in the second line we used the very definition of ${\bf{G}}_{\bf{0},\bg,\cg}$. Rewriting the above equation in terms of the unknown $f$ we obtain
\begin{align}\label{eq:basegrezza2}
\Lg \sq f\dq=&\st 2s_{\omega}- \varepsilon^{2m}\nu_{{\bf 0},\cg}-2s_{\omega_{{\bf 0},\bg,\cg,\hg,\kg}}\dt-\Lg\sq \Pbe \dq -\Lg \sq \hko \dq \\
&+\NN _{\omega}\st -\varepsilon^{2m}{\bf{G}}_{\bf{0},\bg,\cg}+\Pbe+ \hko+ f\dt\,.
\end{align}
The rest of this section is devoted to solve this equation.
\vspace{10pt}

\begin{itemize}
\item[] {\bf Skeleton.} The skeleton is made of {\em multi-poles fundamental solutions} ${\bf{G}}_{\bf{0},\bg,\cg}$  of $\Lg$ introduced in section \ref{multipoles}. These can be regarded as functions defined on $M_{\p}$ that are in $\ker\Lg$ and blow up approaching points in $\p$. For this reason, the existence of a skeleton, is strictly related to balancing conditions \eqref{eq: balancingbc1} and \eqref{eq: balancingbc2} in Remark \ref{gabc} 
with $\ag=0$, namely
\begin{equation}
\begin{array}{rll}
\sum_{j=1}^Nb_j(\Delta\varphi_i)(p_j) \, + \, \sum_{j=1}^Nc_j\varphi_i(p_j)   & = &  0 \\
&&\\
\sum_{j=1}^N c_j  &=& \nu_{{\bf 0},\ccc}  \, {\rm Vol}_\omega(M) 
\end{array}
\end{equation}
so that 
\begin{equation}
\Lg \left[  \mathbf{G}_{{\bf 0},\bbb,\ccc} \right] \,  + \, \nu_{{\bf 0},\ccc}   =    \sum_{j=1}^N  b_j \, \Delta\delta_{p_{j}} \, + \, \sum_{j=1}^N c_j \, \delta_{p_{j}} \, , \qquad \hbox{in \,\,\,$M$}\, .
\end{equation}
for a local description of the skeleton it is useful to keep in mind that,  by Lemma \ref{Gbilapl}, near points  $p_{j}$ we have the expansion      
\begin{equation}
{\bf{G}}_{\bf{0},\bg,\cg}\sim \frac{b_{j} |\Gamma_{j}|}{2\st m-1 \dt |\Sp^{2m-1}|}G_{\Delta}\st p_{j},z \dt.
\end{equation} 
It is clear that the form
\begin{equation}
\omega +i\dd \sq - \varepsilon^{2m}{\bf{G}}_{\bf{0},\bg,\cg}+\st \frac{b_{j}|\Gamma_{j}|}{ 2\cgaj  \st m-1 \dt|\Sp^{2m-1}|}  \dt^{\frac{1}{m}}\varepsilon^{2}\chi_{j}\psi_{\eta_{j}}\st  \st \frac{2 \cgaj  \st m-1 \dt|\Sp^{2m-1}|}{b_{j}|\Gamma_{j}|}  \dt^{\frac{1}{m}} \frac{z}{\varepsilon}  \dt   \dq
\end{equation}
matches exactly at the highest order the form  $\st \frac{b_{j}|\Gamma_{j}|}{2 \cgaj  \st m-1 \dt|\Sp^{2m-1}|}  \dt^{\frac{1}{m}}\eta_{j}$, once we rescale (as we will in the final gluing) the model using the map 
\begin{equation}
x= \st \frac{2 \cgaj  \st m-1 \dt|\Sp^{2m-1}|}{b_{j}|\Gamma_{j}|}  \dt^{\frac{1}{m}} \frac{z}{\varepsilon},
\end{equation} 
where the coefficient $ \cgaj  $ is given by Proposition \ref{asintpsieta}. It is then convenient, from now on, to set the following notation
\begin{equation}\label{eq:Bj}
B_{j}=\st \frac{b_{j}|\Gamma_{j}|}{ 2\cgaj  \st m-1 \dt|\Sp^{2m-1}|}  \dt^{\frac{1}{2m}}.
\end{equation}
It will also be convenient to identify the right constants $C_{j}$ such that 
\begin{equation}
\Lg\st {\bf{G}}_{\bf{0},\bg,\cg}-\sum_{j=1}^{N} \cgaj  B_{j}^{2m}G_{\Delta}\st p_{j},z \dt+ C_{j}G_{\Delta\Delta}\st p_{j},z \dt \dt\in C^{0,\alpha}\st M \dt.
\end{equation}
By Lemma \ref{Glapl} one gets
\begin{equation}\label{eq:Cj}
C_{j}=\frac{|\Gamma_{j}|}{8\st m-2 \dt\st m-1 \dt}\sq  2\cgaj  B_{j}^{2m}\frac{\st m-1 \dt|\Sp^{2m-1}|}{m|\Gamma_{j}|}s_{\omega}\st 1+\frac{\st m-1 \dt^{2}}{\st m+1 \dt} \dt-c_{j} \dq.
\end{equation}
The highest blow-up terms of $G_{\Delta}, G_{\Delta\Delta}$ in  ${\bf{G}}_{\bf{0},\bg,\cg}$ i.e. terms exploding like $|z|^{2-2m}, |z|^{4-2m}$   are the {\em principal asymptotics} of the family of Kcsc metrics $\omega_{{\bf 0},\bg,\cg,\hg,\kg}$. At the moment of data matching, the coefficients $B_{j}$'s and $C_{j}$'s will be ``{\em tuned}" in such a way that, {\em principal asymptotics} of $\omega_{\bf{0},\bg,\cg,\hg,\kg}$ on $M_{\rep}$ will match exactly the "principal asymptotics" of $\varepsilon^{2}\eta_{\tilde{b}_{j},\tilde{h}_{j},\tilde{k}_{j}}$'s  on $X_{\Gamma_{j},\frac{\Rep}{\tilde{b}_{j}}}$'s. More precisely,  under suitable rescalings, the $|z|^{2-2m}$ terms of ${\bf{G}}_{\bf{0},\bg,\cg}$    will match exactly the $|x|^{2-2m}$ terms of the potentials at infinity of $\eta_{j}$'s   and also $|z|^{4-2m}$ terms  will match exactly the  correction terms $|x|^{4-2m}$  that pop up transplanting potential of $\omega$ on $X_{\Gamma_{j}}$. The justification for this procedure will come  at the moment of data matching. indeed, when we will look at the metrics at the boundaries,  it will be clear that the $\varepsilon$-growths of the {\em principal asymptotics} are the maximum among all terms constituting the family  $\omega_{{\bf 0},\bg,\cg,\hg,\kg}$ and are in fact too large to be controlled  by the extensions of  {\em pseudo-boundary data} (introduced just below here). For general $\bg,\cg$ as in assumptions of Proposition \ref{crucialbase} the data matching procedure becomes hence impossible.   To overcome this difficulty we are forced to impose relations on $\bg,\cg$  with the {\em tuning procedure}, and in some sense we fix them, in order to  have that the extensions of  {\em pseudo-boundary data} control all the components of $\omega_{{\bf 0},\bg,\cg,\hg,\kg}$ not perfectly matched. The {\em tuning procedure}, although it could appear as a merely technical procedure, has strong geometric consequences indeed  it  yields to the key condition \eqref{eq:tuning} of Theorem \ref{maintheorem} and hence puts constraints on the "symplectic positions" of singular points. 
\vspace{10pt}
\item[] {\bf Extensions of pseudo-boundary data.} Using the notion of euclidean outer biharmonic extensions of functions on the sphere we define  for $\st \hg,\kg \dt\in \dombd$ 
\begin{equation}\label{eq:hko}
\hko:=\sum_{j=1}^{N}\chi_{j}H_{h_{j}^{(\dagger)},k_{j}^{(\dagger)}}^{out}\st \frac{z}{\rep} \dt\,.
\end{equation}
\noindent  When we will look to this term at the boundary we will see that it has the second $\varepsilon$-growth after the {principal asymptotics} and it will become the highest $\varepsilon$-growth after the ``{\em tuning}" of {\em principal asymptotics}.  We will have, hence, that  extensions of {\em pseudo-boundary data} dominate every other term with respect to $\varepsilon$-growth. Moreover thanks to the arbitrariness of $\st \hg,\kg \dt$, we can perform the Cauchy data matching procedure and glue the various metrics to a unique one.

\vspace{10pt}

\item[] {\bf Transplanted potentials.} As Sz\'ekelyhidi does in \cite{Gabor1} and \cite{Gabor2}, we bring to $M_{\rep}$ the potentials of $\eta_{j}$'s suitably rescaled and cut off in order to have better estimates through algebraic simplifications. Indeed, using the fact that $\eta_{j}$'s are scalar flat we obtain some useful cancellations when compute the magnitude of the error we commit adding to $\omega$ "artificial" terms like the skeleton and the transplanted potentials. In $x$-coordinates on $X_{\Gamma_{j}}$'s we have 
\begin{align}\label{eq:psij1}
0\equiv&\,\s_{eucl}\st -\cgaj|x|^{2-2m}+\psi_{\eta_{j}}\st x \dt \dt\\
=&\,-\frac{1}{2}\Delta^{2}\sq \psi_{\eta_{j}}\st x \dt \dq +\frac{1}{2}\NN_{eucl}\st -\cgaj|x|^{2-2m}+\psi_{\eta_{j}}\st x \dt\dt \,,
\end{align}
with $\psi_{\eta_{j}}$'s potentials "at inifinity" of metrics $\eta_{j}$'s defined in Section \ref{preliminaries}  Proposition \ref{asintpsieta} formula \eqref{eq: poteta}. With the rescaling 
\begin{equation}
x=\frac{B_{j}z}{\varepsilon}\,,
\end{equation} 
where the  coefficients $B_{j}$'s are defined in formula \eqref{eq:Bj}, we consider the term
\begin{equation}\label{eq:transplanted}
\textbf{P}_{\bg,\etag}:=\sum_{j=1}^{N}B_{j}^{2}\varepsilon^{2}\chi_{j}\psi_{\eta_{j}}\st \frac{z}{B_{j}	\varepsilon} \dt\,.
\end{equation}

We can rewrite identities \eqref{eq:psij1}  as follows
\begin{align}\label{eq:psij2}
0\equiv&\,\s_{eucl}\st -\cgaj\varepsilon^{2m}B^{2m}|z|^{2-2m}+ \Pbe\dt\\
=&\,-\frac{1}{2}\Delta^{2}\sq\textbf{P}_{\bg,\etag}\dq+\frac{1}{2}\NN_{eucl}\st -\cgaj\varepsilon^{2m}B^{2m}|z|^{2-2m}+ \textbf{P}_{\bg,\etag} \dt\,.
\end{align}
Unfortunately, since we are not in the euclidean setting, we have 
\begin{equation}
-\frac{1}{2}\Lg\sq\textbf{P}_{\bg,\etag}\dq+\frac{1}{2}\NN_{\omega}\st -\cgaj\varepsilon^{2m}B^{2m}|z|^{2-2m}+ \textbf{P}_{\bg,\etag} \dt\neq 0\,
\end{equation}
and hence we produce an error that has to be corrected by the solution $f$ of the equation \eqref{eq:basegrezza1}. The size of the solution $f$ grows as the error grows and we need $f$ to be small to be able to perform the Cauchy data matching procedure.  So we want to minmize as much as possible this error. Here two facts come into play, the first is  that on a small ball centered at $p_{j}\in \p$ the metric $\omega$ osculates with order two to the euclidean one and the second is that we substitute $\cgaj\varepsilon^{2m}B^{2m}|z|^{2-2m}$ with $\varepsilon^{2m}{\bf{G}}_{\bf{0},\bg,\cg}$ whose principal asymptotic is exactly $\cgaj\varepsilon^{2m}B^{2m}|z|^{2-2m}$. As we will see in the sequel ( precisely in the proof of Proposition \ref{stimascheletrobase}) we can use these two facts and relations \eqref{eq:psij2} to produce sharp estimates for the error $\s_{\omega}\st -\varepsilon^{2m}{\bf{G}}_{\bf{0},\bg,\cg} +\Pbe \dt$ and verify that is sufficiently small to allow us to perform the Cauchy data matching procedure and hence conclude the gluing construction.           

\vspace{10pt}
\item[] {\bf Correction term.} It is the term that ensures the constancy of the scalar curvature of the metric $\omega_{{\bf 0},\bg,\cg,\hg,\kg}$ on $M_{\rep}$ and it is a function $f_{{\bf 0},\bg,\cg,\hg,\kg}^{out}\in \Cqdd$ if $m\geq 3$ and $f_{{\bf 0},\bg,\cg,\hg,\kg}^{out}\in \Cqddd$ if $m=2$,  where the spaces $\Cqdd$ and $\Cqddd$ are defined in Subsection \ref{solutionlinearized} by formulas \eqref{deficiency1} and \eqref{deficiency2}. As the notation suggests, the function $f_{{\bf 0},\bg,\cg,\hg,\kg}^{out}$ depends nonlinearly on  $\st \hg,\kg \dt$ and $\bg$ and we find it by solving a fixed point problem on a suitable closed and bounded subspace of $ \Cqdd$ if $m\geq 3$ and $\in \Cqddd$ if $m=2$.

\end{itemize} 

\begin{notation}
{\em For the rest of the paper we will denote with $\mathsf{C}$ a positive constant, that can vary from line to line, depending only on $\omega$ and $\eta_{j}$'s.  }
\end{notation}

 We can now state the main proposition for the base space, whose proof will fill the rest of this subsection:

\begin{prop}
\label{crucialbase}
Let $(M,g,\omega )$ a Kcsc orbifold with isolated singularities and let $\p$ be the set of singular points with non trivial orbifold group that admit a \K\ Ricci flat resolution.

\begin{itemize}
\item  Assume exist $\bg\in \st\RR^{+}\dt^{N}$ and $\cg\in \RR^{N}$ such that 
	\begin{displaymath}
\left\{\begin{array}{lcl}
\sum_{j=1}^{N}b_{j}\Delta_{\omega}\varphi_{i}\st p_{j} \dt+c_{j}\varphi_{i}\st p_{j} \dt=0 && i=1,\ldots, d\\
&&\\
\st\Theta\st \bg,\cg \dt\dt_{\substack{1\leq i\leq d\\ 1\leq j\leq N}}&& \textrm{has full rank}
\end{array}\right.
\end{displaymath}
where $\st\Theta\st \bg,\cg \dt\dt_{\substack{1\leq i\leq d\\ 1\leq j\leq N}}$ is the matrix introduced in Section \ref{linearanalysis} formula \eqref{eq:nondeggen}. Let ${\bf{G}}_{\bf{0},\bg,\cg}$ be the multi-poles solution of $\Lg$ constructed in Section \ref{linearanalysis}  Remark \ref{gabc}.

\item    Let $\delta\in (4-2m,5-2m)$. Given any $\st \hg,\kg \dt\in \dombd$, where $\dombd$ is the space defined in formula \eqref{eq:dombd},  let $\hko$ be the function defined in formula \eqref{eq:hko}. 
\begin{equation}
\hko:=\sum_{j=1}^{N}\chi_{j}H_{h_{j}^{(\dagger)},k_{j}^{(\dagger)}}^{out}\st \frac{z}{\rep} \dt\,.
\end{equation}
\item Let ${\bf P}_{\bg,\etag}$ be the transplanted potentials defined in formula \eqref{eq:transplanted}
\begin{equation}
{\bf P}_{\bg,\etag}:=\sum_{j=1}^{N}B_{j}^{2}\varepsilon^{2}\chi_{j}\psi_{\eta_{j}}\st \frac{z}{B_{j}	\varepsilon} \dt\,.
\end{equation}
\end{itemize}
Then there exists $f_{{\bf 0},\bg,\cg,\hg,\kg}^{out}\in \Cqdd$ if $m\geq 3$ and $f_{{\bf 0},\bg,\cg,\hg,\kg}^{out}\in \Cqddd$ if $m=2$ such that 
\begin{equation}
\omega_{{\bf 0},\bg,\cg,\hg,\kg}=\omega+i\dd\st {-\varepsilon^{2m}\bf{G}}_{\bf{0},\bg,\cg}  +{\bf P}_{\bg,\etag}+ \hko +f_{{\bf 0},\bg,\cg,\hg,\kg}^{out}  \dt
\end{equation}
is a Kcsc metric on $M_{\rep}$ and the following estimates hold
\begin{equation}
\begin{array}{lll}
\left\|f_{\bf{0},\bg,\cg,\hg,\kg}^{out}\right\|_{\Cqdd}&\leq \mathsf{C} \varepsilon^{2m+2}\rep^{2-2m-\delta}&\textrm{ for } m\geq 3\,,\\
&&\\
\left\|f_{\bf{0},\bg,\cg,\hg,\kg}^{out}\right\|_{\Cqddd}&\leq \mathsf{C} \varepsilon^{6}\rep^{-2-\delta}&\textrm{ for } m= 2\,.
\end{array}
\end{equation}
 Moreover $s_{\omega_{{\bf 0},\bg,\cg,\hg,\kg}}$, the scalar curvature of $\omega_{{\bf 0},\bg,\cg,\hg,\kg}$, is a small perturbation of $s_{\omega}$, the scalar curvature of the background metric $\omega$ and we have
\begin{equation}
\left| s_{\omega_{{\bf 0},\bg,\cg,\hg,\kg}}-s_{\omega}\right|\leq \mathsf{C}\varepsilon^{2m}\,.
\end{equation}
\end{prop}

Since the scalar curvature $s_{\omega_{{\bf 0},\bg,\cg,\hg,\kg}}$ is going to be a small perturbation of $s_{\omega}$ we can write 
\begin{equation}\label{eq:sbchk}
s_{\omega_{{\bf 0},\bg,\cg,\hg,\kg}}:=s_{\omega}+\frac{1}{2}s_{{\bf 0},\bg,\cg,\hg,\kg}
\end{equation}
where $s_{\omega_{{\bf 0},\bg,\cg,\hg,\kg}}$ is a small constant depending on $\varepsilon$ such that
\begin{equation}
\lim_{\varepsilon\rightarrow 0}s_{{\bf 0},\bg,\cg,\hg,\kg}=0\,.
\end{equation}

\noindent  In order to find the correction $f_{\bf{0},\bg,\cg,\hg,\kg}^{out}$ we set up a fixed point problem that will be solved using Banach-Caccioppoli Theorem. We can rewrite equation \eqref{eq:basegrezza2} in the following form.
\begin{align}\label{eq: PDEbase}
\Lg\sq f\dq+s_{\bf{0},\bg,\cg,\hg,\kg}+ \varepsilon^{2m}\nu_{\bf{0},\cg}=&-\Lg\sq\Pbe\dq -\Lg\sq \hko\dq\\
&+\NN _{\omega}\st -\varepsilon^{2m}{\bf{G}}_{\bf{0},\bg,\cg}+\Pbe+ \hko+ f\dt\,.
\end{align}

}

\noindent The assumption of Proposition \ref{crucialbase} that there exist $\bg\in \st\RR^{+}\dt^{N}$ and $\cg\in \RR^{N}$ such that the matrix
\begin{equation}
\st\Theta_{ij}\st \bg,\cg \dt \dt_{\substack{1\leq i\leq d\\1\leq j\leq N}}
\end{equation}  
has full rank enables us, making use of Theorem \ref{invertibilitapesatodef}  Remark \ref{inversadef},  to invert the operator $\mathbb{L}_{{\bf 0},\bg,\cg}^{(\delta)}$ on $M_{\p}$. It is  then useful to consider a PDE  on the whole $M_{\p}$  such that on $M_{\rep}$ reduces to the \eqref{eq: PDEbase}. To this aim we introduce a truncation-extension operator on weighted H\"older spaces.  Let $f\in C_{\delta}^{0,\alpha}\st  M \dt $ we define $\mathcal{E}_{\rep}:C_{\delta}^{0,\alpha}\st  M \dt \rightarrow C_{\delta}^{0,\alpha}\st  M \dt $
\begin{displaymath}
\mathcal{E}_{\rep}\st  f \dt :\left\{ \begin{array}{ll}
f\st  z \dt  & z\in B_{2\rep}\setminus B_{\rep}   \\
f\st  \rep\frac{z}{|z|} \dt \chi\st  \frac{|z|}{\rep} \dt  & z\in B_{ \rep }\setminus B_{\frac{\rep}{2}}\\
0  & z\in B_{\frac{\rep}{2}}  
\end{array} \right.
\end{displaymath}
where $\chi\in C^{\infty}\st [0,+\infty) \dt$ is a cutoff function identically equal to  $1$ on $[1,+\infty)$ and identically equal to $0$ on $\sq 0,\frac{1}{2} \dq$. Now we use the truncation-extension operator and we find our differential equation.

\begin{align}\label{eq: PDEbaseMp}
\Lg\sq f\dq+s_{\bf{0},\bg,\cg,\hg,\kg}+ \varepsilon^{2m}\nu_{\bf{0},\cg}=&-\mathcal{E}_{\rep}\Lg\sq\Pbe\dq -\mathcal{E}_{\rep}\Lg\sq \hko\dq\\
&+\mathcal{E}_{\rep}\NN _{\omega}\st -\varepsilon^{2m}{\bf{G}}_{\bf{0},\bg,\cg}+\Pbe+ \hko+ f\dt\,.
\end{align}

To set up the fixed point problem we use the inverse $\mathbb{J}_{{\bf 0},\bg,\cg}^{(\delta)}$ of $\mathbb{L}_{{\bf 0},\bg,\cg}^{(\delta)}$  of Remark	 \ref{inversadef} and we construct the following operator 
\begin{equation}
\begin{array}{lll}
\mathbb{T}_{{\bf 0},\bg,\cg}^{(\delta)}: & \Cqdd\times \dombd \rightarrow \Cqdd &\textrm{ for } m\geq 3\\
&&\\
\mathbb{T}_{{\bf 0},\bg,\cg}^{(\delta)}:&  \Cqddd\times \dombd \rightarrow \Cqddd   &\textrm{ for } m= 2
\end{array}
\end{equation}

defined as

\begin{align}
\mathbb{T}_{{\bf 0},\bg,\cg}^{(\delta)}\st f,\hg,\kg \dt=&\mathbb{J}^{(\delta)}_{{\bf 0},\bg,\cg}\sq-\mathcal{E}_{\rep}\Lg\sq\Pbe\dq -\mathcal{E}_{\rep}\Lg\sq \hko\dq\right.\\
&\left.\phantom{aaaaa}+\mathcal{E}_{\rep}\NN _{\omega}\st -\varepsilon^{2m}{\bf{G}}_{\bf{0},\bg,\cg}+\Pbe+ \hko+ f\dt-s_{\bf{0},\bg,\cg,\hg,\kg}- \varepsilon^{2m}\nu_{\bf{0},\cg}\dq\,,
\end{align}

with

\begin{align}\label{eq: s}
\st s_{\bf{0},\bg,\cg,\hg,\kg}+ \varepsilon^{2m}\nu_{\bf{0},\cg} \dt \vol\st M \dt=&\int_{M}\sq  \Lg\sq\Pbe\dq -\Lg\sq \hko\dq   \dq\, d\mu_{\omega}\\
&+\int_{M}\sq \mathcal{E}_{\rep}\NN _{\omega}\st -\varepsilon^{2m}{\bf{G}}_{\bf{0},\bg,\cg}+\Pbe+ \hko+ f\dt- \varepsilon^{2m}\nu_{\bf{0},\cg}\dq  \,d\mu_{\omega}\,.
\end{align}

The constant $s_{\bf{0},\bg,\cg,\hg,\kg}$ is an undetermined parameter of our construction and, a priori, there is no restriction on its size. It is precisely in formula \eqref{eq: s} that we are forced to set its value and, as we anticipated, it turns out to be a small constant since 
\begin{equation}
s_{\bf{0},\bg,\cg,\hg,\kg}\approx -\varepsilon^{2m}\nu_{\bf{0},\cg}\,.
\end{equation}

\noindent We prove the existence of a solution of equation \eqref{eq: PDEbase} by finding, for fixed $\st\hg,\kg\dt\in \dombd$, a fixed point of the operator $\mathbb{T}_{{\bf 0},\bg,\cg}^{(\delta)}$
\begin{equation}
\begin{array}{lll}
\mathbb{T}_{{\bf 0},\bg,\cg}^{(\delta)}\st \cdot ,\hg,\kg\dt: & \Cqdd \rightarrow \Cqdd &\textrm{ for } m\geq 3\\
\mathbb{T}_{{\bf 0},\bg,\cg}^{(\delta)}\st \cdot ,\hg,\kg\dt:&  \Cqddd\rightarrow \Cqddd   &\textrm{ for } m= 2
\end{array}
\end{equation}
 hence showing it satisfies the assumptions of  contraction Theorem. More precisely we want to prove that there exist a domain $\Omega\subset \Cqdd$ (respectively $\Omega\subset \Cqddd$)  such that for any $f\in \Omega$ then  $\mathbb{T}_{{\bf 0},\bg,\cg}^{(\delta)}\st f,\hg,\kg \dt\in \Omega$ and $\mathbb{T}_{{\bf 0},\bg,\cg}^{(\delta)}\st \cdot,\hg,\kg \dt$ is a contraction on $\Omega$. The first step  is to estimate at $\mathbb{T}_{{\bf 0},\bg,\cg}^{(\delta)}\st 0,{\bf 0 },{\bf 0} \dt$ that heuristically  tells us ``how far'' is  the metric
\begin{equation}
\omega +i\dd\st  - \varepsilon^{2m}{\bf{G}}_{\bf{0},\bg,\cg}+\Pbe  \dt
\end{equation}   
from being Kcsc on $M_{\rep}$.

\begin{lemma}\label{stimascheletrobase}
Under the assumptions of Proposition \ref{crucialbase} the following estimates hold
\begin{equation}
\begin{array}{lll}
\left\|\mathbb{T}_{{\bf 0},\bg,\cg}^{(\delta)}\st 0,{\bf 0},{\bf 0} \dt\right\|_{\Cqdd}&\leq \mathsf{C}\varepsilon^{2m+2}\rep^{2-\delta-2m}&\textrm{ for  }m\geq 3\\
\left\|\mathbb{T}_{{\bf 0},\bg,\cg}^{(\delta)}\st 0,{\bf 0},{\bf 0} \dt\right\|_{\Cqddd}&\leq \mathsf{C}\varepsilon^{6}\rep^{-2-\delta}&\textrm{ for  }m= 2
\end{array}\,.
\end{equation}

\end{lemma}

\begin{proof} We give the proof for the case $m\geq 3$, the case $m=2$ is identical. For the sake of notation throughout this proof we set 
\begin{equation}
\psi_{j}\st z \dt:=B_{j}^{2}\varepsilon^{2}\chi_{j}\psi_{\eta_{j}}\st \frac{z}{B_{j}\varepsilon} \dt
\end{equation}

\noindent  We note that, on $M_{r_0}$, using estimates of Proposition \ref{asintpsieta} we have 
\begin{equation}\left\|-\mathcal{E}_{\rep}\Lg\sq \Pbe \dq +\mathcal{E}_{\rep}\NN _{\omega}\st -\varepsilon^{2m}{\bf{G}}_{\bf{0},\bg,\cg}+\Pbe\dt\right\|_{C^{0,\alpha}\st M_{r_0}\dt}\leq \mathsf{C} \varepsilon^{2m+2}\,.\end{equation}
 According to the definition of the weighted H\"older spaces  we now estimate on $B_{2r_{0}}\st p_{j} \dt$ the quantity 
\begin{equation}\sup_{\rho\in[\rep ,r_0]} \rho^{4-\delta}\left\| -\mathcal{E}_{\rep}\Lg\sq \psi_{j}\dq +\mathcal{E}_{\rep}\NN _{\omega}\st -\varepsilon^{2m}{\bf{G}}_{\bf{0},\bg,\cg}+\psi_{j}\dt  \right\|_{C^{0,\alpha}\st  B_{2} \setminus B_{1} \dt}\,. \end{equation}

\noindent On $B_{2r_0}$, we have 
\begin{align}
\Lg\sq\psi_{j}\dq -\NN _{\omega}\st -\varepsilon^{2m}{\bf{G}}_{\bf{0},\bg,\cg}+\psi_{j}\dt=&\, \Delta^{2}\sq\psi_{j}\dq -\NN _{eucl}\st - \cgaj  \varepsilon^{2m}B_{j}^{2m}|z|^{2m}+\psi_{j}\dt+  \textrm{\bf I}+ \textrm{\bf II}+\textrm{\bf III}\\
\end{align}

with 
\begin{equation}
\begin{array}{cll}
\textrm{\bf I}&:=&  \st\Lg-\Delta^{2}\dt\sq\psi_{j}\dq\\
\textrm{\bf II}&:=& \sq \NN _{eucl}\st - \cgaj  \varepsilon^{2m}B_{j}^{2m}|z|^{2m}+\psi_{j}\dt-\NN _{\omega}\st - \cgaj  \varepsilon^{2m}B_{j}^{2m}|z|^{2m}+\psi_{j}\dt\dq\\
\textrm{\bf III}&:=&\sq \NN _{\omega}\st - \cgaj  \varepsilon^{2m}B_{j}^{2m}|z|^{2m}+\psi_{j}\dt-\NN _{\omega}\st -\varepsilon^{2m}{\bf{G}}_{\bf{0},\bg,\cg}+\psi_{j}\dt\dq\,.
\end{array}
\end{equation}

\noindent The metric $\eta_{j}$ is Ricci-flat and hence scalar-flat and this fact, by \eqref{eq:psieta} , gives us the algebraic identity 
\begin{equation}
-\Delta^{2}\sq\psi_{j}\dq+\NN _{eucl}\st - \cgaj  \varepsilon^{2m}B_{j}^{2m}|z|^{2m}+\psi_{j}\dt=0
\end{equation}

\noindent With this cancellation, the only terms left to estimate are $\textrm{\bf I},\textrm{\bf II},\textrm{\bf III}$ and with standard, but cumbersome, computations  we obtain  

\begin{equation}\sup_{ \rho\in [\rep, r_0]}\rho^{-\delta+4}\left\| \textrm{\bf I}\right\|_{C^{0,\alpha}\st  B_{2} \setminus B_{1} \dt} \leq \mathsf{C}\varepsilon^{2m+2}\rep^{2-2m-\delta}\,, \end{equation}
\begin{equation}\sup_{ \rho\in [\rep, r_0]}\rho^{-\delta+4}\left\| \textrm{\bf II}\right\|_{C^{0,\alpha}\st  B_{2} \setminus B_{1} \dt} \leq \mathsf{C}\varepsilon^{4m}\rep^{4-\delta-4m}\,, \end{equation}
\begin{equation}\sup_{ \rho\in [\rep, r_0]}\rho^{-\delta+4}\left\| \textrm{\bf III}\right\|_{C^{0,\alpha}\st  B_{2} \setminus B_{1} \dt} \leq \mathsf{C} \varepsilon^{4m}\rep^{4-\delta-4m}\,. \end{equation}

\noindent We can conclude that

\begin{equation}\sup_{\substack{  1\leq j\leq N \\ \rho\in [\rep, r_0]}}\rho^{-\delta+4}\left\| \Lg\psi_{j} -2\NN _{\omega}\st -\varepsilon^{2m}{\bf{G}}_{\bf{0},\bg,\cg}+\psi_{j}\dt\right\|_{C^{0,\alpha}\st  B_{2} \setminus B_{1} \dt} \leq \mathsf{C} \varepsilon^{2m+2}\rep^{2-2m-\delta }\, 
\end{equation}
\noindent  and therefore the lemma is proved.
\end{proof}

\noindent  In light of  Lemma \ref{stimascheletrobase} we can take the quantity    $\left\|\mathbb{T}_{{\bf 0},\bg,\cg}^{(\delta)}\st 0,\textbf{0},\textbf{0} \dt\right\|_{\Cqdd}$ for $m\geq 3$  and  $\left\|\mathbb{T}_{{\bf 0},\bg,\cg}^{(\delta)}\st 0,\textbf{0},\textbf{0} \dt\right\|_{\Cqddd}$ for $m=2$ as a reference for the magnitude of the diameter of the $\Omega$ we are looking for.  Indeed if we consider the set of $f\in \Cqdd$ (respectively $f\in \Cqddd$) such that 
\begin{equation}
\begin{array}{lll}
\left\|f\right\|_{\Cqdd}&\leq 2 \left\|\mathbb{T}_{{\bf 0},\bg,\cg}^{(\delta)}\st 0,\textbf{0},\textbf{0} \dt\right\|_{\Cqdd}&=2\mathsf{C}\varepsilon^{2m+2}\rep^{2-2m-\delta }\\
\left\|f\right\|_{\Cqddd}&\leq 2 \left\|\mathbb{T}_{{\bf 0},\bg,\cg}^{(\delta)}\st 0,\textbf{0},\textbf{0} \dt\right\|_{\Cqddd}&=2\mathsf{C}\varepsilon^{6}\rep^{-2-\delta }
\end{array}
\end{equation}
we find our $\Omega$. The fact that, for fixed $\st \hg,\kg \dt\in \dombd$  
\begin{equation}
\mathbb{T}_{{\bf 0},\bg,\cg}^{(\delta)}\st \cdot, \hg,\kg \dt:\Omega\rightarrow \Omega
\end{equation}

\noindent and  is a well defined contraction follows from the following two Lemmas.

\begin{lemma}\label{stimabiarmonichebase}
Under the assumptions of Proposition \ref{crucialbase}, we have 
\begin{equation}
\left\|\mathcal{E}_{\rep}\Lg\hko\right\|_{C_{\delta-4}^{0,\alpha}\st M_{\p} \dt}\leq \mathsf{C}\left\|\hg^{(\dagger)},\kg^{(\dagger)}\right\|_{\mathcal{B}} \rep ^{2-\delta}\,.
\end{equation}

\end{lemma}
\begin{proof} 

This is a straightforward computation using Remark \ref{nolinear}.
\end{proof}

\begin{lemma}
Let  $ \st \hg',\kg' \dt \in \dombd$ and 
$f,f'\in \Cqdd$ if $m\geq 3$ and $f,f'\in \Cqddd$ if $m=2$ such that 
\begin{equation}
\left\|f\right\|_{\Cqdd},\left\|f'\right\|_{\Cqdd}\leq 2 \left\|\mathbb{T}_{{\bf 0},\bg,\cg}^{(\delta)}\st 0,{\bf 0},{\bf 0} \dt\right\|_{\Cqdd}.
\end{equation}
and respectively
\begin{equation}
\left\|f\right\|_{\Cqddd},\left\|f'\right\|_{\Cqddd}\leq 2 \left\|\mathbb{T}_{{\bf 0},\bg,\cg}^{(\delta)}\st 0,{\bf 0},{\bf 0} \dt\right\|_{\Cqddd}.
\end{equation}
\noindent If assumptions of Proposition \ref{crucialbase} are satisfied then the following estimates hold:

\begin{equation}
\begin{array}{lll}
\left\|\mathbb{T}_{{\bf 0},\bg,\cg}^{(\delta)}\st f,\hg,\kg \dt-\mathbb{T}_{{\bf 0},\bg,\cg}^{(\delta)}\st 0,\hg,\kg \dt\right\|_{\Cqdd}\,&\leq\, \frac{1}{2}\left\|\mathbb{T}_{{\bf 0},\bg,\cg}^{(\delta)}\st 0,{\bf 0},{\bf 0} \dt\right\|_{\Cqdd}&\textrm{ for }m\geq 3\\
\left\|\mathbb{T}_{{\bf 0},\bg,\cg}^{(\delta)}\st f,\hg,\kg \dt-\mathbb{T}_{{\bf 0},\bg,\cg}^{(\delta)}\st 0,\hg,\kg \dt\right\|_{\Cqddd}\,&\leq\, \frac{1}{2}\left\|\mathbb{T}_{{\bf 0},\bg,\cg}^{(\delta)}\st 0,{\bf 0},{\bf 0} \dt\right\|_{\Cqddd}&\textrm{ for }m=2\,;
\end{array}
\end{equation}

\begin{equation}
\begin{array}{lll}
\left\|\mathbb{T}_{{\bf 0},\bg,\cg}^{(\delta)}\st f,\hg,\kg \dt-\mathbb{T}_{{\bf 0},\bg,\cg}^{(\delta)}\st f',\hg,\kg \dt\right\|_{\Cqdd}\,&\leq\, \frac{1}{2}\left\|f-f'\right\|_{\Cqdd}&\textrm{ for }m\geq 3\\
\left\|\mathbb{T}_{{\bf 0},\bg,\cg}^{(\delta)}\st f,\hg,\kg \dt-\mathbb{T}_{{\bf 0},\bg,\cg}^{(\delta)}\st f',\hg,\kg \dt\right\|_{\Cqddd}\,&\leq\, \frac{1}{2}\left\|f-f'\right\|_{\Cqddd}&\textrm{ for }m=2\,;
\end{array}
\end{equation}

\begin{equation}
\begin{array}{lll}
\left\|\mathbb{T}_{{\bf 0},\bg,\cg}^{(\delta)}\st f,\hg,\kg \dt-\mathbb{T}_{{\bf 0},\bg,\cg}^{(\delta)}\st f,\hg',\kg' \dt\right\|_{\Cqdd}\,&\leq\, \frac{1}{2}\left\|\hg-\hg',\kg-\kg'\right\|_{\mathcal{B}}&\textrm{ for }m\geq 3\\
\left\|\mathbb{T}_{{\bf 0},\bg,\cg}^{(\delta)}\st f,\hg,\kg \dt-\mathbb{T}_{{\bf 0},\bg,\cg}^{(\delta)}\st f,\hg',\kg' \dt\right\|_{\Cqddd}\,&\leq\, \frac{1}{2}\left\|\hg-\hg',\kg-\kg'\right\|_{\mathcal{B}}&\textrm{ for }m=2\,;
\end{array}
\end{equation}

\end{lemma}
\begin{proof}
Follows by direct computation as \cite[Lemma 5.2]{ap2}.
\end{proof}

The proof of Proposition \ref{crucialbase} is now complete.

\subsection{Kcsc metrics on the truncated model spaces}\label{analisinonlinearemodello} We now want to perform on the model spaces $X_{\Gamma_{j}}$'s a similar analysis as in the previous Subsection. 

\begin{notation}
{\em To keep notations as short as possible we drop the subscript $j$. } 
\end{notation}
Our starting point is a Ricci-flat ALE K\"ahler manifold $\st X_{\Gamma},\eta, h \dt$ where we want to find $F_{\tilde{b},\tilde{h},\tilde{k}}^{in}\in C^{4,\alpha}\st X_{\Gamma,\frac{\Rep}{\tilde{b}}} \dt$ with $\tilde{b}\in \RR^{+}$ such that

\begin{equation}\eta_{\tilde{b},\tilde{h},\tilde{k}}:=\tilde{b}^{2}\eta+i\dd F_{\tilde{b},\tilde{h},\tilde{k}}^{in}\end{equation} 

\noindent is a metric on $X_{\frac{\Rep}{\tilde{b}}}$ and 
\begin{equation}
\s_{\tilde{b}^{2}\eta}\st F_{\tilde{b},\tilde{h},\tilde{k}}^{in} \dt=\varepsilon^{2}\st s_{\omega}+	\frac{1}{2}s_{\bf{0},\bg,\cg,\hg,\kg}\dt\,.
\end{equation}

with $\s_{\tilde{b}^{2}\eta}$  the operator introduced in \eqref{eq:espsg}.
\noindent The parameters $\tilde{b}, \tilde{h}, \tilde{k}$ will be chosen after  the construction of the familiy of Kcsc metrics on $X_{\Gamma,\frac{\Rep}{\tilde{b}}}$, in particular $\tilde{b}$ will be chosen with a ``manual tuning" of the {\em principal asymptotics} while $\tilde{h}, \tilde{k}$ with the Cauchy data matching procedure. The function $F_{\tilde{b},\tilde{h},\tilde{k}}^{in}$  will be made of three blocks:

\begin{equation}
F_{\tilde{b},\tilde{h},\tilde{k}}^{in}:=\Pbg+ \hkii +\csfii
\end{equation}

$\Pbg$ is the transplanted potential of $\omega$  that keeps the metric near to a Kcsc metric, $\hkii$ is the extension of {\em pseudo-boundary data} that will allow us to perform the Cauchy data matching procedure and a small perturbation $f_{\tilde{b},\tilde{h},\tilde{k}}^{in}$ that ensures the constancy of the scalar curvature. Since $F_{\tilde{b},\tilde{h},\tilde{k}}^{in}$ has to be a small perturbation we can use the expansion in Proposition \ref{espsg} to look for the equation that $\csfii$ has to satisfy and we have

\begin{align}\label{eq: modellogrezza}
\varepsilon^{2} s_{\omega}+\frac{1}{2}\varepsilon^{2}s_{{\bf 0},\bg,\cg,\hg,\kg}=&\,\s_{\tilde{b}^{2}\eta}\st \Pbg+ \hkii +f\dt\\
=&\, \s_{\tilde{b}^{2}\eta}\st 0 \dt -\frac{1}{2}\mathbb{L}_{\tilde{b}^{2}\eta}\sq \Pbg+\hkii f \dq  +\frac{1}{2}\NN _{\tilde{b}^{2}\eta}\st \Pbg+\hkii+f\dt
\end{align}

Remembering that $\s_{\tilde{b}^{2}\eta}\st 0 \dt=0$ since $\eta$ is scalar flat and  
\begin{equation}
\mathbb{L}_{\tilde{b}^{2}\eta}=\frac{1}{\tilde{b}^{4}}\Delta_{\eta}^{2}
\end{equation}
because $\eta$ is also Ricci-flat we can rewrite equation \eqref{eq: modellogrezza} in terms of the unknown $f$
\begin{align} \label{eq: modellogrezza2}
\Delta_{\eta}^{2}\sq f  \dq=&-\varepsilon^{2}\tilde{b}^{4}\st 2s_{\omega}+ s_{\bf{0},\bg,\cg,\hg,\kg}\dt-\Delta_{\eta}^{2}\sq \Pbg+ \hkii\dq+\tilde{b}^{4}\NN_{\tilde{b}^{2}\eta}\st \Pbg+\hkii+f  \dt\,.
\end{align}

\vspace{10pt}
 \begin{itemize}

\item[] {\bf Transplanted potential.} As in \cite{Gabor1} and \cite{Gabor2} we introduce the term $\Pbg$ that is a suitable modification of the function $\psi_{\omega}$ defined in Proposition \ref{proprietacsck}. 
We recall that $\psi_{\omega}$ satisfies
\begin{equation}
\s_{eucl}\st \psi_{\omega}\dt=s_{\omega}
\end{equation}
and hence
\begin{equation}
s_{\omega}=-\frac{1}{2}\Delta^{2}\sq \psi_{\omega} \dq+\frac{1}{2}\NN_{eucl}\st\psi_{\omega}  \dt
\end{equation}
in $z$ coordinates on a small ball. Once we perform the rescaling 
\begin{equation}
z=\tilde{b}\varepsilon x
\end{equation}
we consider the function $\varepsilon^{-2}\psi_{\omega}\st \tilde{b}\varepsilon x\dt$ and we have 
\begin{equation}
\varepsilon^{2}s_{\omega}=-\frac{1}{2\tilde{b}^{4}}\Delta^{2}\sq \frac{\psi_{\omega}\st \tilde{b}\varepsilon x \dt}{\varepsilon^{2}}\dq+\frac{1}{2}\NN_{\tilde{b}^{2}\cdot eucl}\st \frac{\psi_{\omega}\st \tilde{b}\varepsilon x \dt}{\varepsilon^{2}}\dt 
\end{equation}

The aim of the transplanted potential is, hence, to cancel the term $\varepsilon^{2}s_{\omega}$ in equation \eqref{eq: modellogrezza2}. Unfortunately the metric associated to $\eta$ is not the euclidean one  so remainder terms appear and the solution $f$ has to correct them, indeed we have 
\begin{equation}
-\frac{1}{2\tilde{b}^{4}}\Delta^{2}\sq \frac{\psi_{\omega}\st \tilde{b}\varepsilon x \dt}{\varepsilon^{2}}\dq+\frac{1}{2}\NN_{\tilde{b}^{2}\cdot eucl}\st \frac{\psi_{\omega}\st \tilde{b}\varepsilon x \dt}{\varepsilon^{2}}\dt = \varepsilon^{2}s_{\omega}+ \textrm{\em remainder terms}\,.
\end{equation}

\begin{remark}
If the remainder terms of the equation above are too large, then the solution $f$ to the equation \eqref{eq: modellogrezza2} becomes too large and it becomes impossible to perform the Cauchy data matching construction. 
\end{remark}
\begin{remark}
The error produced by the term
\begin{equation}
\frac{1}{\varepsilon^{2}}\sum_{k=6}^{+\infty}\Psi_{k}\st\tilde{b}\varepsilon x \dt
\end{equation}
is tolerable, as we will show in the sequel.
\end{remark}

For simplicity we come back to the pre-rescaling expression of $\psi_{\omega}$ and we observe that by Lemma \ref{proprietacsck}

\begin{eqnarray}
\psi_{\omega}  &=& \sum_{k=0}^{+\infty}\Psi_{4+k}\,,\\
-\Delta^{2}\sq\Psi_{4}\dq&=& 2s_{\omega}\,,\\
-\Delta^{2}\sq\Psi_{5}\dq&=&0\,.
\end{eqnarray}

We have to correct the linear error committed by terms $\Psi_{4},\Psi_{5}$ and hence we look for functions $W_{4},W_{5}$ solutions of 
\begin{equation}
\begin{array}{ll}
\Delta_{\eta}^{2}\sq\Psi_{4}+W_{4}\dq&=-2s_{\omega}\\
&\\
\Delta_{\eta}^{2}\sq\Psi_{5}+W_{5}\dq&=0\,.
\end{array}
\end{equation}

We point out that it will be crucial to obtain a description as explicit as possible of $W_{4},W_{5}$. More precisely these corrections will be made of explicit terms and rapidly decaying terms. The first ones will impose constraints on the parameters of the {\em balancing condition} while the latter will be sufficiently small to be handled in the process of Cauchy data matching. The correction $W_{4}$, more precisely precisely one of its  components,  will give  an extra constraint in the {\em balancing condition} and it is responsible for the requirement  \eqref{eq:tuning} in Theorem \ref{maintheorem}.

\begin{notation}
 {\em For the rest of the subsection $\chi$ will denote a smooth cutoff function identically $0$ on $X_{\Gamma,\frac{R_{0}}{3\tilde{b}}}$ and identically $1$ outside $X_{\Gamma,\frac{R_{0}}{2\tilde{b}}}$.}
\end{notation}
Using Lemmas \ref{espansioniALE} and \ref{proprietacsck} it is easy to see that 
\begin{equation}
\begin{array}{ll}
\Delta_{\eta}^{2}\sq \chi \Psi_{4}\dq&=-2s_{\omega}+\st \Phi_{2}+\Phi_{4} \dt\chi|x|^{-2m}+ \mathcal{O}\st |x|^{-2-2m} \dt\\
& \\
\Delta_{\eta}^{2} \sq\chi \Psi_{5}\dq&=\st \Phi_{3}+\Phi_{5} \dt\chi|x|^{1-2m}+ \mathcal{O}\st |x|^{-1-2m} \dt\,.
\end{array}
\end{equation}  
\noindent If we set 
\begin{equation}
\begin{array}{ll}
u_{4}&:= \begin{cases}
\st \frac{\Phi_{2}}{\Lambda_{2}^{2}}+\frac{\Phi_{4}}{\Lambda_{4}^{2}} \dt\chi|x|^{4-2m}&\textrm{for }m\geq3\\
&\\
\st \frac{\Phi_{2}}{\Lambda_{2}^{2}}+\frac{\Phi_{4}}{\Lambda_{4}^{2}} \dt\chi\log\st|x|\dt&\textrm{for }m=2
\end{cases}\\
&\\
u_{5}&:= \st \frac{\Phi_{3}}{\Lambda_{3}^{2}}+\frac{\Phi_{5}}{\Lambda_{5}^{2}} \dt\chi|x|^{5-2m}
\end{array}
\end{equation}

\noindent for a suitable choice of $\Phi_{2},\Phi_{4},\Phi_{3},\Phi_{5}$ eigenfunctions relative to the eigenvalues $\Lambda_{2},\Lambda_{4},\Lambda_{3},\Lambda_{5}$ of $\Delta_{\Sp^{2m-1}}$ , then 
\begin{equation}
\begin{array}{ll}
\Delta_{\eta}^{2}\sq \chi \Psi_{4}+u_{4}\dq&=-2s_{\omega}+ \mathcal{O}\st |x|^{-2-2m} \dt\\
&\\
\Delta_{\eta}^{2}\sq \chi \Psi_{5}+u_{5}\dq&= \mathcal{O}\st |x|^{-1-2m} \dt\,.
\end{array}
\end{equation}

\noindent Now we would like to find $v_{4}\in C_{\delta}^{4,\alpha}\st X_{\Gamma} \dt$ with $\delta\in (2-2m,3-2m)$ and $v_{5}\in C_{\delta}^{4,\alpha}\st X_{\Gamma} \dt$ with $\delta\in (3-2m,4-2m)$ such that 
\begin{align}
\Delta_{\eta}^{2}\sq \chi \Psi_{4}+u_{4}+v_{4}\dq&=-2s_{\omega}\label{eq:psi4}\\
\Delta_{\eta}^{2}\sq \chi \Psi_{5}+u_{5}+v_{5}\dq&= 0\label{eq:psi5}\,.
\end{align}

\noindent Proposition \ref{GAP} tells us that we can find such $v_{4},v_{5}$ if and only if the integrals  
\begin{align}
&\int_{X_{\Gamma}}\st\Delta_{\eta}^{2}\sq \chi \Psi_{4}+u_{4}\dq +2s_{\omega} \dt \, d\mu_{\eta}\label{eq:int4}\\
&\int_{X_{\Gamma}}\Delta_{\eta}^{2}\sq \chi \Psi_{5}+u_{5}\dq\,d\mu_{\eta} \label{eq:int5}
\end{align}
 vanish identically. To check whether those conditions are satisfied we have to compute the two integrals above. The crucial tool for the calculations is Lemma \ref{misuraeuclidea}. We start computing integral \eqref{eq:int4}. By means of divergence Theorem and Lemma \ref{misuraeuclidea} we can write 
\begin{equation}
\int_{X_{\Gamma}}\st\Delta_{\eta}^{2}\sq \chi \Psi_{4}+u_{4}\dq +2s_{\omega} \dt \, d\mu_{\eta}=\lim_{\rho\rightarrow+\infty}\sq \int_{\partial X_{\Gamma,\rho}}\partial_{\nu}\Delta_{\eta}\st  \chi \Psi_{4}\dt d\mu_{\eta}+\frac{s_{\omega}| \Sp^{2m-1} | }{m\left| \Gamma\right|}\rho^{2m}\dq\,,
\end{equation} 
\noindent with $\nu$ outward unit normal to the boundary. We point out that $u_{4}$ doesn't appear in the right hand side of the equation above because the boundary term produced by the integration by parts tends to zero as $\rho$ tends to infinity, and this is an immediate consequence of Lemma \ref{espansioniALE} and the fact that $u_{4}$ has zero mean on every euclidean sphere. Then using Proposition \ref{proprietacsck} and Lemma \ref{espansioniALE}
\begin{align}
\partial_{\nu}\Delta_{\eta}\sq\Psi_{4}\dq\left.d\mu_{\eta}\right|_{\partial X_{\Gamma,\rho}} =& \sq -\frac{s_{\omega}}{m} \rho^{2m}- \frac{4  \cga \st m-1 \dt^{2}s_{\omega}}{m\st m+1 \dt}\dq \left.d\mu_{0}\right|_{\Sp^{2m-1} / \Gamma }\\
&+\sq\mathcal{O}\st 1\dt\st \Phi_{2}+\Phi_{4}\dt +\mathcal{O}\st\frac{1}{\rho}\dt\dq \left.d\mu_{0}\right|_{\Sp^{2m-1} / \Gamma }\,,
\end{align}
and integrating we obtain
\begin{equation}
\int_{X_{\Gamma}}\st\Delta_{\eta}^{2}\sq \chi \Psi_{4}+u_{4}\dq +2s_{\omega} \dt \, d\mu_{\eta} = - \frac{4  \cga \st m-1 \dt^{2}| \Sp^{2m-1} | s_{\omega}}{m\st m+1 \dt |\Gamma|}\,.
\end{equation}
\noindent this shows that equation \eqref{eq:psi4} cannot be solved in general for $v_{4}\in C_{\delta}^{4,\alpha}\st X_{\Gamma} \dt$ with $\delta\in (2-2m,3-2m)$. To overcome this difficulty we add an explicit function which belongs approximately to $\ker\st \Delta_{\eta}^{2} \dt$, more precisely we can solve the equation
\begin{displaymath}
\begin{array}{lcl}
\Delta_{\eta}^{2}\sq \chi \Psi_{4}+u_{4}+\frac{  \cga \st m-1 \dt s_{\omega} }{2\st m-2 \dt   m\st m+1 \dt }\chi |x|^{4-2m}+v_{4}\dq&=&-2s_{\omega}\qquad \textrm{ for }m\geq3\\
&&\\
\Delta_{\eta}^{2}\sq \chi \Psi_{4}+u_{4}-\frac{  \cga  s_{\omega} }{6 }\chi \log\st|x|\dt+v_{4}\dq&=&-2s_{\omega}\qquad \textrm{ for }m=2
\end{array}
\end{displaymath}
\noindent for $v_{4}\in C_{\delta}^{4,\alpha}\st X \dt$ with $\delta\in (2-2m,3-2m)$. In a completely analogous way we can compute integral \eqref{eq:int5} that vanishes identically  and so we can solve the equation 
\begin{equation}
\Delta_{\eta}^{2}\sq \chi \Psi_{5}+u_{5}+v_{5}\dq= 0\,.
\end{equation}
\noindent for $v_{5}\in C_{\delta}^{4,\alpha}\st X \dt$ with $\delta\in (3-2m,4-2m)$. Now we can write the explicit expression of $W_{4}$
\begin{equation}\label{eq:W4}
W_{4}:=\begin{cases}
\frac{  \cga \st m-1 \dt s_{\omega} }{2\st m-2 \dt   m\st m+1 \dt }\chi |x|^{4-2m}+u_{4}+v_{4}&\qquad \textrm{ for }m\geq3\,,\\
&\\
-\frac{  \cga  s_{\omega} }{6 }\chi \log\st|x|\dt+u_{4}+v_{4} & \qquad \textrm{ for }m=2\,.
\end{cases}\,.
\end{equation}
The structure of the function $W_{4}$ deserves a word of comment, the function $v_{4}$ is what we call the rapidly decaying term, $u_{4}$ has a ``critical" decaying rate but it has no radial components with respect to Fourier decomposition reative to $\Delta_{\Sp^{2m-1}}$ and hence it will be handled by {\em pseudo-boundary data} in the Cauchy data matching, the remaning term is the one that will constrain the coefficients of the {\em balancing condition}.
\begin{remark}
The term $|x|^{4-2m}$ (respectively $\log\st |x| \dt$) in formula \eqref{eq:W4} plays a crucial role in our procedure, not only it is necessary for creating function on $X_{\Gamma}$  that rapidly decays  $\Psi_{4}$ at infinity, but also influence the {\em balancing condition}. It forces, indeed, to require condition \eqref{eq:tuning} in Theorem \ref{maintheorem}. In  Subsection \ref{tuning}, we will see that, in order to be able to perform the data matching procedure, we will  have to match perfectly ({\em tuning} procedure) the terms of the potential at inifinity of $\eta_{\tilde{b},\tilde{h},\tilde{k}}$ decaying as $|x|^{4-2m}$ and $|x|^{2-2m}$ with   the {\em principal asymptotics } of  the potential of $\omega_{{\bf 0},\bg,\cg,\hg,\kg}$ that are the terms exploding as $|z|^{2-2m}$ and $|z|^{4-2m}$. We will  do this by making a specific choice for the parameters $\bg$ and $\cg$  and as a consequence we will get the key condition \ref{eq:tuning} of Theorem \ref{maintheorem}.   
\end{remark}

 Contrarily to the case of $\Psi_{4}$ the correction of $\Psi_{5}$ is much easier indeed it is easy to see, using Lemma \ref{espansioniALE} and the fact that $u_{5}$ has no radial component with respect to the fourier decomposition relative to $\Delta_{\Sp^{2m-1}}$, that 
\begin{equation}
\lim_{\rho\rightarrow +\infty}\int_{\partial X_{\Gamma,\rho}}\partial_{\nu}\Delta_{\eta}\sq \chi\Psi_{5}+u_{5} \dq=0
\end{equation}
and hence it is sufficient to apply Proposition \ref{GAP} to find $v_{5}$. The function $W_{5}$ is then
\begin{equation}
W_{5}:=u_{5}+v_{5}
\end{equation}
and as for $W_{4}$ the function $v_{5}$ is a rapidly decaying term and $u_{4}$ has also a ``critical" decaying rate but it has no radial components with respect to Fourier decomposition reative to $\Delta_{\Sp^{2m-1}}$ and hence it will be handled by {\em pseudo-boundary conditions} in the Cauchy data matching. If we define  
\begin{equation}
V:=\varepsilon^{2}\tilde{b}^{4}W_{4}+\varepsilon^{3}\tilde{b}^{5}W_{5} \,.
\end{equation}
\noindent then we can define the transplanted potential $\Pbg$ as the function in $C^{4,\alpha}\st X_{\Gamma,\frac{\Rep}{\tilde{b}}} \dt$
\begin{equation}
\Pbg:=\begin{cases}
\frac{1}{\varepsilon^{2}}\chi\psi_{\omega}\st \tilde{b}\varepsilon x \dt+V&\qquad \textrm{ for }m\geq3\,,\\
&\\
\frac{1}{\varepsilon^{2}}\chi\psi_{\omega}\st \tilde{b}\varepsilon x \dt+V+C& \qquad \textrm{ for }m=2\,.
\end{cases}\label{eq:transplantedmod} 
\end{equation}
where $C$ is the constant term in the expansion  at $B_{2r_{0}}\st p \dt\setminus B_{\rep}\st p \dt$ of  
\begin{equation}
F_{{\bf 0},\bg,\cg,\hg,\kg}^{out}= - \varepsilon^{2m}{{\bf{G}}}_{\bf{0},\bg,\cg}+\Pbe+\hko+ f_{\bf{0},\bg,\cg,\hg,\kg}^{out}\,.
\end{equation}
introduced in Proposition \ref{crucialbase}. As we will see in Section \ref{matching} the coefficient $\tilde{b}$  is very important and it will force the choice of particular values for the parameters $\bg,\cg$ we used on $M$ to construct $F_{{\bf 0},\bg,\cg,\hg,\kg}^{out}$, and in particular of the skeleton ${{\bf{G}}}_{\bf{0},\bg,\cg}$.

\vspace{10pt}

\item[] {\bf Extensions of pseudo-boundary data.} Using euclidean inner biharmonic extensions of functions on the sphere we want to build a function on $X_{\Gamma}$ that is ``almost'' in the kernel of $\Delta_{\eta}^{2}$. We note that 
\begin{equation}
\begin{array}{ll}
\Delta_{\eta}^{2}\sq\chi|x|^{2}\dq&=\, \mathcal{O}\st |x|^{-2-2m} \dt\,,\\
&\\
\Delta_{\eta}^{2}\sq \chi|x|^{2}\Phi_{2}\dq&=\,\mathcal{O}\st |x|^{-2m-2} \dt\,,\\
&\\
\Delta_{\eta}^{2}\sq\chi|x|^{3}\Phi_{3}\dq&=\, \chi|x|^{-1-2m} \Phi_{3} +\mathcal{O}\st|x|^{-3-2m}\dt\,.
\end{array}
\end{equation}

\noindent As for the transplanted potential we want to correct the functions on the left hand sides of equations in such a way they are in $\ker\st \Delta_{\eta}^{2} \dt$. Precisely we want to solve the equations 
\begin{equation}
\begin{array}{ll}
\Delta_{\eta}^{2}\sq\chi|x|^{2}+v^{(0)}\dq&= 0\,, \\
&\\
\Delta_{\eta}^{2}\sq \chi|x|^{2}\Phi_{2}+v^{(2)}\dq&=0\,,\\
&\\
\Delta_{\eta}^{2}\sq\chi|x|^{3}\Phi_{3}+ u^{(3)} +v^{(3)}\dq&= 0\,.
\end{array}
\end{equation}

 with $v^{(0)},v^{(2)},v^{(3)}\in C_{\delta}^{4,\alpha}\st X_{\Gamma} \dt$ for $\delta\in (2-2m,3-2m)$ and  
\begin{equation}
u^{(3)}:=\chi|x|^{3-2m} \Phi_{3}
\end{equation}
\noindent for a suitable  spherical harmonic $\Phi_{3}$. The existence of  $v^{(0)},v^{(2)},v^{(3)}$ follows from Proposition \ref{isomorfismopesati}, Lemma \ref{GAP} and 
\begin{equation}
\int_{X_{\Gamma}}\Delta_{\eta}^{2}\sq\chi|x|^{2}\dq\,d\mu_{\eta}=\int_{X_{\Gamma}}\Delta_{\eta}^{2}\sq\chi|x|^{2}\Phi_{2}\dq\,d\mu_{\eta}=\int_{X_{\Gamma}}\Delta_{\eta}^{2}\sq\chi|x|^{3}\Phi_{3}\dq\,d\mu_{\eta}=0  
\end{equation}
\noindent as one can easily check using exactly the same ideas exposed for the transplanted potential. We are ready to define the function $\hkii \in C^{4,\alpha}(X_{\Gamma,\frac{\Rep}{\tilde{b}}})$ 
\begin{align}
\hkii :=& H_{\tilde{h},\tilde{k}}^{in}\st 0\dt+ \chi\st H_{\tilde{h},\tilde{k}}^{in}\st \frac{ \tilde{b} x}{\Rep} \dt-  H_{\tilde{h},\tilde{k}}^{in}\st 0\dt\dt+\frac{\tilde{k}^{(0)}\tilde{b}^2}{4m \Rep^2}v^{(0)}  \\
&+ \st \tilde{h}^{(2)}-\frac{\tilde{k}^{(2)}}{4(m+2)} \dt \frac{\tilde{b}^{2}v^{(2)}}{\Rep^2}+ \st \tilde{h}^{(3)}-\frac{\tilde{k}^{(3)}}{4(m+3)}  \dt  \frac{\tilde{b}^{3}\st u^{(3)}+v^{(3)}\dt}{\Rep^3}\,.\label{eq:thki}
\end{align}

\vspace{10pt}

\item[] {\bf Correction term.} It is the term that ensures the constancy of the scalar curvature of the metric $\eta_{\tilde{b},\tilde{h},\tilde{k}}$ on $X_{\Gamma,\frac{\Rep}{\tilde{b}}}$ and it is a function $\csfii\in \Cdx$  where the space $\Cdx$ is a weighted H\"older space  defined in Subsection \ref{lineareALE}. As in the base orbifold, the function $\csfii$ depends nonlinearly on  $\st \tilde{h},\tilde{k} \dt$ and $\tilde{b}$ and we find it by solving a fixed point problem on a suitable closed and bounded subspace of $\Cdx$.
\end{itemize} 

\noindent We are now ready to state the main result on the model spaces.

\begin{prop}\label{crucialmodello}

Let $(X_{\Gamma},h,\eta )$ an ALE  Ricci-Flat K\"ahler resolution of an isolated quotient singularity.

\begin{itemize}
\item    Let $\delta\in (4-2m,5-2m)$. Given any $\st \tilde{h},\tilde{k} \dt\in \mathcal{B}$, such that $\st \varepsilon^{2}\tilde{h},\varepsilon^{2}\tilde{k} \dt\in \dombd$, where $\dombd$ is the space defined in formula \eqref{eq:dombd},  let $\hkii$ be the function defined in formula \eqref{eq:thki}. 
\begin{align}
\hkii :=& H_{\tilde{h},\tilde{k}}^{in}\st 0\dt+ \chi\st H_{\tilde{h},\tilde{k}}^{in}\st \frac{ \tilde{b} x}{\Rep} \dt-  H_{\tilde{h},\tilde{k}}^{in}\st 0\dt\dt+\frac{\tilde{k}^{(0)}\tilde{b}^2}{4m \Rep^2}v^{(0)}  \\
&+ \st \tilde{h}^{(2)}-\frac{\tilde{k}^{(2)}}{4(m+2)} \dt \frac{\tilde{b}^{2}v^{(2)}}{\Rep^2}+ \st \tilde{h}^{(3)}-\frac{\tilde{k}^{(3)}}{4(m+3)}  \dt  \frac{\tilde{b}^{3}\st u^{(3)}+v^{(3)}\dt}{\Rep^3}\,.
\end{align}
\item Let $\Pbg$ be the transplanted potential defined in formula \eqref{eq:transplantedmod}
\begin{equation}
\Pbg:=\begin{cases}
\frac{1}{\varepsilon^{2}}\chi\psi_{\omega}\st \tilde{b}\varepsilon x \dt+V&\qquad \textrm{ for }m\geq3\,,\\
&\\
\frac{1}{\varepsilon^{2}}\chi\psi_{\omega}\st \tilde{b}\varepsilon x \dt+V+C& \qquad \textrm{ for }m=2\,.
\end{cases} 
\end{equation}
\end{itemize}
Then there is $\csfii\in \Cdx$  such that 
\begin{equation}
\etat=\tilde{b}^{2}\eta+i\dd\st \Pbg+ \hkii +\csfii  \dt
\end{equation}
is a Kcsc metric on $X_{\Gamma,\frac{\Rep}{\tilde{b}}}$ and the following estimates hold.
\begin{equation}
\left\|\csfii\right\|_{\Cdx}\leq C\st\kappa \dt \varepsilon^{2m+4}\rep^{-4m-\delta}\Rep^{-2}
\end{equation}
with $C\st \kappa \dt\in \RR^{+}$ depending only on $\omega$ and $\eta_{j}$'s and $\kappa$ the constant appearing in the definition of $\dombd$ (Section \ref{biharmonicext} formula \ref{eq:dombd} ). Moreover $s_{\etat}$, the scalar curvature of $\etat$ is 
\begin{equation}
s_{\etat}=s_{\omega_{{\bf 0},\bg,\hg,\kg}}=s_{\omega}+\frac{1}{2}s_{{\bf 0},\bg,\hg,\kg}\,.
\end{equation}
\end{prop}

\noindent As in the base orbifold case, we set up a fixed point problem for finding the correction $\csfii$  and we will solve it using Banach-Caccioppoli Theorem. Using the very definition of $\Pbg$ we can rewrite equation \eqref{eq: modellogrezza2} in the following form.
\begin{align}\label{eq: PDEmodello}
\Delta_{\eta}^{2}\sq f  \dq=&-\varepsilon^{2}\tilde{b}^{4}s_{\bf{0},\bg,\cg,\hg,\kg}-\frac{1}{\varepsilon^{2}}\Delta_{\eta}^{2}\sq \sum_{k=6}^{+\infty}\chi\Psi_{k}\st \tilde{b}\varepsilon x \dt\dq-\Delta_{\eta}^{2}\sq \hkii\dq\\
&+\tilde{b}^{4}\NN_{\tilde{b}^{2}\eta}\st \Pbg+\hkii+f  \dt\,.
\end{align}

\noindent In analogy with what we did on the base orbifold, we  look for a PDE defined on the whole $X_{\Gamma}$ and such that on $X_{\Gamma,\frac{\Rep}{\tilde{b}}}$ restricts to the \eqref{eq: PDEmodello}. To this aim we introduce a truncation-extension operator on weighted H\"older spaces

\begin{definition}
Let $f\in C_{\delta}^{0,\alpha} \st X_{\Gamma} \dt $, we define $\mathcal{E}_{\Rep}:C_{\delta}^{0,\alpha} \st X_{\Gamma} \dt \rightarrow C_{\delta}^{0,\alpha} \st X_{\Gamma} \dt $
\begin{displaymath}
\mathcal{E}_{\Rep} \st f \dt :\left\{ \begin{array}{ll}
f \st x \dt  & x\in X_{\frac{\Rep}{\tilde{b}}}   \\
f \st R\frac{x}{|x|} \dt \chi \st \frac{|x|\tilde{b}}{\Rep} \dt  & x\in X_{\frac{2\Rep}{\tilde{b}}}\setminus X_{\frac{\Rep}{\tilde{b}}}\\
0  & x\in X\setminus X_{\frac{2\Rep}{\tilde{b}}}  
\end{array} \right.
\end{displaymath}
with $\chi\in C^{\infty}\st [0,+\infty) \dt$ a  cutoff function that is identically $1$ on $[0,1]$ and identically $0$ on $[2,+\infty)$. 
\end{definition}
Now we use the truncation-extension operator and we find our equation.
\begin{align}
\Delta_{\eta}^{2}\sq f  \dq=&-\varepsilon^{2}\tilde{b}^{4}\mathcal{E}_{\Rep}s_{\bf{0},\bg,\cg,\hg,\kg}-\frac{1}{\varepsilon^{2}}\mathcal{E}_{\Rep}\Delta_{\eta}^{2}\sq \sum_{k=6}^{+\infty}\chi\Psi_{k}\st \tilde{b}\varepsilon x \dt\dq-\mathcal{E}_{\Rep}\Delta_{\eta}^{2}\sq \hkii\dq\\
&+\tilde{b}^{4}\mathcal{E}_{\Rep}\NN_{\tilde{b}^{2}\eta}\st \Pbg+\hkii+f  \dt\,.
\end{align}

\noindent Using the right inverse for $\Delta_{\eta}^{2}$ introduced in Remark \ref{inversamod} formula \eqref{eq:inversamod}  
\begin{equation}
\mathbb{J}^{(\delta)}:C_{\delta-4}^{0,\alpha}\st X_{\Gamma} \dt\rightarrow C_{\delta}^{4,\alpha}\st X_{\Gamma} \dt
\end{equation}
we define the nonlinear operator
\begin{equation}
\mathbb{T}_{\tilde{b}}^{(\delta)}:C_{\delta}^{4,\alpha}\st X_{\Gamma} \dt \times \mathcal{B} \rightarrow C_{\delta}^{4,\alpha}\st X_{\Gamma} \dt
\end{equation}

\begin{align}
\mathbb{T}_{\tilde{b}}^{(\delta)}\st f,\tilde{h},\tilde{k} \dt:=&-\varepsilon^{2}\tilde{b}^{4}\mathbb{J}^{(\delta)}\mathcal{E}_{\Rep}s_{\bf{0},\bg,\cg,\hg,\kg}-\frac{1}{\varepsilon^{2}}\mathbb{J}^{(\delta)}\mathcal{E}_{\Rep}\Delta_{\eta}^{2}\sq \sum_{k=6}^{+\infty}\chi\Psi_{k}\st \tilde{b}\varepsilon x \dt\dq-\mathbb{J}^{(\delta)}\mathcal{E}_{\Rep}\Delta_{\eta}^{2}\sq \hkii\dq\\
&+\tilde{b}^{4}\mathbb{J}^{(\delta)}\mathcal{E}_{\Rep}\NN_{\tilde{b}^{2}\eta}\st \Pbg+\hkii+f  \dt\,.
\end{align}

\noindent We prove the existence of a solution of equation \eqref{eq: PDEmodello} finding, for fixed $\st \tilde{h},\tilde{k} \dt\in \mathcal{B}$ such that $\st \varepsilon^{2}\tilde{h},\varepsilon^{2}\tilde{k} \dt\in \dombd$, a fixed point of the operator 
\begin{equation}
\mathbb{T}_{\tilde{b}}^{(\delta)}\st \cdot,\tilde{h},\tilde{k} \dt: C_{\delta}^{4,\alpha}\st X_{\Gamma} \dt  \rightarrow C_{\delta}^{4,\alpha}\st X_{\Gamma} \dt
\end{equation}

following exactly the same strategy we used on the base orbifold. We need to find a domain $\Omega\subseteq C_{\delta}^{4,\alpha}\st X_{\Gamma} \dt$ such that for any $f\in \Omega$ then  $\mathbb{T}_{\tilde{b}}^{(\delta)}\st f,\tilde{h},\tilde{k} \dt\in \Omega$ and $\mathbb{T}_{\tilde{b}}^{(\delta)}\st \cdot,\tilde{h},\tilde{k} \dt$ is a contraction on $\Omega$.To decide what kind of domain will be our $\Omega$ we need some informations on the behavior of $\mathbb{T}_{\tilde{b}}^{(\delta)}$ that we find in the following two lemmas.

\begin{lemma}\label{bihmod} 
Under the assumptions of Proposition \ref{crucialbase} the following estimate holds
\begin{equation}
\left\|\mathcal{E}_{\Rep}\Delta_{\eta}^{2}\sq\hkii\dq\right\|_{\Ccx}\leq  \frac{ \mathsf{C} }{\Rep^4}\left\|\tilde{h}^{(\dagger)},\tilde{k}^{(\dagger)}\right\|_{\mathcal{B}}=\kappa \mathsf{C} \varepsilon^{2m+2}\rep^{2-4m-\delta}\Rep^{-4}\,.
\end{equation}
\end{lemma}
\begin{proof}
Using formula \eqref{eq:thki} we have 
\begin{align}
\Delta_{\eta}^{2}\sq \hkii\dq =& \Delta_{\eta}^{2}\sq H_{\tilde{h},\tilde{k}}^{in}\st 0\dt+ \chi\st H_{\tilde{h},\tilde{k}}^{in}\st \frac{ \tilde{b} x}{\Rep} \dt- \chi H_{\tilde{h},\tilde{k}}^{I}\st 0\dt\dt\dq\\
&+\Delta_{\eta}^{2}\sq \frac{\tilde{k}^{(0)}\tilde{b}^2}{4m \Rep^2}v^{(0)}  + \st \tilde{h}^{(2)}-\frac{\tilde{k}^{(2)}}{4(m+2)} \dt \frac{\tilde{b}^{2}v^{(2)}}{\Rep^2}+ \st \tilde{h}^{(3)}-\frac{\tilde{k}^{(3)}\tilde{b}^{3}}{4(m+3)}  \dt  \frac{u^{(3)}+v^{(3)}}{\Rep^3}\dq\\
=&\st \Delta_{\eta}^{2}-\Delta^{2} \dt\sq \frac{\tilde{k}^{(2)}}{4(m+2)}\chi\left|\frac{\tilde{b}  x}{\Rep}\right|^{4}  \phi_{2}  + \frac{\tilde{k}^{(\gamma)}}{4(m+3)}\chi\left|\frac{\tilde{b}  x}{\Rep}\right|^{5}  \phi_{3}  \dq\\
&+\st \Delta_{\eta}^{2}-\Delta^{2} \dt\sq   \chi\sum_{\gamma=4}^{+\infty}\left(\left(\tilde{h}^{(\gamma)}-\frac{\tilde{k}^{(\gamma)}}{4(m+\gamma)}  \dt \left|\frac{\tilde{b} x}{\Rep}\right|^{\gamma}+\frac{\tilde{k}^{(\gamma)}}{4(m+\gamma)}\left|\frac{\tilde{b}  x}{\Rep}\right|^{\gamma+2} \dt \phi_{\gamma}   \dq
\end{align}
and so we deduce that
\begin{equation}\left\|\Delta_{\eta}^{2} \sq\hkii\dq\right\|_{C^{0,\alpha}\st X_{\Gamma,\frac{R_0}{\tilde{b}}} \dt}\leq  \frac{\mathsf{C}}{\Rep^4}\left\|\tilde{h}^{(\dagger)},\tilde{k}^{(\dagger)}\right\|_{\mathcal{B}}\,.\end{equation}
Now we estimate the quantity
\begin{equation}\sup_{\rho\in [R_0,\Rep]}\rho^{-\delta+4}\left\| \Delta_{\eta}^{2}\sq \hkii\dq\right\|_{C^{0,\alpha}\st  B_{1}\setminus B_{\frac{1}{2}}  \dt}\,.\end{equation}
Using again formula \eqref{eq:thki}, we have  

\begin{align}
\Delta_{\eta}^{2}\sq \hkii\dq =&  \st \Delta_{\eta}^{2}-\Delta^{2} \dt\sq \frac{\tilde{k}^{(2)}}{4(m+2)}\chi\left|\frac{\tilde{b}  x}{\Rep}\right|^{4}  \phi_{2}  + \frac{\tilde{k}^{(\gamma)}}{4(m+3)}\chi\left|\frac{\tilde{b}  x}{\Rep}\right|^{5}  \phi_{3}  \dq\\
&+\st \Delta_{\eta}^{2}-\Delta^{2} \dt\sq   \chi\sum_{\gamma=4}^{+\infty}\left(\left(\tilde{h}^{(\gamma)}-\frac{\tilde{k}^{(\gamma)}}{4(m+\gamma)}  \dt \left|\frac{\tilde{b} x}{\Rep}\right|^{\gamma}+\frac{\tilde{k}^{(\gamma)}}{4(m+\gamma)}\left|\frac{\tilde{b}  x}{\Rep}\right|^{\gamma+2} \dt \phi_{\gamma}   \dq\,.
\end{align}

\noindent So we have 

\begin{equation}
\sup_{\rho\in [R_0,\Rep]}\rho^{-\delta+4}\left\| \Delta_{\eta}^{2} \sq\hkii\dq  \right\|_{C^{0,\alpha}\st   B_{1}\setminus B_{\frac{1}{2}}   \dt}\leq  \frac{\mathsf{C}}{\Rep^{\delta+2m}}\left\|\tilde{h}^{(\dagger)},\tilde{k}^{(\dagger)}\right\|_{\mathcal{B}}\,.
\end{equation}
and therefore the lemma is proved.
\end{proof}

\begin{lemma}\label{transmod}
Under the assumptions of Proposition \ref{crucialbase} the following estimate holds
\begin{equation}\left\| \mathbb{T}_{\tilde{b}}^{(\delta)}\st 0, 0,0 \dt \right\|_{\Ccx}\leq \mathsf{C}\varepsilon^{4}\Rep^{6-2m-\delta}\,.\end{equation}
\end{lemma}
\begin{proof}
We will prove the lemma for the case $m\geq 3$, for the case $m=2$ the proof is identical. By the very definition of $\Pbg$, on $X_{ \Gamma,\frac{R_{0}}{\tilde{b}}  }$, we have    
\begin{equation}
-\varepsilon^{2}\tilde{b}^{4}s_{\omega}-\frac{1}{2}\Delta_{\eta}^{2}\sq\Pbg\dq+\frac{1}{2}\mathbb{N}_{\tilde{b}^{2}\eta}\st \Pbg \dt =-\frac{1}{2}\Delta_{\eta}^{2}\sq \sum_{k=6}^{+\infty}\frac{\tilde{b}^{k}}{\varepsilon^{k-2}}\chi\Psi_{k} \dq+\mathbb{N}_{\tilde{b}^{2}\eta}\st \Pbg \dt
\end{equation}
and so 
\begin{equation}
\left\|-\varepsilon^{2}\tilde{b}^{4}s_{\omega}-\frac{1}{2}\Delta_{\eta}^{2}\sq\Pbg\dq+\frac{1}{2}\mathbb{N}_{\tilde{b}^{2}\eta}\st \Pbg \dt\right\|_{C^{0,\alpha}\st X_{\Gamma,\frac{R_0}{\tilde{b}}} \dt} \leq \mathsf{C} \varepsilon^4\,.
\end{equation}

Now we estimate the weighted part of the norm. On $X_{\frac{\Rep}{\tilde{b}}}\setminus X_{\frac{R_0}{ 2 \tilde{b}}}$ we have
\begin{equation}
-\varepsilon^{2}\tilde{b}^{4}s_{\omega}-\frac{1}{2}\Delta_{\eta}^{2}\sq\Pbg\dq+\NN _{\tilde{b}^{2}\eta}\st \Pbg \dt=-\varepsilon^{2}\tilde{b}^{4}s_{\omega}-\frac{1}{2}\Delta^{2}\sq \psi_{\omega}\dq+\NN_{\tilde{b}^{2}\,eucl}\st \psi_{\omega} \dt+\textrm{\bf I}+\textrm{\bf II}+\textrm{\bf III} 
\end{equation}

with 
\begin{equation}
\begin{array}{cl}
\textrm{\bf I}&= -\frac{1}{2}\st \Delta_{\eta}^{2}-\Delta^{2} \dt\sq \sum_{k=6}^{+\infty}\frac{\tilde{b}^{k}}{\varepsilon^{k-2}}\Psi_{k} \dq\\
&\\
\textrm{\bf II}&=  \NN_{\tilde{b}^{2}\eta}\st \Pbg \dt-\mathbb{N}_{\tilde{b}^{2}\eta}\st  \frac{1}{\varepsilon^{2}}\psi_{\omega}\st \tilde{b}\varepsilon x \dt  \dt\\
&\\
\textrm{\bf III}&= \NN_{\tilde{b}^{2}\eta}\st  \frac{1}{\varepsilon^{2}}\psi_{\omega}\st \tilde{b}\varepsilon x \dt  \dt-\NN _{\tilde{b}^{2}\,eucl}\st  \frac{1}{\varepsilon^{2}}\psi_{\omega}\st \tilde{b}\varepsilon x \dt  \dt
\end{array}
\end{equation}

Using Proposition \ref{proprietacsck}, precisely the algebraic identity
\begin{equation}
-\frac{1}{2}\Delta^{2}\sq \psi_{\omega}\dq+\NN_{eucl}\st \psi_{\omega} \dt=s_{\omega}\,,
\end{equation}
we see that the only remaining terms are $\textrm{\bf I},\textrm{\bf II},\textrm{\bf III}$. With standard, but cumbersome, computations  we obtain  

\begin{equation}
\sup_{ \rho\in [R_0,\Rep]}\rho^{-\delta+4}\st\left\| \textrm{\bf I}\right\|_{C^{0,\alpha}\st  B_{1} \setminus B_{\frac{1}{2}} \dt}+ \left\| \textrm{\bf II}\right\|_{C^{0,\alpha}\st  B_{1} \setminus B_{\frac{1}{2}} \dt} +\left\| \textrm{\bf III}\right\|_{C^{0,\alpha}\st  B_{1} \setminus B_{\frac{1}{2}} \dt}\dt \leq \mathsf{C}\varepsilon^4 \Rep^{6-2m-\delta}\,, 
\end{equation}

\noindent and hence  the lemma is proved.
 \end{proof}

\noindent  

We consider the subset of $\Omega\subset C_{\delta}^{4,\alpha}\st X_{\Gamma} \dt$ with $\delta\in (4-2m,5-2m)$ such that for any $f\in \Omega$
\begin{equation}
\left\| f \right\|_{\Cdx}\leq 2\left\|\mathbb{J}^{(\delta)}\mathcal{E}_{\Rep}\Delta_{\eta}^{2}\sq\hkii\dq\right\|_{\Cdx}\,.
\end{equation}

and we study continuity properties of $\mathbb{T}_{\tilde{b}}^{(\delta)}$ on $\Omega\times \mathcal{B}$.

\begin{lemma}\label{lipschitzmod}
If  $\st\varepsilon^{2}\tilde{h}',\varepsilon^{2}\tilde{k}'\dt\in \dombd$, $f,f'\in  C_{\delta}^{4,\alpha}\st X_{\Gamma} \dt$ 

\begin{equation}\left\|f\right\|_{\Cdx},\left\|f'\right\|_{\Cdx}\leq  2\left\|\mathbb{J}^{(\delta)}\mathcal{E}_{\Rep}\Delta_{\eta}^{2}\sq\hkii\dq\right\|_{\Ccx} \end{equation}
and assumptions of Proposition \ref{crucialmodello} are satisfied, then the following estimates hold:

\begin{align}
\left\|\mathbb{T}_{\tilde{b}}^{(\delta)}\st f,\tilde{h},\tilde{k} \dt - \mathbb{T}_{\tilde{b}}^{(\delta)}\st 0,0,0 \dt\right\|_{\Ccx}\,\leq&\, \frac{3}{2} \left\|\mathbb{J}^{(\delta)}\mathcal{E}_{\Rep}\Delta_{\eta}^{2}\sq\hkii\dq\right\|_{\Ccx}\\
\left\|\mathbb{T}_{\tilde{b}}^{(\delta)}\st f,\tilde{h},\tilde{k} \dt - \mathbb{T}_{\tilde{b}}^{(\delta)}\st f',\tilde{h},\tilde{k} \dt\right\|_{\Ccx}\,\leq&\, \frac{1}{2} \left\|f-f'\right\|_{\Cdx}\\
\left\|\mathbb{T}_{\tilde{b}}^{(\delta)}\st f,\tilde{h},\tilde{k} \dt - \mathbb{T}_{\tilde{b}}^{(\delta)}\st f,\tilde{h}',\tilde{k}' \dt\right\|_{C_{\delta}^{4,\alpha}\st X_{\Gamma} \dt}\,\leq&\, \frac{1}{2} \left\|\tilde{h}-\tilde{h}',\tilde{k}-\tilde{k}'\right\|_{\mathcal{B}}\,.
\end{align}

\end{lemma}
\begin{proof}
Follows by direct computations as \cite[Lemma 5.3]{ap2}
\end{proof}

Now Proposition \ref{crucialbase} easily follows from Lemma \ref{bihmod}, Lemma \ref{transmod},  Lemma \ref{lipschitzmod}.

\section{Data matching}\label{matching}

Now that we have the families of metrics on the base orbifold and on model spaces we want to glue them. To perform the data matching construction we will rescale all functions involved in such a way that functions on $X_{\Gamma_{j}}$ are  functions on the annulus $\overline{B_{1}}\setminus B_{\frac{1}{2}}$ and functions on $M$ are functions on the annulus $\overline{B_{2}}\setminus B_{1}$. The main technical tool we will use in this section is the ``Dirichet to Neumann'' map for euclidean biharmonic extensions that we introduce with the following Theorem whose proof can be found in \cite[Lemma 6.3]{ap1}.

\begin{teo} \label{dirneu}
The map
\begin{equation}\mathcal{P}:C^{4,\alpha} \st \Sp^{2m-1} \dt \times C^{2,\alpha} \st \Sp^{2m-1} \dt \rightarrow C^{3,\alpha} \st \Sp^{2m-1} \dt \times C^{1,\alpha} \st \Sp^{2m-1} \dt \end{equation} 
\begin{equation}\mathcal{P} \st h,k \dt = \st \partial_{|w|} \st H_{h,k}^{out}-H_{h,k}^{in} \dt ,\partial_{|w|}\Delta \st H_{h,k}^{out}-H_{h,k}^{in} \dt  \dt \end{equation}
is an isomorphism of Banach spaces with inverse $\mathcal{Q}$.
\end{teo}

\begin{proof}[Proof of Theorem \ref{maintheorem}]: We carry on the proof for the case $m\geq 3$, for $m=2$ it is identical. Let $\mathcal{V}^{out}_{j,\bf{0},\bg,\cg,\hg,\kg}$ be K\"ahler potential of $\omega_{\bf{0},\bg,\cg,\hg,\kg}$ at the annulus $\overline{B_{2\rep}\st p_{j} \dt}\setminus B_{\rep}\st p_{j} \dt$ under the homothety
\begin{equation}
z=\rep w\,.
\end{equation}
We have then the expansion 
\begin{align}
\mathcal{V}^{out}_{j,\bf{0},\bg,\cg,\hg,\kg}=&\frac{\rep^{2}}{2}|w|^{2}+\psi_{\omega}\st \rep w \dt+\varepsilon^{2}\psi_{\eta_{j}}\st \frac{\rep w}{B_{j}	\varepsilon} \dt\\
&+\st 1-\frac{\st f_{{\bf 0},\bg,\cg,\hg,\kg}^{out}\dt^{j}}{\varepsilon^{2m}} \dt\st-c\st \Gamma_{j} \dt B_{j}^{2m}\varepsilon^{2m}\rep^{2-2m}|w|^{2-2m}+C_{j}\varepsilon^{2m}\rep^{4-2m}|w|^{4-2m}\dt\\
&+\hkoid\\
&-\sq {\varepsilon^{2m}\bf{G}}_{\bf{0},\bg,\cg} -c\st \Gamma_{j} \dt B_{j}^{2m}\varepsilon^{2m}\rep^{2-2m}|w|^{2-2m}+C_{j}\varepsilon^{2m}\rep^{4-2m}|w|^{4-2m} \dq\\
&+\sq f_{{\bf 0},\bg,\cg,\hg,\kg}^{out}+\st f_{{\bf 0},\bg,\cg,\hg,\kg}^{out}\dt^{j}\st c\st \Gamma_{j} \dt B_{j}^{2m}\rep^{2-2m}|w|^{2-2m}-C_{j}\rep^{4-2m}|w|^{4-2m}\dt\dq\,.
\end{align}
For the sake of notation we set
\begin{align}
\mathcal{R}_{j}^{out}:=&-\sq {\varepsilon^{2m}\bf{G}}_{\bf{0},\bg,\cg} -c\st \Gamma_{j} \dt B_{j}^{2m}\varepsilon^{2m}\rep^{2-2m}|w|^{2-2m}+C_{j}\varepsilon^{2m}\rep^{4-2m}|w|^{4-2m} \dq\\
&+\sq f_{{\bf 0},\bg,\cg,\hg,\kg}^{out}+\st f_{{\bf 0},\bg,\cg,\hg,\kg}^{out}\dt^{j}\st c\st \Gamma_{j} \dt B_{j}^{2m}\rep^{2-2m}|w|^{2-2m}-C_{j}\rep^{4-2m}|w|^{4-2m}\dt\dq\,.
\end{align}

We recall that, using notations of Theorem \ref{invertibilitapesatodef}, $f_{{\bf 0},\bg,\cg,\hg,\kg}^{out}\in C_{\delta}^{4,\alpha}\st M_{\p} \dt\oplus \mathcal{D}\st \bg,\cg \dt$ and we have 
\begin{equation}
f_{{\bf 0},\bg,\cg,\hg,\kg}^{out}=\tilde{f}_{{\bf 0},\bg,\cg,\hg,\kg}^{out}+\sum_{j=1}^{N}\st f_{{\bf 0},\bg,\cg,\hg,\kg}^{out}\dt^{j}W_{\bg,\cg}^{j}
\end{equation}
with $\tilde{f}_{{\bf 0},\bg,\cg,\hg,\kg}^{out}\in C_{\delta}^{4,\alpha}\st M_{\p} \dt$ for $\delta\in (4-2m,5-2m)$ and the numbers $\st f_{{\bf 0},\bg,\cg,\hg,\kg}^{out}\dt^{j}$'s are the coefficients of the deficiency components of $f_{{\bf 0},\bg,\cg,\hg,\kg}^{out}$. In writing the expansion of $\mathcal{V}^{out}_{j,\bf{0},\bg,\cg,\hg,\kg}$, precisely in the second and fourth lines ,  we used the only {\em principal asymptotics} of $\st f_{{\bf 0},\bg,\cg,\hg,\kg}^{out}\dt^{j}W_{\bg,\cg}^{j}$ exposed in formula \eqref{eq:Wjbc} while the remaining part falls in the remainder term $\mathcal{R}_{j}^{out}$.

Let also  $\mathcal{V}^{in}_{j,\tilde{b}_{j},\tilde{h}_{j},\tilde{k}_{j}}$ be the K\"ahler potential  of $\varepsilon^{2}\eta_{j,\tilde{b}_{j},\tilde{h}_{j},\tilde{k}_{j}}$ at the annulus $\overline{X_{\Gamma_{j},\frac{\Rep}{\tilde{b}_{j}}}}\setminus X_{\Gamma_{j},\frac{\Rep}{2\tilde{b}_{j}}}$ under the homothety
\begin{equation}
x=\frac{\Rep w}{\tilde{b}_{j}}\,.
\end{equation}
We have the expansion
\begin{align}
\mathcal{V}^{in}_{j,\tilde{b}_{j},\tilde{h}_{j},\tilde{k}_{j}}=&\frac{\varepsilon^{2}\Rep^{2}}{2}|w|^{2}+\varepsilon^{2}B^{2}\psi_{\eta_{j}}\st \frac{\Rep w}{B_{j}} \dt+\psi_{\omega}\st \varepsilon \Rep w \dt\\
&  -c\st \Gamma_{j} \dt \tilde{b}_{j}^{2m}\varepsilon^{2}\Rep^{2-2m}|w|^{2-2m}+\frac{  \cgaj \st m-1 \dt s_{\omega}\tilde{b}^{4}\varepsilon^{4}\Rep^{4-2m} }{2\st m-2 \dt   m\st m+1 \dt } |w|^{4-2m}\\
&+H_{\varepsilon^{2}\tilde{h}_{j},\varepsilon^{2}\tilde{k}_{j}}^{in}\\
&+\sq \varepsilon^{2}\Pbg-\psi_{\omega}\st \varepsilon \Rep w \dt-\frac{  \cgaj \st m-1 \dt s_{\omega}\tilde{b}^{4}\varepsilon^{4}\Rep^{4-2m} }{2\st m-2 \dt   m\st m+1 \dt } |w|^{4-2m} \dq\\
&+\sq\varepsilon^{2}\tilde{b}_{j}^{2}\psi_{\eta_{j}}\st \frac{\Rep w}{\tilde{b}_{j}} \dt-\varepsilon^{2}B_{j}^{2}\psi_{\eta_{j}}\st \frac{\Rep w}{B_{j}} \dt\dq+\sq {\bf H}_{\varepsilon^{2}\tilde{h}_{j},\varepsilon^{2}\tilde{k}_{j}}^{in}-H_{\varepsilon^{2}\tilde{h}_{j},\varepsilon^{2}\tilde{k}_{j}}^{in} \dq\\
&+\varepsilon^{2}f_{\tilde{b}_{j},\tilde{h}_{j},\tilde{k}_{j}}^{in}\,. 
\end{align}
For the sake of notation we set
\begin{align}
\mathcal{R}_{j}^{in}:=&\sq \varepsilon^{2}\Pbg-\psi_{\omega}\st \varepsilon \Rep w \dt-\frac{  \cgaj \st m-1 \dt s_{\omega}\tilde{b}^{4}\varepsilon^{4}\Rep^{4-2m} }{2\st m-2 \dt   m\st m+1 \dt } |w|^{4-2m} \dq\\
&+\sq\varepsilon^{2}\tilde{b}_{j}^{2}\psi_{\eta_{j}}\st \frac{\Rep w}{\tilde{b}_{j}} \dt-\varepsilon^{2}B_{j}^{2}\psi_{\eta_{j}}\st \frac{\Rep w}{B_{j}} \dt\dq+\sq {\bf H}_{\varepsilon^{2}\tilde{h}_{j},\varepsilon^{2}\tilde{k}_{j}}^{in}-H_{\varepsilon^{2}\tilde{h}_{j},\varepsilon^{2}\tilde{k}_{j}}^{in} \dq\\
&+\varepsilon^{2}f_{\tilde{b}_{j},\tilde{h}_{j},\tilde{k}_{j}}^{in}\,.
\end{align}

We want to find $\bg,\cg,\tilde{\bg},\hg,\kg, \tilde{\hg},\tilde{\kg}$ such that the functions 
\begin{equation}
\mathcal{V}_{j}:=\begin{cases}
\mathcal{V}^{in}_{j,\tilde{b}_{j},\tilde{h}_{j},\tilde{k}_{j}}&\textrm{ on } B_{1}\setminus B_{\frac{1}{2}}\\
&\\
\mathcal{V}^{out}_{j,\bf{0},\bg,\cg,\hg,\kg}&\textrm{ on }\overline{B_{2}}\setminus B_{1}
\end{cases}
\end{equation}
are smooth on $\overline{B_{2}}\setminus B_{\frac{1}{2}}$ for every $j=1,\ldots,N$.
We have written the expansions of  $\mathcal{V}^{out}_{j,\bf{0},\bg,\cg,\hg,\kg}$'s and $\mathcal{V}^{in}_{j,\tilde{b}_{j},\tilde{h}_{j},\tilde{k}_{j}}$'s in such a way we can see immediately perfectly matched terms in the first rows,    {\em principal asymptotics} in the second rows,  biharmonic extensions of {\em pseudo-boundary data} in the third rows,  and "remainder" terms.

\subsection{Tuning Procedure.}\label{tuning}
We would like to have that also the {\em principal asymptotics} match perfectly  and biharmonic extensions of {\em pseudo-boundary data} dominate all the "remainder terms" in  $\varepsilon$-growth. Moreover we need to recover a degree of freedom in biharmonic extensions since we have have taken meanless functions $\hg^{(\dagger)}, \kg^{(\dagger)}$ as parameters. To overcome these problems we  have  to perform a ``{\em tuning}" of the {\em principal asymptotics} i.e. we have to set 
\begin{align}\label{eq:force1}
\cgaj  \tilde{b}_{j}^{2m}\varepsilon^{2}\Rep^{2-2m}|w|^{2-2m}=&\st 1-\frac{\st f_{{\bf 0},\bg,\cg,\hg,\kg}^{out}\dt^{j}}{\varepsilon^{2m}} \dt\cgaj  B_{j}^{2m}\varepsilon^{2m}\rep^{2-2m}|w|^{2-2m}\\
&+\st h_{j}^{(0)}+\frac{k_{j}^{(0)}}{4m-8}\dt|w|^{2-2m}\\
\end{align}

\begin{align}\label{eq:force2}
\frac{  \cga \st m-1 \dt s_{\omega}\tilde{b}^{2m}\varepsilon^{4}\Rep^{4-2m} }{2\st m-2 \dt   m\st m+1 \dt } |w|^{4-2m}=& - \st 1-\frac{\st f_{{\bf 0},\bg,\cg,\hg,\kg}^{out}\dt^{j}}{\varepsilon^{2m}} \dt C_{j}\varepsilon^{2m}\rep^{4-2m}|w|^{4-2m}\\
&-\frac{k_{j}^{(0)}}{4m-8}|w|^{4-2m}
\end{align}
With the specialization above we regain the means of functions $h_{j}$ and $k_{j}$. In fact, as we can see from formula \eqref{eq:bihout3}, choosing meanless functions we were missing exactly the radial terms in the Fourier expansion of $H_{h,k}^{out}$ that incidentally have exactly the same growth of the {\em principal asymptotics}. So perturbing a bit the coefficients $b_{j}$'s we can recover these missing asymptotics in the biharmonic extensions but equation \eqref{eq:force1} imposes us to set the value of parameter $\tilde{b}_{j}$. Moreover, we point out that once we have set the value of $\tilde{b}_{j}$ the equation \eqref{eq:force2} imposes us to choose a particular value for the parameter $c_{j}$ and hence we see, as anticipated in Subsection \ref{analisinonlinearemodello}, how the nonlinear analysis on $X_{\Gamma_{j}}$'s constrains the parameters of {\em balancing condition}. We recall that coefficients $B_{j}$ and $C_{j}$ are defined in Section \ref{nonlinbase} respectively by equations \eqref{eq:Bj} and \eqref{eq:Cj}. Conditions above force us to set: 

\begin{equation}\label{eq:sceltaB}
\tilde{b}_{j}^{2m}=B_{j}^{2m}\st 1-\frac{\st f_{{\bf 0},\bg,\cg,\hg,\kg}^{out}\dt^{j}}{\varepsilon^{2m}} \dt+\frac{1}{ \cgaj  }\st h_{j}^{(0)}+\frac{k_{j}^{(0)}}{4m-8}\dt\frac{\rep^{2m-2}}{\varepsilon^{2m}}
\end{equation}

\begin{align}\label{eq:sceltaC}
C_{j}=&-\frac{1}{2\st m-2 \dt\st  \varepsilon^{2m}-\st f_{\bf{0},\bg,\cg,\hg,\kg}^{out}\dt^{j} \dt}\st \frac{  \cgaj \st m-1 \dt s_{\omega}\tilde{b}_{j}^{2m}\varepsilon^{4}\Rep^{4-2m} }{  m\st m+1 \dt }+k_{j}^{(0)} \dt
\end{align}

and hence we must set

\begin{equation}\label{eq:sceltagamma}
c_{j}=s_{\omega}b_{j}
\end{equation}

that is the assumption \eqref{eq:tuning} in Theorem \ref{maintheorem}.

\begin{remark}
 At this point, the presence of  $|x|^{4-2m}$ term in the correction $W_{4}$, introduced in Subsection \ref{analisinonlinearemodello} formula \eqref{eq:W4}, shows its effects.  That term indeed,  introduced as a technical tool for obtaining better estimates, puts now strong geometric constraints on our construction defining the correct form of {\em non degeneracy condition} and  {\em balancing condition} forcing us to impose Equation \eqref{eq:force2} and giving as consequence relations \eqref{eq:sceltaB}, \eqref{eq:sceltaC} and the key condition \eqref{eq:sceltagamma}.
 \end{remark} 

 We can see also that 
\begin{equation}\label{eq:reminderrad}
\left\|\st\mathcal{R}_{j}^{out}\dt^{(0)}\right\|_{C^{4,\alpha}\st \overline{B_{2}}\setminus B_{1} \dt},\left\|\st\mathcal{R}_{j}^{out}\dt^{(0)}\right\|_{C^{4,\alpha}\st \overline{B_{1}}\setminus B_{\frac{1}{2}} \dt}=o\st \varepsilon^{4m+2}\rep^{-6m+4-\delta}  \dt
\end{equation}
and 
\begin{equation}\label{eq:remindernonrad}
\left\|\st\mathcal{R}_{j}^{out}\dt^{(\dagger)}\right\|_{C^{4,\alpha}\st \overline{B_{2}}\setminus B_{1} \dt},\left\|\st\mathcal{R}_{j}^{out}\dt^{(\dagger)}\right\|_{C^{4,\alpha}\st \overline{B_{1}}\setminus B_{\frac{1}{2}} \dt}=o\st \varepsilon^{2m+4}\rep^{2-4m-\delta} \dt\,,
\end{equation}
therefore  the  biharmonic extensions  dominate all remainder terms in $\varepsilon$-growth indeed 
\begin{equation}
|\hg^{(0)}|+|\kg^{(0)}|=\mathcal{O}\st  \varepsilon^{4m+2}\rep^{-6m+4-\delta} \dt\qquad\textrm{ and }\qquad \left\|\hg^{(\dagger)},\kg^{(\dagger)}\right\|_{\dombd}=\mathcal{O}\st \varepsilon^{2m+4}\rep^{2-4m-\delta} \dt\,.
\end{equation}

\subsection{Cauchy data matching procedure}
Now we want to  find the correct parameters such that at $\Sp^{2m-1	}$ there is a $C^{3}$ matching of potentials $\mathcal{V}^{out}_{j,\bf{0},\bg,\cg,\hg,\kg}$ and $\mathcal{V}^{in}_{j,\tilde{b}_{j},\tilde{h}_{j},\tilde{k}_{j}}$. As proved in \cite{ap1}   there is the $C^{3}$ matching at the boundaries  if and only if the following system is verified    
\begin{equation}
(\Sigma_{j}):\left\{
\begin{array}{rcl}
\mathcal{V}^{out}_{j,\bf{0},\bg,\cg,\hg,\kg}&=&\mathcal{V}^{in}_{j,\tilde{b}_{j},\tilde{h}_{j},\tilde{k}_{j}}\\
&&\\
\partial_{|w|}\sq\mathcal{V}^{out}_{j,\bf{0},\bg,\cg,\hg,\kg}\dq&=&\partial_{|w|}\sq\mathcal{V}^{in}_{j,\tilde{b}_{j},\tilde{h}_{j},\tilde{k}_{j}}\dq\\
&&\\
\Delta\sq\mathcal{V}^{out}_{j,\bf{0},\bg,\cg,\hg,\kg}\dq&=&\Delta\sq\mathcal{V}^{in}_{j,\tilde{b}_{j},\tilde{h}_{j},\tilde{k}_{j}}\dq\\
&&\\
\partial_{|w|}\Delta\sq\mathcal{V}^{out}_{j,\bf{0},\bg,\cg,\hg,\kg}\dq&=&\partial_{|w|}\Delta\sq\mathcal{V}^{in}_{j,\tilde{b}_{j},\tilde{h}_{j},\tilde{k}_{j}}\dq
\end{array}\right.
\end{equation}

After choices \eqref{eq:sceltaB},\eqref{eq:sceltaC}, \eqref{eq:sceltagamma} and some algebraic manipulations, systems $(\Sigma_{j})$ become 


\begin{equation}
(\Sigma_{j}):\left\{
\begin{array}{rcl}
\varepsilon^{2}\tilde{h}_{j}&=&h_{j}- \xi_{j}\\
\varepsilon^{2}\tilde{k}_{j}&=&k_{j}- \Delta\sq\xi_{j}\dq\\
\partial_{|w|}\sq H_{h_{j},k_{j}}^{out}-H_{h_{j},k_	{j}}^{in}\dq&=& \partial_{|w|}\sq \xi_{j}-H_{\xi_{j},\Delta\xi_{j}}^{in}\dq\\
\partial_{|w|}\Delta\sq H_{h_{j},k_{j}}^{out}-H_{h_{j},k_{j}}^{I}\dq&=& \partial_{|w|}\Delta\sq\xi_{j}-H_{\xi_{j},\Delta\xi_{j}}^{in}\dq
\end{array}\right.
\end{equation}
with  $\xi_{j}$ a function depending linearly $\mathcal{R}_{j}^{out}$ and $\mathcal{R}_{j}^{in}$.  Using Theorem \ref{dirneu} we define the operators 

\begin{equation}
\mathcal{S}_{j}\st \varepsilon^{2}\tilde{h}_{j},\varepsilon^{2}\tilde{k}_{j} , h_{j},k_{j} \dt:=\st h_{j}- \xi_{j},  k_{j}- \Delta\xi_{j},\mathcal{Q}\sq \partial_{|w|}\st \xi_{j}-H_{\xi_{j},\Delta\xi_{j}}^{in}\dt,\partial_{|w|}\Delta\st\xi_{j}-H_{\xi_{j},\Delta\xi_{j}}^{in}\dt   \dq \dt
\end{equation}

and then the operator $\mathcal{S}: \dombd^{2}\rightarrow \mathcal{B}^{2}$

\begin{equation}
\mathcal{S}:=\st \mathcal{S}_{1},\ldots, \mathcal{S}_{N} \dt\,.
\end{equation}

Note also that biharmonic extensions, seen as operators 
\begin{equation}
H_{\cdot,\cdot}^{out},H_{\cdot,\cdot}^{in}:C^{4,\alpha}\st \Sp^{2m-1} \dt\times C^{2,\alpha}\st \Sp^{2m-1} \dt\rightarrow C^{4,\alpha}\st \Sp^{2m-1} \dt
\end{equation}
and the operator
\begin{equation}
\mathcal{Q}:C^{3,\alpha}\st \Sp^{2m-1} \dt\times C^{1,\alpha}\st \Sp^{2m-1} \dt\rightarrow C^{4,\alpha}\st \Sp^{2m-1} \dt\times C^{2,\alpha}\st \Sp^{2m-1} \dt
\end{equation}
defined in Proposition \ref{dirneu}, preserve eigenspaces of $\Delta_{\Sp^{2m-1}}$. Thanks to the explicit knowledge of the various  terms composing $\mathcal{R}_{j}^{out}$'s and $\mathcal{R}_{j}^{in}$'s, in particular estimates \eqref{eq:reminderrad} and \eqref{eq:remindernonrad},  we can find  $\kappa>0$ such that 
\begin{equation}
\mathcal{S}: \dombd^{2}\rightarrow \dombd^{2}\,.
\end{equation}
Now the conclusion follows immediately applying a Picard iteration scheme and standard regularity theory.
\end{proof}

\section{Examples}

\noindent In this Section we list few examples where our results can be applied. We have confined ourselves to the case when $M$ is a {\em{toric}} K\"ahler-Einstein orbifold, but there is no doubt that this is far from a comprehensive list.

\begin{example}
Consider $\st\PP^{1}\times\PP^{1},\pi_{1}^{*}\omega_{FS}+\pi_{2}^{*}\omega_{FS}\dt$ and let $\ZZ_{2}$ act in the following way
\begin{equation}\st[x_{0}:x_{1}],[y_{0}:y_{1}]\dt\longrightarrow  \st[x_{0}:-x_{1}],[y_{0}:-y_{1}]\dt\end{equation} 

\noindent It's immediate to check that this action is in $SU(2)$ with four fixed points
\begin{gather}
p_{1}= \st[1:0],[1:0]\dt \\
p_{2}= \st[1:0],[0:1]\dt \\
p_{3}= \st[0:1],[1:0]\dt \\
p_{4}= \st[0:1],[0:1]\dt 
\end{gather}

\noindent 
The quotient space $X_{2}:=\PP^{1}\times\PP^{1}/\ZZ_{2}$ is a K\"ahler-Einstein, Fano orbifold and thanks to the embedding into $\PP^{4}$
\begin{equation}
\st [x_{0}:x_{1}],[y_{0}:y_{1}] \dt\mapsto [ x_{0}^{2}y_{0}^{2}:x_{0}^{2}y_{1}^{2}:x_{1}^{2}y_{0}^{2}:x_{1}^{2}y_{1}^{2}: x_{0}x_{1}y_{0}y_{1} ]
\end{equation}
it is isomorphic to the intersection of singular quadrics 
\begin{equation}
\sg z_{0}z_{3}-z_{4}^{2}=0 \dg\cap \sg z_{1}z_{2}-z_{4}^{2}=0 \dg
\end{equation}
that by \cite{AGP} is a limit of \K -Einstein surfaces, namely the intersection of two smooth quadrics. Since it is K\"ahler-Einstein, conditions for applying our construction become exactly the conditions of  \cite{ap2}, so we have to verify that the matrix
\begin{equation}
\Theta\st \textrm{\bf{1}} , s_{\omega} \textrm{\bf{1}}\dt  =\st \frac{ s_{\omega}}{2} \varphi_{j}\st p_{i} \dt \dt_{\substack{ 1\leq i \leq 2 \\1\leq j\leq 4}}\,.
\end{equation}
has full rank and there exist a positive element in $\ker \Theta\st \textrm{\bf{1}} , s_{\omega} \textrm{\bf{1}}\dt$. It is immediate to see that we have 
\begin{equation}H^{0}\st X_{2}, T^{(1,0)}X_{2} \dt=H^{0}\st \PP^{1}/\ZZ_{2}, T^{(1,0)}\st  \PP^{1}/\ZZ_{2} \dt \dt\oplus H^{0}\st \PP^{1}/\ZZ_{2}, T^{(1,0)}\st  \PP^{1}/\ZZ_{2} \dt \dt\,.\end{equation}
Moreover \begin{equation}H^{0}\st \PP^{1}/\ZZ_{2}, T^{(1,0)}\st  \PP^{1}/\ZZ_{2} \dt \dt\end{equation} is generated by holomorphic vector fields on $\PP^{1}$ that vanish on points 
$[0:1],[1:0]$ so 
\begin{equation}\dim_{\CC} H^{0}\st \PP^{1}/\ZZ_{2}, T^{(1,0)}\st  \PP^{1}/\ZZ_{2} \dt \dt=1\end{equation}
and an explicit generator is the vector field
\begin{equation}V=z^{1}\partial_{1}\,.\end{equation} 
We can compute explicitly its potential $\varphi_{V}$ with respect to $\omega_{FS}$ that is
\begin{equation}\varphi_{V}\st [z_{0}:z_{1}] \dt=- \frac{|z_{0}z_{1}|}{|z_{0}|^{2}+|z_{1}|^{2}}+\frac{1}{2}\end{equation}
and it is easy to see that it is a well defined function and 
\begin{equation}\int_{\PP^1}\varphi_{V}\omega_{FS}=0\,.\end{equation}
Summing up everything, we have that the matrix $\Theta\st \textrm{\bf{1}} , s_{\omega} \textrm{\bf{1}}\dt$ for $X_{2}$ is a $2\times 4$ matrix and can be written explicitly
\begin{equation}\Theta\st \textrm{\bf{1}} , s_{\omega} \textrm{\bf{1}}\dt=\frac{s_{\omega}}{2}\begin{pmatrix}
-1&-1 & 1&1\\
-1&1&-1&1
\end{pmatrix}\end{equation}
that has rank $2$ and every vector of type $\st a,b,b,a \dt$ for $a,b>0$ lies in  $\ker\Theta\st \textrm{\bf{1}} , s_{\omega} \textrm{\bf{1}}\dt$. 

\end{example}

\begin{example}

Consider $\st\PP^{2},\omega_{FS}\dt$ and let $\ZZ_{3}$ act in the following way
\begin{equation}[z_{0}:z_{1}:z_{2}]\longrightarrow  [x_{0}:\zeta_{3}x_{1}:\zeta_{3}^{2}x_{2}]\qquad \zeta_{3}\neq 1, \zeta_{3}^{3}=1\end{equation} 

\noindent It's immediate to check that this action is in $SU(2)$ with three fixed points
\begin{gather}
p_{1}= [1:0:0]\\
p_{2}= [0:1:0] \\
p_{3}= [0:0:1] 
\end{gather}

noindent The quotient space $X_{3}:=\PP^{2}/\ZZ_{3}$ is a K\"ahler-Einstein, Fano orbifold and it is isomorphic, via the embedding
\begin{equation}
[x_{0}:x_{1} :x_{2}]\mapsto [x_{0}^{3}:x_{1}^{3}: x_{2}^{3}: x_{0}x_{1}x_{2}]\,,
\end{equation}
to the singular cubic surface in $\PP^{3}$
\begin{equation}
\sg z_{0}z_{1}z_{2}-z_{3}^{3}=0\dg\,.
\end{equation}
that by \cite{Tian} we know to be a point of the boundary of the moduli space of Fano K\"ahler-Einstein surfaces, namely smooth cubic hypersurfaces. 
Again, conditions for applying our construction become exactly the conditions of Theorem \cite[Theorem  ]{ap2}, so we have to verify that the matrix
\begin{equation}
\Theta\st \textrm{\bf{1}} , s_{\omega} \textrm{\bf{1}}\dt  =\st \frac{ 2s_{\omega}}{3} \varphi_{j}\st p_{i} \dt \dt_{\substack{ 1\leq i \leq 2 \\1\leq j\leq 3}}\,.
\end{equation}
has full rank and there exist a positive element in $\ker \Theta\st \textrm{\bf{1}} , s_{\omega} \textrm{\bf{1}}\dt $. It is immediate to see that we have 
\begin{equation}
\dim_{\CC} H^{0}\st X_{3}, T^{(1,0)}X_{3} \dt=2
\end{equation}
because $H^{0}\st X_{3}, T^{(1,0)}X_{3} \dt$ it is generated by holomorphic vector fields on $\PP^{2}$ vanishing at points $p_{1},p_{2},p_{3}$.
Explicit generators are the vector fields
\begin{gather}
V_{1}=z^{1}\partial_{1}+z^{2}\partial_{2}\\
V_{2}=z^{0}\partial_{0}+z^{1}\partial_{1}
\end{gather} 
We can compute explicitly their potentials $\phi_{V_{1}},\phi_{V_{2}}$ with respect to $\omega_{FS}$ that are
\begin{equation}\phi_{V_{1}}\st [z_{0}:z_{1}:z_{2}] \dt=  -\frac{|z^{0}|^{2}}{|z^{0}|^{2}+|z^1|^{2}+|z^{2}|^{2}}+\frac{1}{3}\end{equation}
\begin{equation}\phi_{V_{2}}\st [z_{0}:z_{1}:z_{2}] \dt=-\frac{|z^{2}|^{2}}{|z^{0}|^{2}+|z^1|^{2}+|z^{2}|^{2}}+\frac{1}{3}\end{equation}
and it is easy to see that are well defined functions and 
\begin{equation}\int_{\PP^2}\phi_{V_{1}}\frac{\omega_{FS}^{2}}{2}=\int_{\PP^2}\phi_{V_{1}}\frac{\omega_{FS}^{2}}{2}=0\,\end{equation}
One can check that

\begin{align}
\varphi_{1}=&-3\st \phi_{1}+2\phi_{2}  \dt\\
\varphi_{2}=&-3\st 2\phi_{1}+\phi_{2}  \dt
\end{align}
is a basis of the space of potentials of holomorphic vector fields vanishing somewhere on $X_{3}$. Summing up everything, we have that the matrix $\Theta\st \textrm{\bf{1}} , s_{\omega} \textrm{\bf{1}}\dt$ for $X_{3}$ is a $2\times 3$ matrix and can be written explicitly
\begin{equation}\Theta\st \textrm{\bf{1}} , s_{\omega} \textrm{\bf{1}}\dt=\frac{2s_{\omega}}{3}\begin{pmatrix}
1&-1 & 0\\
0&-1&1
\end{pmatrix}\end{equation}
that has rank $2$ and every vector of type $\st a,a,a \dt$ for $a>0$ lies in  $\ker \Theta\st \textrm{\bf{1}} , s_{\omega} \textrm{\bf{1}}\dt$. 

\end{example}

\subsection{Equivariant version and partial desingularizations}

\noindent If the orbifold is acted on by a compact group it is immediate to observe that our proof goes through taking at every step of the proof equivariant spaces and averaging on the group with its Haar measure. We can then use the following

\begin{teo}\label{maintheoremequiv}
Let $\st M,\omega, g\dt$ be a compact Kcsc orbifold with isolated singularities and let $G$ be a compact subgroup of holomorphic isometries
such that $\omega $ is invariant under the action of $G$. Let $\p\:=\sg p_1,\ldots,p_{N}\dg\subseteq M$ the set of points  with neighborhoods biholomorphic to  a ball of  $\CC^m/\Gamma_{j}$ with $\Gamma_{j}$ nontrivial subgroup of $SU(m)$ such that $\CC^{m}/\Gamma_{j}$ admits an ALE  Kahler Ricci-flat resolution $\st X_{\Gamma_{j}},h,\eta_{j} \dt$
and
\begin{align}
\ker\st   \Lg  \dt^{G} :=&\,\ker\st   \Lg  \dt\cap \left\{f\in C^{2}\st M \dt| \gamma^{*}\dd f=\dd f \quad \forall\gamma \in G \right\}\\
=&\,span_{\RR}\left\{1,\varphi_1,\ldots,\varphi_d \right\}\,.
\end{align}

Suppose moreover that there exist $ \bg \in (\mathbb{R}^{+})^{N}$ and $\cg\in\RR^{N}$ such that
			
		\begin{displaymath}\label{eq:matricebalequiv}
		\left\{\begin{array}{lcl}
		\sum_{j=1}^{N}b_{j}\Delta_{\omega}\varphi_{i}\st p_{j} \dt+c_{j}\varphi_{i}\st p_{j} \dt=0 && i=1,\ldots, d\\
		&&\\
		\st b_{j}\Delta_{\omega}\varphi_{i}\st p_{j} \dt+c_{j}\varphi_{i}\st p_{j} \dt \dt_{\substack{1\leq i\leq d\\1\leq j\leq N}}&& \textrm{has full rank}
		\end{array}\right.
		\end{displaymath}
	
	If 
	\begin{equation}\label{eq:tuningequiv}
	c_{j}=s_{\omega}b_{j}\,,
	\end{equation}

then there is $\bar{\varepsilon}$ such that for every $\varepsilon \in (0,\bar{\varepsilon})$ the orbifold
\[
\tilde{M} : = M \sqcup _{{p_{1}, \varepsilon}} X_{\Gamma_1} \sqcup_{{p_{2},\varepsilon}} \dots
\sqcup _{{p_N, \varepsilon}} X_{\Gamma_N}
\]

\noindent has a Kcsc metric in the class 
\begin{equation}
\pi^{*}\sq\omega \dq+ \sum_{j=1}^{N}\varepsilon^{2m}\tilde{b}_{j}^{2m}\sq \tilde{\eta}_{j} \dq\qquad\textrm{ with }\qquad \mathfrak{i}_{j}^{*}\sq \tilde{\eta}_{j} \dq=[\eta_{j}]
\end{equation}
\noindent where $\pi$ is the canonical surjection of $\tilde{M}$ onto $M$ and $\mathfrak{i}_{j}$ the natural embedding of $X_{\Gamma_{j},\Rep}$ into $\tilde{M}$. Moreover 
\begin{equation}
\left|\tilde{b}_{j}^{2m} - \frac{|\Gamma_{j}|b_{j}}{2\st m-1 \dt}\right| \leq \mathsf{C} \varepsilon^{\gamma}\qquad\textrm{ for some }\qquad \gamma>0\,.
\end{equation}\end{teo}

\noindent If the K\"ahler orbifold $(M,\omega )$ is  a toric variety,  $\omega $ is K\"ahler-Einstein and $G=\st S^{1} \dt^{m}$ then  $\omega $ is $G$-invariant (by Matsushima-Lichnerowicz) and 
\begin{equation}
\ker\st \Lg \dt^{G}=\left\{ 1,\varphi_{1},\ldots, \varphi_{m} \right\}\,.
\end{equation}

By definition, the functions  $\varphi_{j}$ are such that  
\begin{equation}
\partial^{\sharp}\varphi_{j}\st p \dt=\left.\frac{d}{dt}\sq \st e^{t\log\st \lambda_{j}^{1} \dt},\ldots, e^{t\log\st \lambda_{j}^{m} \dt} \dt \cdot p  \dq\right|_{t=0}\qquad (\lambda_{j}^{1},\ldots,\lambda_{j}^{m})\in \st\CC^{*}\dt^{m}
\end{equation}
\noindent and can be chosen in such a way that,  having set
\begin{equation}
\mu:M\rightarrow \RR^{m}\qquad \mu\st p\dt := \st \varphi_{1}\st p \dt,\ldots,\varphi_{d}\st p \dt \dt\,,
\end{equation}
the set $\mu\st M \dt$ is a convex polytope that coincides up to transformations in $SL(2,\ZZ)$ with the polytope associated to the pluri-anticanonical polarization of the toric variety $M$. Moreover 
\begin{equation}
\Lg=\Delta_{\omega}^{2}+\frac{s_{\omega}}{m}\Delta_{\omega}
\end{equation}
and 
\begin{equation}
\Delta\varphi_{j}=-\frac{s_{\omega}}{m}\varphi_{j}
\end{equation}
so 
\begin{equation}
\Theta\st \textrm{\bf{1}} , s_{\omega} \textrm{\bf{1}}\dt=\Theta\st \textrm{\bf{0}} , \frac{\st  m-1 \dt s_{\omega}}{m} \textrm{\bf{1}}\dt  =\st \frac{\st  m-1 \dt s_{\omega}}{m} \varphi_{j}\st p_{i} \dt \dt_{\substack{1\leq j\leq d\\1\leq i \leq N}}\,.
\end{equation}
Moreover the set $\mu\st \p \dt$ is a subset of the vertices of $\mu\st M \dt$,  indeed points of $\p$ are critical points for $\varphi_{j}$ since their gradients vanish at these points (indeed the holomorphic vector fields $\partial^{\sharp}\varphi_{j}$ must vanish at these points since they must preserve the isolated singularities). Assumptions of Theorem \ref{maintheoremequiv} are then satisfied if the barycenter of the set $\mu\st \p \dt$ is the origin of $\RR^{m}$.

\begin{example}
Let  $X^{(1)}$ be the toric K\"ahler-Einstein threefold whose  1-dimensional fan  $\Sigma^{(1)}_{1}$ is generated by points
\begin{equation}\Sigma^{(1)}_{1}=\left\{(1,3,-1), (-1,0,-1), (-1,-3,1), (-1,0,0), (1,0,0), (0,0,1), (0,0,-1), (1,0,1)\right\}
\end{equation}
and its $3$-dimensional fan $\Sigma^{(1)}_{3}$  is generated by $12$ cones 
\begin{align}
C_{1}:=& \left<        (-1,  0, -1),(-1, -3,  1),(-1,  0,  0)\right>\\
C_{2}:=& \left<        ( 1,  3, -1),(-1,  0, -1),(-1,  0,  0)\right>\\
C_{3}:=& \left<        (-1, -3,  1),(-1,  0,  0),( 0,  0,  1)\right>\\
C_{4}:=&\left<        ( 1,  3, -1),(-1,  0,  0),( 0,  0,  1)\right>\\
C_{5}:= &\left<        ( 1,  3, -1),(-1,  0, -1),( 0,  0, -1)\right>\\
C_{6}:=&\left<        (-1,  0, -1),(-1, -3,  1),( 0,  0, -1)\right>\\
C_{7}:= & \left<        (-1, -3,  1),( 1,  0,  0),( 0,  0, -1)\right>\\
C_{8}:=& \left<        (1,  3, -1),(1,  0,  0),(0,  0, -1)\right>\\
C_{9}:= &\left<        (1,  3, -1),(0,  0,  1),(1,  0,  1)\right>\\
C_{10}:=& \left<        (-1, -3,  1),( 1,  0,  0),( 1,  0,  1)\right>\\
C_{11}:=& \left<        (1,  3, -1),(1,  0,  0),(1,  0,  1)\right>\\
C_{12}:= & \left<        (-1, -3,  1),( 0,  0,  1),( 1,  0,  1)\right>
\end{align}

\noindent All these cones are singular and $C_{1},C_{4},C_{5},C_{7},C_{11},C_{12}$ are cones relative to affine open subsets of $X^{(1)}$ containing a $SU(3)$ singularity, while the others  are cones relative to affine open subsets of $X^{(1)}$ containing a $U(3)$ singularity. 


\noindent The  3-anticanonical polytope  $P_{-3K_{X^{(1)}}}$ is the convex hull of vertices
\begin{align}
P_{-3K_{X^{(1)}}}:=&\left<(0,-2,-3), (-3,0,0), (-3,1,3), (0,0,3), (3,-2,0),\right.\\
&\left. (0,2,3), (0,0,-3), (-3,2,0), (-3,3,3), (3,0,0), (3,-1,-3), (3,-3,-3)\right>
\end{align}


\noindent With 2-faces 
\begin{align}
F_{1}:=&\left<        ( 0, -2, -3),( 3, -3, -3),(-3,  0,  0),(-3,  1,  3),( 0,  0,  3),( 3, -2,  0)\right>\\
F_{2}:=&\left<        (-3,  1,  3),( 0,  0,  3),( 0,  2,  3),(-3,  3,  3)\right>\\
F_{3}:=&\left<        (0,  0,  3),(3, -2,  0),(0,  2,  3),(3,  0,  0)\right>\\
F_{4}:=&\left<        ( 0, -2, -3),(-3,  0,  0),( 0,  0, -3),(-3,  2,  0)\right>\\
F_{5}:=&\left<        ( 3, -1, -3),( 0,  2,  3),( 0,  0, -3),(-3,  2,  0),(-3,  3,  3),( 3,  0,  0)\right>\\
F_{6}:=&\left<        (-3,  0,  0),(-3,  1,  3),(-3,  2,  0),(-3,  3,  3)\right>\\
F_{7}:=&\left<        (3, -1, -3),(0, -2, -3),(3, -3, -3),(0,  0, -3)\right>\\
F_{8}:=&\left<        (3, -1, -3),(3, -3, -3),(3, -2,  0),(3,  0,  0)\right>
\end{align}


\noindent We have the following correspondences between cones containing a $SU(3)$-singularity and vertices of $P_{-3K_{X^{(1)}}}$
\begin{align}
C_{1}& \longleftrightarrow F_{3}\cap F_{5}\cap F_{8}=\sg (3,0,0) \dg\\
C_{4}& \longleftrightarrow F_{1}\cap F_{7}\cap F_{8}=\sg (3,-3,-3) \dg\\
C_{5}& \longleftrightarrow F_{1}\cap F_{2}\cap F_{3}=\sg (0,0,3) \dg\\
C_{7}& \longleftrightarrow F_{2}\cap F_{5}\cap F_{7}=\sg (-3,3,3) \dg\\
C_{11}& \longleftrightarrow F_{1}\cap F_{4}\cap F_{6}=\sg (-3,0,0) \dg\\
C_{12}& \longleftrightarrow F_{4}\cap F_{5}\cap F_{7}=\sg (0,0,-3) \dg
\end{align}

\noindent Since in complex dimension $3$ every $SU(3)$-singularity admits a K\"ahler crepant resolution it is then immediate to see that all assumptions of Theorem \ref{maintheoremequiv} are satisfied. 

\end{example}

\begin{example}
Let  $X^{(4)}$ be the toric K\"ahler-Einstein threefold whose  1-dimensional fan  $\Sigma^{(3)}_{1}$ is generated by points
\begin{equation}\Sigma^{(4)}_{1}=\left\{(0,3,1), (1,1,2), (1,0,0), (-1,0,0), (-2,-1,-2), (1,-3,-1)\right\}\end{equation}
and its $3$-dimensional fan $\Sigma^{(4)}_{3}$  is generated by $8$ cones

\begin{align}
C_{1}:=&\left<( 0,  3,  1),( 1,  1,  2),(-1,  0,  0)\right>\\
C_{2}:=&\left<(0, 3, 1),(1, 1, 2),(1, 0, 0)\right>\\
C_{3}:=&\left<( 0,  3,  1),(-1,  0,  0),(-2, -1, -2)\right>\\
C_{4}:=&\left<( 0,  3,  1),( 1,  0,  0),(-2, -1, -2)\right>\\
C_{5}:=&\left<( 1,  0,  0),(-2, -1, -2),( 1, -3, -1)\right>\\
C_{6}:=&\left<( 1,  1,  2),(-1,  0,  0),( 1, -3, -1)\right>\\
C_{7}:=&\left<(-1,  0,  0),(-2, -1, -2),( 1, -3, -1)\right>\\
C_{8}:=&\left<(1,  1,  2),(1,  0,  0),(1, -3, -1)\right>
\end{align}

\noindent The cones $C_{1},C_{4},C_{7},C_{8}$ are relative to affine open subsets of $X^{(4)}$ containing a $SU(3)$ singularity and the  other cones are  relative to affine open subsets of $X^{(4)}$ containing a $U(3)$ singularity.

\noindent The  5-anticanonical polytope  $P_{-5K_{X^{(4)}}}$ is the convex hull of vertices
\begin{align}
P_{-5K_{X^{(4)}}}:=&\left< (5,-1,-2), (5,0,-5), (-5,-2,1), (-5,0,0),\right.\\
&\left. (5,5,-5), (-5,-5,10), (-5,-3,9), (5,6,-8) \right>
\end{align}
With 2-faces 
\begin{align}
F_{1}:=&\left<( 5,  0, -5),(-5, -2,  1),(-5,  0,  0),( 5,  6, -8)\right>\\
F_{2}:=&\left<( 5, -1, -2),( 5,  0, -5),(-5, -2,  1),(-5, -5, 10)\right>\\
F_{3}:=&\left<(5, -1, -2),(5,  0, -5),(5,  5, -5),(5,  6, -8)\right>\\
F_{4}:=&\left<( 5, -1, -2),( 5,  5, -5),(-5, -5, 10),(-5, -3,  9)\right>\\
F_{5}:=&\left<(-5, -2,  1),(-5,  0,  0),(-5, -5, 10),(-5, -3,  9)\right>\\
F_{6}:=&\left<(-5,  0,  0),( 5,  5, -5),(-5, -3,  9),( 5,  6, -8)\right>
\end{align}


\noindent We have the following correspondences between cones containing a $SU(3)$-singularity and vertices of $P_{-5K_{X^{(4)}}}$
\begin{align}
C_{1}& \longleftrightarrow F_{1}\cap F_{2}\cap F_{5}=\sg (-5,-2,1) \dg\\
C_{4}& \longleftrightarrow F_{2}\cap F_{3}\cap F_{4}=\sg (5,-1,-2) \dg\\
C_{7}& \longleftrightarrow F_{4}\cap F_{5}\cap F_{6}=\sg (-5,-3,9) \dg\\
C_{8}& \longleftrightarrow F_{1}\cap F_{3}\cap F_{6}=\sg (5,6,-8) \dg
\end{align}

\noindent Since in complex dimension $3$ every $SU(3)$-singularity admits a K\"ahler crepant resolution it is then immediate to see that all assumptions of Theorem \ref{maintheoremequiv} are satisfied.

\end{example}


\providecommand{\bysame}{\leavevmode\hbox to3em{\hrulefill}\thinspace}
\providecommand{\MR}{\relax\ifhmode\unskip\space\fi MR }
\providecommand{\MRhref}[2]{%
  \href{http://www.ams.org/mathscinet-getitem?mr=#1}{#2}
}
\providecommand{\href}[2]{#2}


\end{document}